\documentclass[a4paper]{amsart}

\usepackage{fullpage}
\usepackage[colorlinks=True, citecolor=blue, linkcolor=blue]{hyperref}
\usepackage{amssymb}
\usepackage{tikz-cd}
\usepackage{tkz-graph}
\usepackage{float}

\usetikzlibrary{shapes.geometric}
\usepackage{algorithm}
\usepackage{algorithmic}
\usepackage[nameinlink]{cleveref}
\usepackage{booktabs}
\usepackage{longtable}
\usepackage{placeins}
\usepackage{colortbl}
\usepackage{color}
\usepackage[shortlabels,inline]{enumitem}
\usepackage[normalem]{ulem}

\usepackage{comment}
\definecolor{lightgray}{gray}{0.9}

\theoremstyle{definition}
\newtheorem{definition}{Definition}[section]
\newtheorem{example}[definition]{Example}
\newtheorem{remark}[definition]{Remark}

\theoremstyle{plain}
\newtheorem{theorem}[definition]{Theorem}
\newtheorem{lemma}[definition]{Lemma}
\newtheorem{proposition}[definition]{Proposition}
\newtheorem{corollary}[definition]{Corollary}

\DeclareMathOperator{\Nef}{Nef}

\DeclareMathOperator{\rk}{rk}

\DeclareMathOperator{\id}{id}

\DeclareMathOperator{\Hom}{Hom}

\DeclareMathOperator{\Aut}{Aut}

\DeclareMathOperator{\aut}{Aut^*}

\DeclareMathOperator{\im}{im}

\DeclareMathOperator{\stab}{stab}

\newcommand{\QQ}{\mathbb{Q}}
\newcommand{\FF}{\mathbb{F}}
\newcommand{\RR}{\mathbb{R}}

\newcommand{\ZZ}{\mathbb{Z}}
\newcommand{\PP}{\mathbb{P}}

\renewcommand{\P}{\mathcal{P}}

\newcommand{\R}{\mathcal{R}}

\newcommand{\RN}[1]{\textup{\uppercase\expandafter{\romannumeral#1}}}

\makeatletter
\newcommand*{\defeq}{\mathrel{\rlap{
                     \raisebox{0.3ex}{$\m@th\cdot$}}
                     \raisebox{-0.3ex}{$\m@th\cdot$}}
                     =}
\makeatother

\newcommand{\Gbar}{\overline{G}}
\newcommand{\Pbar}{\overline{\P}^\mathbb{Q}}

\newcommand{\DeltabarY}{\overline{\Delta}(Y)}

\makeatletter
\def\blfootnote{\xdef\@thefnmark{}\@footnotetext}

\title[]{Orbits of smooth rational curves on Enriques surfaces}

\author{Simon Brandhorst}

\address{Simon Brandhorst,
Fakult\"at f\"ur Mathematik und Informatik, Universit\"at des Saarlandes, Campus E2.4, 66123 Saarbr\"ucken, Germany}
\email{brandhorst@math.uni-sb.de}
\author{Víctor González-Alonso}
\address{Víctor González-Alonso,
Institut für Algebraische Geometrie,
Leibniz Universit\"at Hannover
Welfengarten 1,
30167 Hannover, Germany}
\email{gonzalez@math.uni-hannover.de}
\thanks{Gefördert durch die Deutsche Forschungsgemeinschaft (DFG) – Projektnummer 286237555 – TRR 195.
Funded by the Deutsche Forschungsgemeinschaft (DFG, German Research Foundation) – Project-ID 286237555 – TRR 195.}
\begin{document}
\begin{abstract}
We give a closed formula for the number of orbits of smooth rational curves under the automorphism group of an Enriques surface in terms of its Nikulin root invariant and its Vinberg group. 
\end{abstract}
\maketitle
\addtocontents{toc}{\protect\setcounter{tocdepth}{1}}
\section{Introduction}
An Enriques surface over a field $k$ is a smooth proper algebraic surface $Y/k$ with second (étale) Betti number $b_2(Y)=10$ and numerically trivial canonical bundle. Besides abelian, K3 and bielliptic surfaces, Enriques surfaces appear as one of the four classes of minimal surfaces of Kodira dimension $0$.
In this work we assume that $k$ is an algebraically closed field of characteristic not two. 

The number of smooth rational curves on an Enriques surface $Y$ can be zero, finite or infinite. 
They are in bijection with the facets of the nef cone of $Y$.
The Morrison-Kawamata cone conjecture predicts that the automorphism group $\Aut(Y)$ has a rational polyhedral fundamental domain on the nef cone. It is known for Enriques surfaces in characteristic not two \cite{namikawa:periods_of_enriques,wang}. In particular, the number of $\Aut(Y)$-orbits of facets of the nef cone is finite and coincides with the number of $\Aut(Y)$-orbits of smooth rational curves on $Y$.

Our main result is a formula for the number of orbits. To state it, we need to introduce two ingredients which control the automorphism group and the nef cone: the \emph{Vinberg group} and the \emph{Nikulin root invariant}.
Let $S_Y$ denote the numerical lattice of $Y$, $\R(Y)\subseteq S_Y$ the set of classes of smooth rational curves and $W(Y)=W(\R(Y))$ the Weyl group. Let $\Aut^*(Y)$ denote the image of the representation $\Aut(Y) \to O(S_Y)$.

The \emph{Vinberg group} $\overline{G}_Y$ of $Y$ is the image of the natural map
\[G_Y:=W(Y)\rtimes \Aut^*(Y) \to O(S_Y\otimes \FF_2).\]
Note that the class $r \in S_Y$ of a smooth rational curve has $r^2=-2$. 
Set 
\[\Delta^+(Y)=\{r \in S_Y \mid r^2=-2, r \mbox{ is effective}\},\]
and let $\DeltabarY$ denote the image of $\Delta(Y)$ under the reduction mod $2$ map $S_Y \to S_Y\otimes \FF_2$. 
We view $\DeltabarY$ as an undirected graph by connecting two elements $\overline{r_1},\overline{r_2}$ by an edge if $r_1.r_2 \equiv 1 \mod 2$.
As it turns out, the connected components $\sigma$ of this graph are (graphs of) simply laced irreducible root systems, i.e. ADE-root systems. Let $R_\sigma$ denote the corresponding root lattice (see \Cref{def:R_sigma}). The \emph{Nikulin root invariant} is the pair
\[\left(R:=\bigoplus_\sigma R_\sigma, \;\ker (R\otimes \FF_2 \to S_Y\otimes \FF_2)\right).\]
The \emph{type} of a single connected component $\sigma$ is given by $R_\sigma$ and the kernel $K_\sigma$ of $R_\sigma \otimes \FF_2 \to S_Y\otimes \FF_2$. For $R \subseteq S_Y$ let $\overline{\Delta}(R)$ denote the images of the roots $\Delta(R)$ of $R$ via $S_Y \to S_Y \otimes \FF_2$.

The following proposition gives a first interpretation of Nikulin's root invariant: its connected components give ADE-configurations of rational curves and therefore faces of the nef cone.
\begin{proposition}
    Let $Y/k$ be an Enriques surface and $\sigma$ a connected component of $\DeltabarY$ of type $(R_\sigma,K_\sigma)$. Then there exists a configuration of smooth rational curves $b_1,\dots, b_n$ on $Y$ such that $R_\sigma \cong\langle b_1, \dots, b_n \rangle=:R$, $\sigma = \overline{\Delta}(R)$ and $K_\sigma = \ker R \otimes \FF_2 \to S_Y\otimes \FF_2$.
\end{proposition}

Since we work in characteristic not $2$, $Y$ is the quotient of a K3 surface $X$ by a fixed point free involution $\epsilon$. Then the finite set $\DeltabarY$ can be calculated easily from the action of $\epsilon$ on the Néron-Severi lattice $S_X$ of $X$, see \Cref{sec:root_invariants}. 
The same is true for $\Gbar_Y$ under the assumption that every automorphism on $Y$ is semi-symplectic, i.e. acts trivially on $H^0(Y,2K_Y)$, see e.g. \cite[Theorem 3.9, Remarks 3.12, 3.13]{brandhorst-gonzalez:527}. 
\begin{theorem}\label{thm:main}
Let $Y/k$ be an Enriques surface, and $a_n,d_n,e_n,a_7',d_8'$ be the number of $\Gbar_Y$-orbits on the set of connected components of $\DeltabarY$ of type $(A_n,0)$ ,$(D_n, 0)$, $(E_n,0)$, $(A_7, \ZZ/2\ZZ)$ and $(D_8, \ZZ/2\ZZ)$. 

Then the number of $\Aut(Y)$-orbits of (-2)-curves is given by
\[\sum_{i=1}^9 a_i + \sum_{i=1}^8 d_i + 2 d_9 + e_6+2e_7+4e_8+a_7'+2d_8'.\]
\end{theorem}

\begin{remark}
Most of the connected components of $\overline{\Delta}(Y)$ contribute with one orbit. The exceptions $e_7, d_8',e_8,d_9$ are explained (and proven) by geometric phenomena:

\begin{itemize}
\item A component of type $(D_8, \ZZ/2\ZZ)$ implies the existence of a numerically trivial automorphism $g$ in the center of $\Aut(Y)$ whose fixed locus consists of $4$ isolated points, $4$ smooth rational curves and possibly an elliptic curve \cite[Table 8.8,Theorem 8.2.21, Remark 8.2.22]{enriquesII}. Therefore, $\Aut(Y)$ permutes the $4$ curves and any other $(-2)$-curve on $Y$ (for instance in a $D_8$ configuration giving sigma) must lie in a different orbit. 

\item A component of type $(E_8, 0)$ implies that $Y$ admits a cohomologically trivial involution and is of zero entropy, i.e. has a unique elliptic fibration of positive rank \cite{martin2024enriquessurfaceszeroentropy}. Then the fibration and the automorphism give $\Aut(Y)$-invariant configurations of rational curves, which allows us to separate the orbits. The weight $4$ may be also be explained as follows: the cohomologically trivial involution fixes pointwise $4$ smooth rational curves (and $4$ isolated fixed points). 
There exists an Enriques surface with root invariant $(E_8,0)$ which admits a numerically trivial automorphism of order $4$ whose square is the cohomologically trivial involution \cite[Theorem 8.2.23]{enriquesII}.
Two of the fixed curves are fixed by the order $4$ automorphism and two are not. The orbits are separated by their distance to these two kinds of fixed curves.

\item A component of type $(D_9, 0)$ implies that $Y$ has finite automorphism group of type II and only $12$ rational curves \cite[Main Theorem]{kondo:enriques_finite_automorphism}. Their dual graph has vertices of valency $2$ or $3$ which therefore lie in different orbits.

\item Finally, the multiplicity $2$ for $e_7$ comes from the presence of a numerically trivial automorphism on a deformation $Y'$ of $Y$ with root invariant ($E_7+A_1,\ZZ/2\ZZ$) that preserves the $(E_7, 0)$ component. 
\end{itemize}
\end{remark}

\begin{corollary}
Let $Y$ be an Enriques surface in characteristic $p\neq 2$ and $\rho$ the Picard number of its K3 cover. Then the number of $\Aut(Y)$-orbits of $(-2)$-curves on $Y$ is at most $\rho - 10$.
\end{corollary}
\begin{proof}
Each connected component $\sigma$ of $\DeltabarY$ corresponds to a root lattice $R_\sigma$ and their direct sum embeds into the anti-invariant Néron-Severi lattice $S_{X-}$ of the K3 cover $X$ (cf. \Cref{sec:root_invariants}), which is of rank $\rho-10$.
\end{proof}

Precursors to the main theorem include the work of Cossec-Dolgachev \cite{cossec-dolgachev:automorphisms_nodal}, where they prove, using Coxeter groups, that for a generic $1$-nodal Enriques surface (i.e. the ones with root invariant $(A_1,0)$) every rational curve is in the same $\Aut(Y)$-orbit.
Another result is the computer aided computation with Borcherds' method of the $\Aut(Y)$-orbits of $(-2)$-curves for $184$ examples of Enriques surfaces (with prescribed ADE singularities) by the first author and Shimada in \cite{brandhorst_shimada:tautaubar}. This wealth of examples allowed us to conjecture \Cref{thm:main} and then work towards a proof. 

\subsection*{Outline}
The idea for the proof is the following:
The orbits of $\Aut(Y)$ on the facets of the nef cone are not easily controlled in terms of $\DeltabarY$ and $\Gbar_Y$ because the surjection $\R(Y)/\Aut(Y) \to \Delta(Y)/G_Y$ may fail to be injective (\Cref{prop:rational_curves} and its fibers are hard to control explicitly \Cref{lem:phi-fiber}). 
However, the orbits on the minimal (nonisotropic) faces are more accessible: They coincide with the $G_Y$-orbits on maximal ADE configurations in $\Delta(Y)$ (\Cref{lem:transitive_on_induced_chambers}). 
Since $G_Y$ always contains the $2$-congruence subgroup $O^+(S_Y)(2)$ of $O^+(S_Y)$, we classify ADE sublattices of the $E_{10}$-lattice up to the action of $O(E_{10})(2)$ in \Cref{sec:orbits-by-2congruence}. Then in \Cref{sec:maximal-ADE} we determine the maximal ADE configurations in $\R(Y)$ in terms of the root-invariant of $Y$. 
Now, the strategy is as follows: Find a configuration $B$ of $(-2)$-curves on $Y$ such that it contains a representative of each $\Aut(Y)$-orbit of maximal $ADE$-configurations. In the final step one has to determine which elements of $B$ lie in the same $\Aut(Y)$-orbit.
This strategy is carried out in \Cref{sec:orbits-of-rational-curves} leading to the analysis of $8$ cases depending on the root invariant.

\section*{Acknowledgments}
We thank Gebhard Martin, Giacomo Mezeedimi, Ichiro Shimada and Davide Veniani for discussions.
This work is supported by SFB 195 No. 286237555 of DFG.

\section{Preliminaries and Notation.}
Given a group $G$ acting on a set $X$ from the right we use exponential notation $x^g$ for the action of $g\in G$ on $x \in X$. 
For $A \subseteq X$ we write $G_A$ for its pointwise stabilizer. In this notation $G_{\{A\}}$ is the pointwise stabilizer of the 1-element set $\{A\}$ with $G$ acting on the powerset of $X$, i.e. the setwise stabilizer of $A$. Likewise we define $G_{\{A\},\{B\}}:=G_{\{A,B\}}=G_{\{A\}}\cap G_{\{B\}}$, and $G_{x,\{A\}}:=G_x \cap G_{\{A\}}$.

\subsection{Root lattices}
Here we recall well-known facts about (definite) irreducible root lattices, see e.g. \cite{ebeling:lattices-and-codes}. They come in $3$ families $A_n,n\geq 1$, $D_n,n\geq 4$ and $E_6,E_7,E_8$. 
Every root lattice is a direct sum of irreducible root lattices.
Given a root lattice $R$, we write $\tau(R)$ for its isomorphism class. 
We write it as a sum of $ADE$-symbols, e.g. $\tau(R) = A_1 + E_6$.

Let $r_1,\dots, r_n$ be a fundamental root system of $R$. Their Coxeter-Dynkin diagram is an $ADE$ graph. 
Let $r_0$ be the highest root of the corresponding positive root system. Then the Coxeter-Dynkin diagram of $r_0,\dots, r_n$
is an extended $ADE$-Dynkin diagram. These correspond to root lattices with a one dimensional radical. 
The coordinates of a primitive generator $f$ of the radical (necessarily an isotropic vector) are given in \Cref{figure:Antilde,figure:Dntilde,figure:E6tilde,figure:E7tilde,figure:E8tilde} next to each node. We call them multiplicities. The multiplicities of the original vertices $r_1,\dots, r_n$ are the coordinates of the highest root $r_0$.

Let $r_1^\vee,\dots,r_n^\vee \in R^\vee$ be the dual basis of $r_1,\dots, r_n$. The non-zero elements of the discriminant group $D_R:=R^\vee/R$
are in bijection with the fundamental roots of multiplicity one via $r_i\mapsto r_i^\vee+R$. For example the discriminant group of $E_6$ contains $2$ non-zero elements and similarly the $E_6$ diagram has two roots of multiplicity one. 

The discriminant group of $A_n$ is $\ZZ/(n+1)\ZZ$.
That of $D_n$ is $\ZZ/4\ZZ$ if $n$ is odd and $\ZZ/2\ZZ \times \ZZ/2\ZZ$ if $n$ is even.
The discriminant group of $E_6$ is $\ZZ/3\ZZ$ and that of $E_7$ is $\ZZ/2\ZZ$. The one of $E_8$ is trivial.

\begin{figure}
\begin{tikzpicture}
\tikzset{VertexStyle/.style= {fill=black, inner sep=1.5pt, shape=circle}}
\Vertex[NoLabel,x=0.62,y=0.78]{1a}
\Vertex[NoLabel,x=-0.22,y=0.97]{2a}
\Vertex[NoLabel,x=-0.9,y=0.433]{3a}
\Vertex[NoLabel,x=-0.9,y=-0.433]{4a}
\Vertex[NoLabel,x=-0.22,y=-0.97]{5a}
\Vertex[NoLabel,x=0.62,y=-0.78]{6a}
\Vertex[NoLabel,x=1,y=0]{7a}
\Edges(1a,2a,3a,4a,5a,6a,7a)
\tikzset{VertexStyle/.style= {inner sep=1.5pt, shape=circle}}
\Vertex[x=0.81,y=1.02]{1}
\Vertex[x=-0.29,y=1.28]{1}
\Vertex[x=-1.17,y=0.564]{1}
\Vertex[x=-1.17,y=-0.564]{1}
\Vertex[x=-0.29,y=-1.27]{1}
\Vertex[x=0.81,y=-1.01]{1}
\Vertex[x=1.3,y=0]{1}
\begin{scope}   [dashed]
     \draw (1,0) -- (0.62,0.78)  ; 
\end{scope}
\end{tikzpicture}
\caption{The $\widetilde{A}_n$ diagram.}\label{figure:Antilde}
\end{figure}
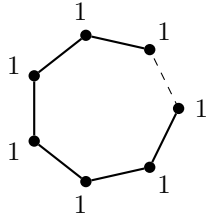
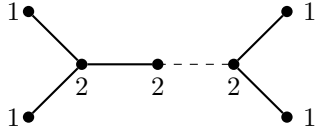
\begin{figure}
\begin{tikzpicture}
\tikzset{VertexStyle/.style= {fill=black, inner sep=1.5pt, shape=circle}}
\Vertex[NoLabel,x=-0.7,y=0.7]{0a}
\Vertex[NoLabel,x=-0.7,y=-0.7]{0b}
\Vertex[NoLabel,x=0,y=0]{1a}
\Vertex[NoLabel,x=1,y=0]{2a}
\Vertex[NoLabel,x=2,y=0]{3a}
\Vertex[NoLabel,x=2.7,y=0.7]{4a}
\Vertex[NoLabel,x=2.7,y=-0.7]{4b}
\Edges(0a,1a,2a)
\Edges(0b,1a)
\Edges(4a,3a)
\Edges(4b,3a)
\tikzset{VertexStyle/.style= {inner sep=1.5pt, shape=circle}}
\Vertex[x=-0.9,y=0.7]{1}
\Vertex[x=-0.9,y=-0.7]{1}
\Vertex[x=0,y=-0.3]{2}
\Vertex[x=1,y=-0.3]{2}
\Vertex[x=2,y=-0.3]{2}
\Vertex[x=3,y=0.7]{1}
\Vertex[x=3,y=-0.7]{1}
\begin{scope}   [dashed]   now dashed is for the lines inside the scope
     \draw (1,0) -- (2,0)  ; 
\end{scope}
\end{tikzpicture}
\caption{The $\widetilde{D}_n$ diagram.}\label{figure:Dntilde}
\end{figure}

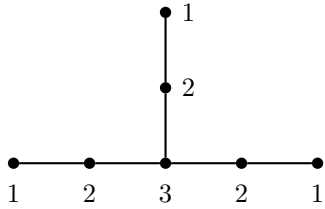
\begin{figure}
\begin{tikzpicture}
\tikzset{VertexStyle/.style= {fill=black, inner sep=1.5pt, shape=circle}}
\Vertex[NoLabel,x=3,y=0]{4a}
\Vertex[NoLabel,x=4,y=0]{5a}
\Vertex[NoLabel,x=5,y=0]{6a}
\Vertex[NoLabel,x=6,y=0]{7a}
\Vertex[NoLabel,x=7,y=0]{8a}
\Vertex[NoLabel,x=5,y=2]{9a}
\Vertex[NoLabel,x=5,y=1]{10a}
\Edges(4a,5a,6a,7a,8a)
\Edges(6a,10a,9a)
\tikzset{VertexStyle/.style= {inner sep=1.5pt, shape=circle}}
\Vertex[x=3,y=-0.4]{1}
\Vertex[x=4,y=-0.4]{2}
\Vertex[x=5,y=-0.4]{3}
\Vertex[x=6,y=-0.4]{2}
\Vertex[x=7,y=-0.4]{1}
\Vertex[x=5.3,y=2]{1}
\Vertex[x=5.3,y=1]{2}
\end{tikzpicture}
\caption{The $\widetilde{E}_{6}$ diagram.}\label{figure:E6tilde}
\end{figure}

\begin{figure}
\begin{tikzpicture}
\tikzset{VertexStyle/.style= {fill=black, inner sep=1.5pt, shape=circle}}
\Vertex[NoLabel,x=2,y=0]{3a}
\Vertex[NoLabel,x=3,y=0]{4a}
\Vertex[NoLabel,x=4,y=0]{5a}
\Vertex[NoLabel,x=5,y=0]{6a}
\Vertex[NoLabel,x=6,y=0]{7a}
\Vertex[NoLabel,x=7,y=0]{8a}
\Vertex[NoLabel,x=8,y=0]{9a}
\Vertex[NoLabel,x=5,y=1]{10a}
\Edges(3a,4a,5a,6a,7a,8a,9a)
\Edges(6a,10a)
\tikzset{VertexStyle/.style= {inner sep=1.5pt, shape=circle}}
\Vertex[x=2,y=-0.4]{1}
\Vertex[x=3,y=-0.4]{2}
\Vertex[x=4,y=-0.4]{3}
\Vertex[x=5,y=-0.4]{4}
\Vertex[x=6,y=-0.4]{3}
\Vertex[x=7,y=-0.4]{2}
\Vertex[x=8,y=-0.4]{1}
\Vertex[x=5,y=1.4]{2}
\end{tikzpicture}
\caption{The $\widetilde{E}_{7}$ diagram.}\label{figure:E7tilde}
\end{figure}
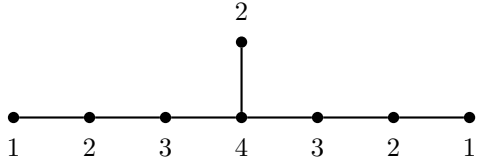

\begin{figure}
\begin{tikzpicture}
\tikzset{VertexStyle/.style= {fill=black, inner sep=1.5pt, shape=circle}}
\Vertex[NoLabel,x=1,y=0]{2a}
\Vertex[NoLabel,x=2,y=0]{3a}
\Vertex[NoLabel,x=3,y=0]{4a}
\Vertex[NoLabel,x=4,y=0]{5a}
\Vertex[NoLabel,x=5,y=0]{6a}
\Vertex[NoLabel,x=6,y=0]{7a}
\Vertex[NoLabel,x=7,y=0]{8a}
\Vertex[NoLabel,x=8,y=0]{9a}
\Vertex[NoLabel,x=6,y=1]{10a}
\Edges(2a,3a,4a,5a,6a,7a,8a,9a)
\Edges(7a,10a)
\tikzset{VertexStyle/.style= {inner sep=1.5pt, shape=circle}}
\Vertex[x=1,y=-0.4]{1}
\Vertex[x=2,y=-0.4]{2}
\Vertex[x=3,y=-0.4]{3}
\Vertex[x=4,y=-0.4]{4}
\Vertex[x=5,y=-0.4]{5}
\Vertex[x=6,y=-0.4]{6}
\Vertex[x=7,y=-0.4]{4}
\Vertex[x=8,y=-0.4]{2}
\Vertex[x=6,y=1.4]{3}
\end{tikzpicture}
\caption{The $\widetilde{E}_{8}$ diagram.}\label{figure:E8tilde}
\end{figure}
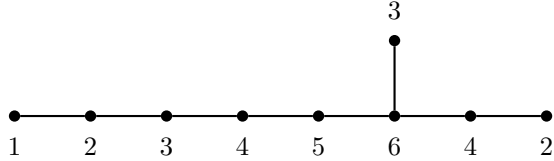

\subsection{The Vinberg lattice}\label{sec:vinberg}
\begin{figure}
\begin{tikzpicture}
\tikzset{VertexStyle/.style= {fill=black, inner sep=1.5pt, shape=circle}}
\Vertex[NoLabel,x=0,y=0]{1a}
\Vertex[NoLabel,x=1,y=0]{2a}
\Vertex[NoLabel,x=2,y=0]{3a}
\Vertex[NoLabel,x=3,y=0]{4a}
\Vertex[NoLabel,x=4,y=0]{5a}
\Vertex[NoLabel,x=5,y=0]{6a}
\Vertex[NoLabel,x=6,y=0]{7a}
\Vertex[NoLabel,x=7,y=0]{8a}
\Vertex[NoLabel,x=8,y=0]{9a}
\Vertex[NoLabel,x=6,y=1]{10a}
\Edges(1a,2a,3a,4a,5a,6a,7a,8a,9a)
\Edges(7a,10a)
\tikzset{VertexStyle/.style= {inner sep=1.5pt, shape=circle}}
\Vertex[x=0,y=-0.4]{1}
\Vertex[x=1,y=-0.4]{2}
\Vertex[x=2,y=-0.4]{3}
\Vertex[x=3,y=-0.4]{4}
\Vertex[x=4,y=-0.4]{5}
\Vertex[x=5,y=-0.4]{6}
\Vertex[x=6,y=-0.4]{7}
\Vertex[x=7,y=-0.4]{8}
\Vertex[x=8,y=-0.4]{9}
\Vertex[x=6,y=1.4]{10}
\end{tikzpicture}
\caption{The $E_{10}$ diagram.}\label{figure:E10}
\end{figure}
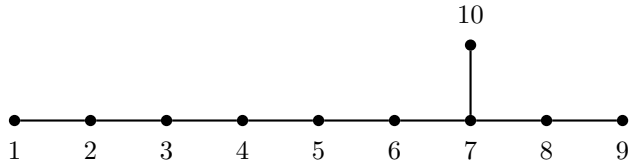
Let $E_{10}$ be the lattice corresponding to the Coxeter-Dynkin diagram in \Cref{figure:E10}.
It is an even unimodular lattice of signature $(1,9)$ and therefore $E_{10}\cong U \oplus E_8$.
Note that if $Y$ is an Enriques surface, then $S_Y\cong E_{10}$.
We denote by $e_{1}, \dots, e_{10}$ the basis of $E_{10}$ corresponding to its Coxeter-Dynkin diagram.

The space $E_{10} \otimes \FF_2$ comes equipped with a non-degenerate $\FF_2$-valued quadratic form $q$ defined by $q(x)= x^2/2 \mod 2$.
It can be interpreted as the discriminant form of $E_{10}(2)$ by identifying $\tfrac{1}{2}E_{10}/E_{10}$ with $E_{10}\otimes \FF_2$. The quadratic space $E_{10} \otimes \FF_2$ has a totally isotropic subspace of dimension $5$. For a subset $B \subseteq E_{10}$, we denote its image in $E_{10}\otimes \FF_2$ by $\overline{B}$.

\begin{theorem}\cite{shimada:ADE}
If $R\subseteq E_{10}$ is a root lattice, then its primitive closure $R'$ is a root lattice. There are exactly $184$ $O(E_{10})$-orbits of root sublattices $R \subseteq E_{10}$, and they are distinguished by 
their type $(\tau(R),\tau(R'))$.
\end{theorem}
\begin{remark}\label{shimada:irreducible}
If $R$ is an \emph{irreducible} root sublattice of rank at most $9$ and $R'$ an even overlattice, then $(\tau(R),\tau(R'))$ appears in the list of Shimada.
The cases with $R\neq R'$ are $(A_7,E_7)$, $(D_8,E_8)$ and $(A_8,E_8)$. Indeed, $A_7,D_8,A_8$ are exactly the irreducible root lattices of rank at most $9$ whose discriminant group contains an isotropic vector. 
This vector is unique for $A_7$ and $A_8$, but there are two isotropic vectors in the discriminant group of $D_8$ exchanged by the diagram symmetry. Therefore the diagram symmetry of an $(A_7,E_7)$ and $(A_8,E_8)$ extends to an isometry of $E_{10}$, but not the one of of $(D_8,E_8)$.
\end{remark}
The following lemma will be convenient to choose coordinates.
\begin{lemma}\label{lem:isotropic_in_Rperp}
Let $R\subseteq E_{10}$ be an irreducible negative definite root sublattice of rank at most $8$. Then $R^\perp$ contains a unique isotropic vector up to the action of $O(E_{10})_{R}$.
\end{lemma}
\begin{proof}
This is a standard computation. One first checks that $O(E_{10})_{R} \to O(R^\perp)$ is surjective and then computes the number of isotropic vectors in $R^\perp$ up to $O(R^\perp)$.
\end{proof}

\subsection*{Quadratic forms in characteristic $2$.}
We refer to the book of M. Kneser \cite{kneser} for the general theory of quadratic forms over fields.
Let $K$ be a field and $q:V \to K$ be a quadratic form. The pair $(V,q)$ is called a quadratic space. By definition 
\[b(x,y)=q(x+y)-q(x)-q(y)\] 
is a symmetric bilinear form. 
For $X \subseteq V$ set $X^\perp = \{v \in V \mid \forall x \in X\colon b(x,v)=0 \}$.
We call $V$ regular if the radical of its bilinear form $\mathrm{rad}(b):=V^\perp=0$ and half-regular if 
$\mathrm{rad}(q):=q^{-1}(\{0\})\cap V^\perp = 0$.

Throughout the paper we will use the following quadratic spaces over $\FF_2$:
\begin{itemize}
    \item $U_2=(\FF_2^2,q)$, where $q(x_1,x_2)=x_1x_2$,
    \item $V_2=(\FF_2^2,q)$ with $q(x_1,x_2)=x_1^2+x_1x_2+x_2^2$,
    \item $[0]=(\FF_2,q)$ with $q(x)=0$, and
    \item $[1]=(\FF_2,q)$ with $q(x)=x^2=x$. 
\end{itemize}
The relations 
\[U_2\oplus U_2 \cong V_2 \oplus V_2,\;\; U_2 \oplus [1]\cong V_2 \oplus [1],\;\; [1]\oplus [1] \cong [1] \oplus [0]\] 
lead to the following classification. 
\begin{theorem}
Let $W$ be a quadratic space over $\FF_2$. Then there are unique integers $u,v,k,e$ with $v,e\leq 1$ and $ev=0$ such that 
\[W \cong U_2^u \oplus V_2^v \oplus [1]^e \oplus [0]^k.\]
\end{theorem}
\begin{proof}
    Special case of \cite[Satz 2.15]{kneser}.
\end{proof}
\begin{table}
\begin{tabular}{cl|cl|ccl}
$R=R'$ & $\overline{R}$  & $R=R'$& $\overline{R}$  & $R$   &$R'$ & $\overline{R}$ \\
\hline
$A_1$&  $[1]$            & $D_4$ & $V_2\oplus [0]^2$ & $A_7$ & $E_7$ & $U_2^3$\\
$A_2$&  $V_2$            & $D_5$ & $U_2\oplus V_2 \oplus [0]$ & $A_8$ & $E_8$ & $U_2^4$\\
$A_3$&  $V_2 \oplus [0]$ & $D_6$ & $U_2\oplus V_2 \oplus [0]^2$ & $D_8$ & $E_8$ & $U_2^3 \oplus [0]$\\
$A_4$&  $U_2\oplus V_2$  & $D_7$ & $U_2^3 \oplus [0]$ & $E_{10}$ & $E_{10}$ & $U_2^5$ \\
$A_5$&  $U_2^2\oplus[1]$ & $D_8$ & $U_2^3 \oplus [0]^2$ & &  & \\
$A_6$&  $U_2^3$          & $D_9$ & $U_2^4\oplus [0]$ & &  & \\
$A_7$&  $U_2^3\oplus [0]$& $E_6$ & $U_2^2\oplus V_2$ &  & \\
$A_8$&  $U_2^4$          & $E_7$ & $U_2^3 \oplus [1]$&  & \\
$A_9$&  $U_2^4\oplus [1]$  & $E_8$ & $U_2^4$ &  & 
\end{tabular}
\caption{Irreducible root sublattices of $E_{10}$ and their quadratic forms mod $2$.}\label{table:root_mod2}
\end{table}

\begin{example}\label{example:root_invariant}
The following examples illustrate how \Cref{table:root_mod2} can be computed. There we denote by $R\subseteq E_{10}$ an irreducible root sublattice, by $R'\subseteq E_{10}$ its primitive closure, and by $\overline{R}\subseteq E_{10}\otimes\FF_2$ the image of $R\hookrightarrow E_{10}\to E_{10}\otimes\FF_2$.
In general, $\overline{R}\neq R\otimes \FF_2$ and the difference is a subtle but important part of the Nikulin root invariant defined in \Cref{sec:root_invariants}. 
\begin{enumerate}
\item Let $R\cong A_7$. It has a unique even overlattice $R' \cong E_7$. We obtain from a computation that its 2-adic Jordan decomposition is 
$R \otimes \ZZ_2 = U \oplus V$ with $U$ even unimodular of rank $6$ and $V$ of rank one with gram matrix $[8]$. Then $R' \otimes \ZZ_2 = U \oplus \frac{1}{2}V$, and we see that 
the kernel of $R \otimes \FF_2 \to R'\otimes \FF_2$ is given by $V \otimes \FF_2$. Moreover, we note that $R \otimes \FF_2$ is a quadratic space of rank $7$, and has a $1$-dimensional radical (of its quadratic and bilinear form). The space $R'\otimes \FF_2$ is $7$-dimensional. The radical $\mathrm{rad}(b)$ of its bilinear form $b$ is $\mathrm{rad}(b)=\left(\frac{1}{2}V \right)\otimes \FF_2$ and 
the radical of its quadratic form $\mathrm{rad}(q)=0$. 
\item For $R\cong D_8$ and $R'\cong E_8$, we obtain in a similar fashion that $R \otimes \ZZ_2 = U \oplus V(2)$ with $U$, $V$ even unimodular of rank $6$, respectively $2$. 
We see that $R \otimes \FF_2$ is a quadratic space with a two dimensional radical.
Since $[R':R]=2$, the kernel $R \otimes \FF_2 \to R'\otimes \FF_2$ has order $2$. It follows that its image $\overline{R}$ is a quadratic space of dimension $7$ which has a one dimensional radical $\mathrm{rad}(q)=\mathrm{rad}(b)$. 
\item Finally, for $R\cong A_8$ and $R'\cong E_8$. We have $[R':R]=3$, and so $R\otimes \FF_2=R'\otimes \FF_2=\overline{R}$.
\end{enumerate}
\end{example}

\subsection*{Weyl groups}
Let $L$ be a lattice. An element $r \in L\otimes \RR$ with $r^2=n$ is called an $n$-vector. It is an $n$-\emph{root} if the reflection
\[s_r\colon L \otimes \QQ \to L \otimes \QQ \quad x \mapsto x-\frac{2x.r}{r^2}r\]
preserves $L$. 
Any vector $x \in L$ of square $x^2=k \in\{\pm 1, \pm 2\}$ is a $k$-root. Since we are mostly interested in $(-2)$-roots, 
when we speak of a root, we mean a $(-2)$-root.
For $\Delta \subseteq L$ a set of roots, we call $W(\Delta)=\langle s_r | r \in \Delta \rangle$ the \emph{Weyl group} of $\Delta$.  We write $W(L)$ in the case where $\Delta=\Delta(L)$ is the set of roots of $L$.

For $(V,q)$ a quadratic space over $\FF_2$, we denote by $\Delta(V)=\{v \in V \mid  q(v)=1\}$. For $v \in \Delta(V)$, the formula $s_v(x)= x + b(x,v)v$ defines an isometry (a so called symplectic transvection). 
The case we have in mind is when $V = E_{10}\otimes \FF_2$ and $v=\overline{r}$ for $r \in E_{10}$ a root. Then $s_r$ induces the isometry  $s_{v}$ of $V$.
If $\sigma \subseteq \Delta(V)$ is a subset, we denote by $W(\sigma)\subseteq O(V)$ the group generated by the symplectic transvections $s_v$ with $v \in \sigma$.

\subsection*{Chambers}
Let $L$ be a hyperbolic lattice. Fix one of the two connected components of $\{x \in L \otimes \RR \mid x^2>0\}$. We denote it by $\P_L$ and call it the positive cone of $L$. We also denote by $\overline{\P}^{\QQ}_L$ the union of $\P_L$ with the rays $\RR_{>0} v$ generated by isotropic $v\in L\setminus\{0\}$ in the closure of $\P_L$.
For $Y$ an Enriques surface and $L=S_Y$, we call the component containing an ample class the \emph{positive cone} $\P_Y$ of $Y$.

Let $\Delta \subseteq L \otimes \RR$ be a subset such that $x^2<0$ for all $x\in \Delta$ and such that $x \in \Delta$ implies $-x \in \Delta$.
The connected components of $\P_L\setminus \bigcup_{r \in \Delta} r^\perp$ are called $\Delta$-\emph{chambers}. The closure of a chamber in $L \otimes \RR$ is called a \emph{closed $\Delta$-chamber}. 
Let $C$ be a $\Delta$-chamber. It defines a partition $\Delta = \Delta^+_C \cup -\Delta^+_C$ with 
\[\Delta^+_C:=\{r \in \Delta \mid \forall c \in C \colon r.c > 0\} =\{r \in \Delta \mid \exists c \in C \colon r.c>0\}.\]
We call the elements of $\Delta^+_C$ \emph{positive} (with respect to $C$). 

\subsection*{Indecomposable roots}
Let $\Delta$ consist of roots. Then an element $r$ of $\Delta^+_C$ is called \emph{indecomposable} (with respect to the positive root system $\Delta^+_C$) if $r$ cannot be written as a sum $r = \sum_i r_i$ with each $r_i \in \Delta^+_C \cup (\Pbar_L \cap L)$. In a geometric context this would mean that each $r_i$ is "effective".
The set of indecomposable elements of $\Delta^+_C$ is written as $\R(\Delta^+_C)$. 
Their reflections generate the Weyl group, i.e., $W(\Delta^+_C) = W(\R(\Delta^+_C))$.
\subsection*{Induced chambers} 
Let $S \subseteq L$ be a sublattice and $\P_S$ the positive cone of $S$ contained in $\P_L$. 
The connected components of 
\[\P_S \setminus \bigcup_{r \in \Delta} r^\perp\] 
are called \emph{induced chambers}. 
We see that any induced chamber $C_S$ can be written as $C \cap \P_S$ for some $\Delta$-chamber $C$. The converse, however, is not true:
for $C$ a $\Delta$-chamber it can happen that $C \cap \P_S$ is of smaller dimension than $S$. 

The walls cutting out the induced chambers from the positive cone $\P_S$ can be described as follows.
For $x \in L$ write $x_S \in S^{\vee}$ for the orthogonal projection of $x$ to $S^\vee\subset S\otimes\QQ$.
Set
\[\Delta|_S :=\{ x_S \mid x \in \Delta: x_S^2<0\}.\]
The induced chambers are precisely the $\Delta|_S$-chambers. 

For $\Delta$ the set of $(-2)$-roots of a lattice $L$, we speak of an $L|S$-chamber.
Note that if $r \in \Delta$ is a root, then $r_S$ is not necessarily a root of $S$ and the reflection $s_{r_S} \in O(S \otimes \QQ)$ does not necessarily preserve $\Delta|_S$; hence there is not necessarily a Weyl group acting on the set of induced chambers. 

\subsection{The Nef cone of an Enriques surface.}
Let $Y$ be an Enriques surface with covering K3 surface $\pi \colon X \to Y$ and covering involution $\epsilon$. Let $S_X:=\mathrm{Num}(X)$ and $S_Y:=\mathrm{Num}(Y)$ be the numerical lattices of $X$ and $Y$.
We identify the invariant lattice $(S_X)^\epsilon = (S_X)_+$ with $S_Y(2)$ via $\pi^*$.

The set of \emph{splitting roots} of $Y$ is defined as 
\[\Delta(Y):=\Delta(S_X)|_{S_Y(2)}.\]
It consists entirely of $(-4)$-roots with respect to $S_Y(2)$ (see \Cref{rmk:splitting}). 
With respect to the quadratic form on $S_Y$ they are $(-2)$-roots.  We shall regard $\Delta(Y)$ as a subset of $S_Y$.
\begin{remark} \label{rmk:splitting}
Explicitly, 
\begin{equation}\label{eqn:splitting}
     \Delta(Y)=\{r \in S_Y \mid r^2=-2, \exists v \in S_{X-}: v^2=-4 \mbox{ and } (\pi^*(r)+v)/2 \in S_X\}.
\end{equation}
Note that if $r\in\Delta(Y)$ and $r'$ is a root in $S_Y$ with $r' \equiv r \mod 2S_Y$, then $(\pi^*(r')+v)/2 \in S_X$ and $r'$ is also a splitting root.
\end{remark}
\begin{remark}
By \cite[Lemma 3.3, Lemma 3.4]{brandhorst-gonzalez:527} a $(-2)$-root $r$ in $S_Y$ is a splitting root if and only if there exists a $(-2)$-root $\widetilde{r}\in S_X$ with 
\[\pi^*(r) = \widetilde{r} + \epsilon^*(\widetilde{r}).\]
By Riemann-Roch on the K3 surface $X$, $\widetilde{r}$ or $-\widetilde{r}$ is effective. Hence $r = \pi_*(\widetilde{r})$ or $-r = \pi_*(-\widetilde{r})$ is effective. The set of \emph{effective splitting roots} is denoted by $\Delta^+(Y)$.  For a root lattice $R \subseteq S_Y$ with $\Delta(R) \subseteq \Delta(Y)$, we set $\Delta^+(R):=\Delta(R)\cap \Delta^+(Y)$.
\end{remark}
The Weyl group of $Y$ is defined as the group $W(Y):=W(\Delta(Y))$. 
Let $\aut(Y)$ be the image of $\Aut(Y)$ in $O(S_Y)$.
Define the modular stabilizer of $Y$ as $G_Y:= \langle W(Y) \cup \aut(Y)\rangle $ and let $\R(Y)$ be the set of classes of smooth rational curves of $Y$. Finally, let 
\[G_0:=\ker\left(O^+(S_Y) \to O(S_Y\otimes \FF_2)\right)\] 
be the level $2$ congruence subgroup of $O^+(S_Y)$.
We have the following:
\begin{proposition}\label{aut_and_weyl_group}\cite{dolgachev:OnAutomorphismsOfEnriquesSurfaces}\cite[Proposition 3.10]{brandhorst-gonzalez:527}
Let $Y$ be an Enriques surface over an algebraically closed field of characteristic not two.
\begin{enumerate}
\item The Weyl group $W(Y)$ is a normal subgroup of the modular stabilizer $G_Y$. 
\[G_Y = W(Y) \rtimes \aut(Y)\]
\item The Weyl group acts simply transitively on the set of $\Delta(Y)$-chambers.
\item The ample cone of $Y$ is the $\Delta(Y)$-chamber determined by the positive root system $\Delta^+(Y)$,
and the nef cone $\Nef_Y$ is its closure.
\item $\aut(Y)$ is the stabilizer of $\Nef_Y$ in $G_Y$.
\item The indecomposable elements of $\Delta^{+}(Y)$ are exactly the classes of smooth rational curves of $Y$, i.e. $\R(\Delta^{+}(Y))=\R(Y)$.
\item $\R(Y)$ is in bijection with the set of facets of $\Nef_Y$ via $r\mapsto r^\perp \cap \Nef_Y$ and $W(Y)=W(\R(Y))$, i.e. $\R(Y)$ is a root basis. 
\item $G_0 \subseteq G_Y$.
\end{enumerate}
\end{proposition}
\begin{proof}
(1-4) Are essentially proven in \cite[Proposition 3.2]{dolgachev:OnAutomorphismsOfEnriquesSurfaces} (See also \cite[Proposition 3.8]{brandhorst-gonzalez:527}).
(5) is surely known to experts. For lack of a reference, we give a proof. We first show $\R(Y)\subseteq\R(\Delta^{+}(Y))$: let $C\subseteq Y$ be a smooth rational curve, so that $C^2=-2$. 
If $[C]\in\Delta(Y)$ were not irreducible, we could write $[C]=r_1+r_2+\ldots+r_k$ with $r_1,\ldots,r_k\in \Delta(Y)^+\cup (\P_Y \cap S_Y)$, in particular $r_i=[D_i]$ for some effective divisors $D_i\subseteq Y$. 
Then $-2=C^2=C\cdot(D_1+\ldots+D_k)$ implies that $C$ is a component of at least one $D_i$, say $D_1$. We would then have $[D_1-C]+[D_2]+\ldots+[D_k]=0\in S_Y$. 
But this can only happen if $D_1-C=D_2=\ldots=D_k=0$. So $[C]=[D_1]$ is the only possible decomposition of $[C]$.

Let $r \in \Delta^+(Y)$ be indecomposable. 
Since $r$ is effective, we can write $r = [R]$ for an effective divisor $R$. Write $R = \sum_{i=1}^n C_i$ for irreducible curves $C_1,\dots, C_n$ in $Y$. Then $r=\sum_{i=1}^n c_i$ for $c_1,\dots, c_n \in S_Y$ classes of irreducible curves. Note that $c_i \neq 0$ since each $C_i$ is a curve.
In particular, $c_i \in \Delta^+(Y) \cup (\Pbar_Y \cap S_Y)$. 
Since $r$ is indecomposable, $n=1$ and $r = c_1$ is the class of an irreducible curve of self intersection $-2$, i.e. a smooth rational curve. 
(6) See e.g. \cite[Proposition 2.2.1, \S 0.8]{cdl:enriquesI}
(7) is \cite[Proposition 3.10]{brandhorst-gonzalez:527}
\end{proof}
It would be interesting to know how much of this proposition stays true for Enriques surfaces in characteristic $2$ with or without a K3-cover.\\

\noindent The image of $G_Y$ in $O(S_Y\otimes \FF_2)$ is called the \emph{Vinberg group} of $Y$ and denoted $\Gbar_Y$.  
\begin{remark}\label{remark:GbarY}
By \Cref{aut_and_weyl_group}(6) the knowledge of $\Gbar_Y\cong G_Y/G_0$ is sufficient to recover $G_Y$ as the preimage of $\Gbar_Y$ under $O^+(S_Y)\to O(S_Y \otimes \FF_2)$.
In particular $g \in O^+(S_Y)$ lies in $G_Y$ if and only if $\overline{g} \in \Gbar_Y$.
\end{remark}

\section{Root invariants}\label{sec:root_invariants}
For $R \leq E_{10}$ a sublattice, we denote by $\overline{\Delta}(R)$ the image of $\Delta(R)$ in $E_{10}\otimes \FF_2$. Define $\DeltabarY$ as the image of the set of splitting roots $\Delta(Y) \subseteq S_Y$ in $S_Y \otimes \FF_2$. Note that by \Cref{rmk:splitting} a root $r$ of $S_Y$ is a splitting root if and only if $\overline{r} \in \DeltabarY$.

The set $\DeltabarY$ is computed as follows:
Let $\pi_{\pm}\colon S_X \to S_{X\pm}^\vee$ be the orthogonal projection onto the invariant, respectively anti-invariant, lattice of $S_X$ with respect to the Enriques involution $\epsilon$. Note that $S_{X-}$ cannot contain any roots because $\epsilon$ preserves the ample cone of $X$.

Let $R$ be the $(-2)$-root sublattice of $(2\pi_-(S_X))(\frac{1}{2})$. Let $\alpha$ be given by the composition of
\begin{equation}\label{eq:alpha1}
R/2R \cong \tfrac{1}{2}R/R \xrightarrow{\beta} \pi_-(S_X)/S_{X-}
\end{equation}
and
\begin{equation}\label{eq:alpha2}
    \pi_-(S_X)/S_{X-} \xrightarrow[\gamma]{\sim} \pi_+(S_X)/S_{X+} \hookrightarrow S_{X+}^\vee/S_{X_+}\cong S_Y\otimes \FF_2
\end{equation}
where $\gamma$ is the glue map of the primitive extension $S_{X+}\oplus S_{X_-} \subseteq S_X$. (Keep in mind that $S_X$ is not necessarily unimodular.)
Then $\DeltabarY$ is the image of $\overline{\Delta}(R)$ under the homomorphism $\alpha$. 

The \emph{Nikulin root invariant} is the ADE-type of $R$ together with the kernel of the map $\alpha$, i.e the tuple $(\tau(R),\ker \alpha)$. It is more easily computed as follows:
Let $R'$ be the primitive closure of $R$ in $S_X(\frac{1}{2})$. The map $\beta$ factors as
\[\tfrac{1}{2}R/R \to \tfrac{1}{2}R'/R' \hookrightarrow \pi_-(S_X)/S_{X_-}.\]
Hence, the kernel of $\alpha$ agrees with the kernel of
\[R/2R \to R'/2R'.\]

Given a positive set of roots $\Delta^+$, we may view it as a graph by taking $\Delta^+$ as the set of vertices and by connecting two distinct $x,y \in \Delta^+$ with $n$ edges if and only if $x.y=n$. We view $\DeltabarY$ (and $\overline{\Delta}(R)$) as a graph by connecting $x,y \in \DeltabarY$ (respectively, $x,y \in \overline{\Delta}(R)$) with an edge if and only if $x.y=1$. Since $\alpha(x).\alpha(y)\equiv x.y \mod 2$, $\alpha\colon \overline{\Delta}(R) \to S_Y\otimes \FF_2$ preserves edges, i.e. restricts to a morphism of graphs $\Delta^+(R) \to \DeltabarY$.
From the next lemma it will follow that this morphism is injective.
\begin{lemma}\label{roots_mod2_injective}
Let $R$ be a negative definite root lattice and $R' \supseteq R$ an overlattice.
Let $\Delta^+(R)$ be a positive root system.
If $R'$ does not contain any $(-1)$-vectors, then $\Delta^+(R) \to R'/2R'$ is injective.
\end{lemma}
\begin{proof}
For any $\delta,\delta' \in \Delta^+(R)$ we have $\delta.\delta' \in \{0,\pm 1, \pm 2\}$ and $\delta.\delta'=\pm 2$ if and only if $\delta=\pm \delta'$ (since $(\delta \pm \delta')^2=-2 + 4 - 2=0$ implies $\delta-\delta'=0$ because the lattice $R$ is definite.) Now, suppose that $\delta \neq \pm\delta'$ have the same image in $R'/2R'$. Then $\delta.\delta'=0$ because $\delta.\delta' \equiv \delta^2 =-2 \equiv 0 \mod 2$.
Moreover, $(\delta-\delta')/2 \in R'$ is of square $-1$ in $R'$ contradicting the assumption.
\end{proof}

Let $R$ once again be the $(-2)$-root sublattice of $2\pi_-(S_X)(\frac{1}{2})\subseteq S_{X-}(\frac{1}{2})$ and $R'$ its primitive closure in $S_{X-}$. By \Cref{roots_mod2_injective},
$\Delta^+(R) \to R'/2R'$ is injective: indeed, any $(-1)$-root of $R'$ is of square $-2$ in $R'(2) \subseteq S_{X-}$, which is absurd since the latter lattice does not contain any $(-2)$-roots.
This proves the following proposition.
\begin{proposition}\label{prop:root_systems_mod2}
The map $R \to R/2R \xrightarrow{\alpha} S_Y \otimes \FF_2$ with $\alpha$ as in \cref{eq:alpha1,eq:alpha2} restricts to an isomorphism $\Delta^+(R)\to \DeltabarY$ of graphs. In particular, their connected components $\sigma$ agree.
\end{proposition}
\begin{definition}\label{def:R_sigma}
We denote by $R_\sigma \leq R$ the irreducible root sublattice whose roots are in bijection with $\sigma$ via $\alpha$ as in \cref{eq:alpha1,eq:alpha2}.
\end{definition}
Since one can recover the ADE-type of a positive root system from its graph, $\DeltabarY$ has a well-defined ADE-type, and so do its connected components.
The component $\sigma$ also contains information on the primitive closure $R_\sigma'$. The type, as defined below, keeps track of this. 

\begin{definition}\label{def:rank_type_sigma}
Let $\sigma \subseteq \DeltabarY$ be a connected component. 
Its \emph{rank} $r$ is defined as the rank of its underlying root system (and coincides with the rank of the root lattice $R_\sigma$). Its \emph{dimension} is
defined as $d=\dim_{\FF_2} \langle \sigma \rangle$ and its kernel $k$ is $k=r-d$.
Its \emph{type} is defined as the tuple $(\tau(R_\sigma),(\ZZ/2\ZZ)^k)$. 
\end{definition}

\begin{proposition}\label{prop:connected_comp_delta}
Let $\sigma$ be a connected component of $\DeltabarY$ of rank $r$ and kernel $k$. 
Then $r\leq 9$ and either 
\begin{itemize} 
\item $k=0$ or 
\item $k=1$, in which case $R_\sigma$ is of type $D_8$ or $A_7$ . 
\end{itemize}

\end{proposition}
\begin{proof}
We keep the notation of this section.
The rank of $R\leq S_{X-}$ is at most 
\[\rk R \leq \rk S_{X_-}=\rho(X) - 10 \leq 22-10=12.\] 

Let $R_\sigma'$ denote the primitive closure of $R_\sigma$ in $S_{X-}$. The vector space $\langle \sigma \rangle \subseteq S_Y\otimes \FF_2$ is isomorphic to the image of $R_\sigma/2R_\sigma \to R_\sigma'/2R_\sigma'$. Hence $\dim \langle \sigma \rangle = \rk R_\sigma - k$,
where $k$ is given by $2^ks = [R_\sigma':R_\sigma]$ with $2\nmid s \in \ZZ$. Inspecting the irreducible root lattices of rank at most $12$, we find that $k\leq 1$ and that $k=0$ unless $R_\sigma \cong A_7$ or $R_\sigma\cong D_8$.

If $\rk R_{\sigma} \geq 9$, then $\rk R_\sigma = \dim \langle \sigma \rangle \leq \dim S_Y \otimes \FF_2 = 10$. Otherwise, $\rk R_\sigma \leq 8$.
In any case $\rk R_\sigma \leq 10$. If $\rk R_\sigma = 10$, then $R_\sigma \cong A_{10}$ or $D_{10}$. By the above there is an isomorphism of quadratic spaces $5U_2 \cong S_Y\otimes \FF_2 = \langle \sigma \rangle \cong R_\sigma/2R_\sigma \cong 4U_2\oplus 2[0]$. But $D_{10}$ has a degenerate bilinear form mod $2$. The quadratic form of $A_{10}\otimes \FF_2=4U_2 \oplus V_2$ is non-degenerate but of a different isomorphism class than that of $S_Y\otimes \FF_2$. Hence, both give a contradiction.
\end{proof}

\begin{lemma}\label{lem:existsRwithDeltaRequalSigma}
Let $\sigma$ be a connected component of $\DeltabarY$.
Then there exists a root lattice $R \leq S_Y$ with $\sigma=\overline{\Delta}(R)$.
\end{lemma}
\begin{proof}

By \Cref{prop:connected_comp_delta}, the rank of $\sigma$ is at most $9$. By \Cref{shimada:irreducible}, we find a root sublattice $R_0 \leq S_Y$ with $R_\sigma \cong R_0$ and $R_0' \cong R_\sigma'$.
Fix an isometry $f\colon R_0 \to R_\sigma$. 
Then $R_\sigma'':=(f\otimes \id_{\QQ})(R_0')$ is an even overlattice of $R_\sigma$ with $[R_\sigma':R_\sigma]=[R_\sigma'':R_\sigma]$.
Since $O(R_\sigma)$ acts transitively on the set of even overlattices of $R_\sigma$ of a given index (see e.g. the discussion in Section 3.3 of \cite{ebeling:lattices-and-codes}), we can compose $f$ by an element of $O(R_\sigma)$ if necessary, such that $f$ extends to an isometry $f'\colon R_0' \to R_\sigma'$. 
Then 
\[\beta:=\alpha\circ(f\otimes \id_{\FF_2}) \colon R_0\otimes \FF_2 \to \langle \sigma \rangle\]
is an isometry. Since $f \otimes \id_{\FF_2}$ is injective, $\ker \beta$ is mapped to $\ker \alpha|_{R_\sigma \otimes \FF_2}= \ker(R_\sigma \otimes \FF_2 \to R_\sigma'\otimes \FF_2)=:K_\sigma$ under $f\otimes \id_{\FF_2}$. Likewise $K_0:=\ker (R_0 \otimes \FF_2 \to R_0' \otimes \FF_2)$ is mapped to $K_\sigma$.
$$\begin{tikzcd}
K_0 \arrow[r]\arrow[d]& R_0\otimes\FF_2 \arrow[r] \arrow[d, "\cong", "f\otimes\id_{\FF_2}"'] & R_0'\otimes\FF_2 \arrow[d, "\cong"', "f'\otimes\id_{\FF_2}"] \\
K_\sigma \arrow[r]& R_{\sigma}\otimes\FF_2 \arrow[r]                      & R_{\sigma}'\otimes\FF_2             
\end{tikzcd}$$
This shows that $\ker \beta =K_0.$
Then $\beta$ descends to an isometry $\bar \beta\colon (R_0\otimes \FF_2) /K_0 \cong (R_0+2S_Y)/2S_Y \to \langle \sigma \rangle$.
The point is that $(R_0+2S_Y)/2S_Y$ is a submodule of $S_Y \otimes \FF_2$.

By Kneser's theorem \cite[(4.4)]{kneser}, $\bar \beta$ is the restriction of some $\overline{g}\in O(E_{10}\otimes \FF_2)$, which in turn is the reduction mod $2$ of some $g \in O(E_{10})$. 
Then $R:=R_0^g$, does the job:
\[\overline{\Delta}(R_0^g)=\overline{g}(\overline{\Delta}(R_0))=\overline{\beta}(\overline{\Delta}(R_0))=\alpha(\overline{\Delta}(R_\sigma))=\sigma.\]
\end{proof}

\begin{lemma}\label{lem:BofTypeSigmaexists}
Let $\sigma \subseteq \DeltabarY$ be a connected component of rank $n$. 
Then there exist $b_1,\dots, b_n \in \R(Y)$ with $\overline{\Delta}(\langle b_1, \dots, b_n \rangle)=\sigma$.
\end{lemma}
\begin{proof}
By \Cref{lem:existsRwithDeltaRequalSigma}, there exists a root lattice $R_0$  with $\overline{\Delta}(R_0)=\sigma$.
Let $h' \in \P_Y \cap (R_0)^\perp$. There exists $w \in W(Y)$ with $h:=w(h')$ nef. With $R = R_0^w$ we have $\overline{\Delta}(R)=\overline{\Delta}(R_0)^w = \sigma^w=\sigma$ where the last equality follows because the Weyl group $W(Y)$ preserves each connected component of $\overline{\Delta}(Y)$.
Then define $B$ as the set of indecomposable elements of $\Delta^+(R)$. By \Cref{aut_and_weyl_group} (5), $B \subseteq \R(Y)$.
\end{proof}

\section{Minimal faces - maximal configurations}
Recall from \Cref{aut_and_weyl_group} that the smooth rational curves on an Enriques surface $Y$ are in bijection with the facets of its nef cone.
To count their $\aut(Y)$-orbits, we take a closer look at the faces of the nef cone of higher codimension.
For $i \in \{1,\dots, 9\}$ we define
\begin{eqnarray}
\R^i(Y)\defeq \{B \subseteq \R(Y) \mid  \# B =i, \langle B \rangle \mbox{ is negative definite of rank } i\},\\
\Delta^i(Y)\defeq \{B \subseteq \Delta(Y) \mid  \# B =i, \langle B \rangle \mbox{ is negative definite of rank } i\}.
\end{eqnarray}
Let $\sigma \subseteq \DeltabarY$ be a connected component.
Define
\begin{eqnarray}
\R^i(Y,\sigma) = \{B \in \R^i(Y) \mid \overline{B}\subseteq \sigma\},\\
\Delta^i(Y,\sigma) = \{B \in \Delta^i(Y) \mid \overline{B}\subseteq \sigma\}.
\end{eqnarray}
For $B \in \R^i(Y)$, $B^\perp$ defines a $10-i$ dimensional face of the nef cone and every face arises in this way.

\begin{proposition}\label{prop:rational_curves}
The natural map
\[\phi \colon \R^i(Y)/\aut(Y) \to \Delta^i(Y)/G_Y\] 
is surjective for each $i \in \{1,\dots,9\}$.
\end{proposition}
\begin{proof}
Let $B \in \Delta^i(Y)$. Then $B^\perp$ contains a codimension $i$ face of some closed $\Delta(Y)$-chamber $D$ with $D\subseteq H_B=\{x \in \P_Y \mid \forall b \in B: x.b\geq0\}$. By \Cref{aut_and_weyl_group} (2-3), we find an element $w$ of the Weyl group $W(Y)\subseteq G_Y$ with $D^w = \Nef_Y$. Then $(B^w)^{\perp} \cap \Nef_Y$ is a codimension $i$ face of $\Nef_Y$ and each $b \in B^w$ is effective. Therefore, $B \in \R^i(Y)$.
\end{proof}

We consider the fibers of the map $\phi$ in \Cref{prop:rational_curves}.
For any subset $A \subseteq L\otimes \RR$ we denote by $A^\circ$ the relative interior of $A$ in its linear hull.
\begin{lemma}\label{lem:phi-fiber}
Let $\phi$ be as in \Cref{prop:rational_curves}, $i \in \{1,\dots, 9\}$ and $B \in \R^i(Y)$.
Then the map
\[\psi\colon \phi^{-1}(B G_Y) \to \{\Delta(Y)|_{B^\perp}-\mbox{chambers}\}/
(G_Y)_{\{B^\perp\}}\]
\[B^g \mapsto \left((\Nef_Y)^{g^{-1}}\cap B^\perp\right)^\circ(G_Y)_{\{B^\perp\}}\]
is bijective.
\end{lemma}
\begin{proof}
\underline{The map $\psi$ is well defined:} Let $B^g \in \R^i(Y)$  represent an element in the fiber of $\phi$ over $BG_Y$  where $g\in G_Y$.
Since $B^g \in \R^i(Y)$, \(\Nef_Y \cap (B^g)^\perp\) is a face of dimension $10-i$ of the nef cone.
In particular, its relative interior is non-empty. Hence, its image under $g^{-1}$, $\Nef_Y^{g^{-1}}\cap B^\perp$, has non-empty relative interior as well.
Thus, $\left(\Nef_Y^{g^{-1}}\cap B^\perp\right)^{\circ}$ is indeed a $\Delta(Y)|_{B^\perp}$-chamber. Let now $f \in \aut(Y)$. Then 
\[\Nef_Y^{(gf)^{-1}}\cap \; B^\perp = \Nef_Y^{g^{-1}} \cap\; B^\perp\] because $f$ preserves $\Nef_Y$. In particular their $(G_Y)_{\{B^\perp\}}$ orbits agree. 

\underline{The map $\psi$ is injective:}
Let \(g,\widetilde{g} \in G_Y \) with \(B^g,B^{\widetilde{g}} \in \R^{i}(Y)\)
and \(\psi(B^g\aut(Y)) = \psi(B^{\widetilde{g}}\aut(Y))\).
This means that there exists $h \in \stab(G_Y,B^\perp)$ such that
\[\left(\left(\Nef_Y^{g^{-1}}\cap \;B^\perp\right)^\circ \right)^h = \left(\Nef_Y^{\widetilde{g}^{-1}}\cap\; B^\perp\right)^\circ.\]
Thus $\Nef_Y^{g^{-1}}$ and $ \Nef_Y^{\widetilde{g}^{-1}h^{-1}}$ contain a common face defined by $B^\perp$. Let $x$ be a point in the relative interior of this face. Then $W(B)$ acts transitively on the set of closed $\Delta(Y)$-chambers containing $x$ (cf. \cite[Lemma 4.2]{brandhorst-gonzalez:527}), and therefore we find $w \in W(B)$ with
 $\Nef_Y^{g^{-1}}=\Nef_Y^{\widetilde{g}^{-1}h^{-1}w}$.
Now $k = \widetilde{g}^{-1}h^{-1}wg \in G_Y$ preserves the nef cone. Hence $k \in \aut(Y)$.
By construction, $k$ maps $\langle B^{\widetilde{g}}\rangle $ to $\langle B^g \rangle$. In particular, it maps the facet of the nef cone defined by $B^g$ to the facet defined by $B^{\widetilde{g}}$.
The rational curves cutting out these facets are $B^g$ and $B^{\widetilde{g}}$. Hence, $k$ maps $B^g$ to $B^{\widetilde{g}}$, and thus $B^g\aut(Y)=B^{\widetilde{g}}\aut(Y)$.

\underline{The map $\psi$ is surjective:}
Let $C$ be a $\Delta(Y)|_{B^\perp}$-chamber. Write $C = \left(\overline{C'}\cap B^\perp\right)^\circ$ for some $\Delta(Y)$-chamber $C'$. We have $\overline{C'} = \Nef_Y^w$ for some $w \in W(Y)$. Now $\overline{C} = \Nef_Y^w \cap B^\perp$. Thus $\overline{C}^{w^{-1}}= \Nef_Y \cap (B^{w^{-1}})^\perp$ is a $10-i$ dimensional face of the nef cone. As before we find $w' \in W(B)$ such that $B^{w'w^{-1}} \in \R^i(Y)$.
Its $\aut(Y)$-orbit maps to the orbit of $C$ under $\psi$.
\end{proof}

If $R\subseteq S_Y$, and $\Delta(R)\subseteq \Delta(Y)$, set $\Delta^+(R)=\Delta^+(Y)\cap \Delta(R)$. It is a positive root system and therefore contains a unique fundamental root system $B$. Its connected components form $ADE$ configurations.

\begin{lemma}\label{lem:external_root}
Let $B \in \R^i(Y)$, denote by $\pi_B\colon S_Y \to (B^\perp)^\vee$ the orthogonal projection and let $\delta \in \Delta(Y)$ be such that $\pi_B(\delta) \in \Delta(Y)|_{B^\perp}$. Then $B':=\{\delta\} \cup B \subseteq \Delta(Y)$ is a union of ADE configurations. In particular, $B' \in \Delta^{i+1}(Y)$.
If, in addition, $\delta \not\in B^\perp$, then $B'$ has at most as many connected components as $B$.
\end{lemma}
\begin{proof}
Since $\pi_B(\delta) \in \Delta(Y)|_{B^\perp}$, $\pi_B(\delta)^2<0$. Thus we find a vector $h \in \pi_B(\delta)^\perp \cap B^\perp$ with $h^2>0$. 
The root lattice $\langle B' \rangle$ is contained in the negative definite quadratic space $h^\perp \otimes \RR$, hence negative definite itself.
Since $B'$ is linearly independent, this implies that $B'$ is an $ADE$-configuration.
Now suppose that $\delta \not \in B^\perp$, then there exists some $b \in B$ with $b.\delta \neq 0$. This means that $\delta$ and $b$ are in the same connected component. 
\end{proof}
\begin{definition}
We say that a connected $B\subseteq \Delta^i(Y)$ is \emph{maximal} in $\Delta(Y)$
if it is not contained in any connected $B' \in \Delta^{i+1}(Y)$.

Similarly, we say that a connected $B \in \R^i(Y)$ is \emph{maximal} in $\R(Y)$ if it is not contained in any connected $B' \in \R^{i+1}(Y)$.
\end{definition}
\begin{lemma}\label{lem:extend}
Let $B \in \R^i(Y)$ be connected. Then $B$ is maximal in $\R(Y)$ if and only if it is maximal in $\Delta(Y)$.
\end{lemma}
\begin{proof}
If $B$ is maximal in $\Delta(Y)$, then it is maximal in $\R(Y)$, because $\R(Y) \subseteq \Delta(Y)$.

Conversely, suppose that $B \in \R^{i}(Y)$ is not maximal in $\Delta(Y)$.
Then there exists $\delta' \in \Delta(Y)\setminus B^\perp$ such that $\{\delta'\} \cup B =:B' \in \Delta^{i+1}(Y)$ is connected. In particular, $\delta' \neq \pi_B(\delta')$. Thus $\pi_B(\delta') \in \Delta(Y)|_{B^\perp} \setminus \Delta(B^\perp)$.
Since $B \in \R^i(Y)$, $\Nef(Y) \cap \P_{B^\perp}$ has the expected codimension, and is therefore a $\Delta(Y)|_{B^\perp}$-chamber. 

We  claim that it has a facet defined by some $\pi_B(\delta'') \in \Delta(Y)|_{B^\perp} \setminus \Delta(B^\perp)$ with $\delta'' \in \R(Y)$. Otherwise, its facets would all be defined by roots contained in $\Delta(Y)\cap B^\perp$, i.e.
it would be a $(\Delta(Y) \cap B^\perp)$-chamber.
Then its $W(\Delta(Y)\cap B^\perp)$-translates would tesselate the positive cone $\P_{B^\perp}$ of $B^{\perp}$. This would imply $\Delta(Y)|_{B^\perp} = \Delta(Y)\cap B^\perp$, which is not the case because of $\delta'$.
Set $B'' = \{\delta''\}\cup B \in \R^{i+1}(Y)$. By \Cref{lem:external_root} we know that $B''$ is connected.
\end{proof}
\begin{lemma}\label{lem:transitive_on_induced_chambers}
Suppose that $B \in \R^i(Y)$ is connected and maximal. Then:
\begin{enumerate}
\item $\Delta(Y)|_{B^\perp} = \Delta(Y) \cap B^\perp$,
\item $W(\Delta(Y)\cap B^\perp) \subseteq \stab(G_Y,B^\perp)$ acts transitively on the set of $\Delta(Y)|_{B^\perp}$-chambers,
\item the preimage of $B G_Y$ under $\phi$ is a single point.
\end{enumerate}
\end{lemma}
\begin{proof}
\begin{enumerate}
\item From the definitions it follows that $\Delta(Y)|_{B^\perp} \supseteq \Delta(Y) \cap B^\perp$. We show the other inclusion.
Suppose by contradiction that there exists a $\delta \in \Delta(Y)$ such that $\pi_B(\delta) \in \Delta(Y)|_{B^\perp} \setminus B^\perp$. Then \Cref{lem:external_root} applies and gives that 
\[B':=\{\delta\}\cup B \in \Delta^{i+1}(Y)\]
is a larger connected ADE graph than $B$, contradicting the assumption on $B$.

\item By (1) the walls of the $\Delta(Y)|_{B^\perp}$-chambers are all cut out by $(-2)$-vectors in $B^\perp \cap \Delta(Y)$. The corresponding reflections preserve $B^\perp$, $\Delta(Y) \cap B^\perp$ and are in $G_Y$. It is a classical fact that the Weyl group generated by reflections in $\Delta(Y) \cap B^\perp$ acts transitively on the collection of $\Delta(Y) \cap B^\perp$-chambers.

\item follows from (2) and \Cref{lem:phi-fiber}.
\end{enumerate}
\end{proof}

\section{Orbits by the 2-congruence subgroup}\label{sec:orbits-by-2congruence}
In this section we determine the ADE-sublattices of $E_{10}$ up to the action of $O(E_{10})(2)$.
The reader interested mainly in Enriques surfaces may skip the proofs and jump directly to \Cref{prop:B-mod2-orbit}.\\

Let $V$ be a quadratic space over a field $k$ and $W \leq V$ a subspace. We call a subspace $U\leq V$ dual to $W$ in $V$ if the map $U \to \Hom(W,k), u \mapsto b(u, \cdot)|_W$ is an isomorphism. Typically, such a space is denoted by $W^\vee:=U$. It is however not unique: if $U,U'$ are dual to $W$ then $U+W^\perp = U'+W^\perp$. 
\begin{lemma}\label{lemma:order_stabilizer_W_in_V}
Let $V$ be a regular quadratic space over \(\FF_2\) and $W \leq V$ a subspace. Let $K_b = \mathrm{rad}(b_{\mid W})= W^\perp \cap W$ be the radical of the bilinear form of $W$ and let $K_q =\mathrm{rad}(q_{\mid W})\subseteq K_b$ be the radical of the quadratic form on $W$.
Let $w = \dim W^\perp$, $k=\dim K_q$, $n = \dim K_b$.
Then the order of the pointwise stabilizer of $W$ in $O(V)$ is
\[\#O(V)_W = \# O(W^\perp/K_q)  2^{e}\]
where $e=n(n-1)/2 + k(w-n) + n - k$.
\end{lemma}
\begin{proof}
Write $W = W_1 \oplus K_b$ for some $W_1 \leq W$ and write $W^\perp = U_1 \oplus K_b$ for some $U_1 \leq W^\perp$. 
Since $V$ is regular, there exist a subspace $K_b^\vee \leq V$ dual to $K_b$ in $V$ such that 
\[V = W_1 \oplus (K_b + K_b^\vee)\oplus U_1,\]
where the sum in the middle is direct but not orthogonal.

We claim that the composition
\[O(V)_W \to O(W^\perp) \to O(W^\perp/K_q)\] is surjective. 
To this end write $K_b = K_q \oplus R$ and correspondingly $K_b^\vee = K_q^\vee + R^\vee$. 
In summary we have:
\begin{eqnarray}
V &=& W_1 \oplus \left((R \oplus K_q)+(R^\vee \oplus K_q^\vee)\right) \oplus U_1\\
W &=& W_1 \oplus R \oplus K_q\\
W^\perp &=& U_1 \oplus R \oplus K_q
\end{eqnarray}
Note that $W^\perp / K_q\cong R \oplus U_1$ as quadratic spaces. 
Let 
\[t: R \oplus U_1 \to R \oplus U_1\] 
be an isometry.
Then $t$ preserves $R$ because it is the radical. Since $R$ consists of at most two elements, one of which is zero, $t$ fixes $R$ pointwise. 

We check the conditions of \cite[Satz 4.3]{kneser} (in the notation of loc. cit.) with $E:=V$, $F:=G:=W^\perp$, $H:= (R +R^\vee)\oplus U_1$ and $t$ as above.
Trivially, $tx \equiv x \mod H$ for all $x \in R \oplus U_1$. By \cite[Satz 4.3]{kneser}, $t$ extends to an isometry $t'$ of $V$ satisfying $t'x \equiv x \mod H$ and fixing $H^\perp=W_1 \oplus (K_q+K_q^\vee)$ pointwise. 

Since $t$ fixes $R$ pointwise as well, $t'$ restricts to the identity on $H^\perp \oplus R$. One checks that $W \subseteq H^\perp \oplus R$, hence $t'$ restricts to the identity of $W$, i.e. $t' \in O(V)_W$.
The claim is proven.

Let $C$ be the kernel of $O(V)_W \to O(W^\perp/K_q)$. 
First consider the image of 
\begin{equation}\label{eq:CtoOWperp}
C \to O(W^\perp).
\end{equation}
Let $f$ be an element of the image. Then $f$ is of the form 
$f = \id_{W^\perp} + g$ with $g\colon U_1 \to K_q$ linear.
Conversely, given $g \in \Hom(U_1, K_q)$,
set $f = \id_{W^\perp}+g$. By \cite[Satz 4.3]{kneser} once more, this time with $E:=V$, $F:=G:=W^\perp$, $H:=U_1 \oplus K_b \oplus K_b^\vee$ and $t:=f$, $f$ can be extended to an isometry $f' \in O(V)$ 
fixing $H^\perp=W_1$. Since $f$ already fixes $K_b$, $f'$ fixes $W_1+K_b=W$, i.e. $f' \in O(V)_W$.
This shows that the image of $C$ in $O(W^\perp)$ is of order $2^c$ with 
\[c = \dim \Hom(U_1,K_q)=\dim K_q \cdot \dim U_1=k(w-n).\]

Let $D \leq C$ be the kernel of the map in \cref{eq:CtoOWperp}. It is the pointwise stabilizer $O(V)_{W+W^\perp}$ of $W+W^\perp$. We determine the order of $D$.
Since any $t \in D$ preserves $W_1\oplus U_1$ and this space is regular, we may assume that $W_1 \oplus U_1 =0$, i.e., that 
\[W=W^\perp = K_b.\]
Let $t \in D$. Then $t(y)=y$ for all $y \in K_b$ and $t(x)\equiv x \mod K_b$ for all $x \in K_b^\vee$.
This means that we can write $t = 1 + (0_{K_b}\oplus a)$ with a linear map 
\[a\colon K_b^\vee \to K_b \subseteq V\]
Conversely, let $a\colon K_b^\vee \to K_b$ be any linear map. The condition that $t$ preserves the bilinear form $b$ translates to 
\[b(x,a(y))=b(a(x),y) \quad  \mbox{for all } x,y \in K_b^\vee.\] 
Let $r_1 \in K_b^\vee$ be the (possibly zero), so called, \emph{characteristic vector} satisfying $q(x) = b(r_1,x)$ for all $x \in K_b$. That $t$ also preserves the quadratic form $q$ amounts to the linear condition
\[b(r_1+x,a(x))=0 \mbox{ for all } x \in K_b^\vee.\] 
We compute the rank of these linear conditions.
Choose a basis of $K_b$ and take the corresponding dual basis on $K_b^\vee$. Let $A$ be the matrix representing $a$ with respect to these bases. The condition that $a$ preserves the bilinear form translates to $A+A^t=0$. Because of characteristic $2$, this is a linear system of equations of rank $n(n-1)/2$.
The preservation of $q$ amounts to another $n$ equations independent from the previous ones involving the diagonal of $A$. 
If $r_1 = 0$, i.e., $K_q=K_b$, then they have rank $n$, but if $r_1\neq 0$, they have rank $n-1$ (because the equation $b(r_1+x,a(x))=0$ with $x=r_1$ is trivial). 

Thus, $\# D = 2^d$ with $d=n(n-1)/2$ for $K_q=K_b$ and $d=n(n-1)/2+1$ else.
This can be reformulated as 
\[d = n(n-1)/2+(n-k)\] 
in either case.

In total 
\[\# O(V)_W = \# O(W^\perp/K_q) \# C = \# O(W^\perp/K_q) 2^c \# D = \# O(W^\perp/K_q) 2^{c+d}. \]
\end{proof}
\begin{remark}
Let $V$ be a quadratic space. Write $V=V_1 \oplus V_2$ with $V_1$ half-regular and $q(V_2)=0$. 
Then $O(V) \to O(V/V_2) \times \mathrm{GL}(V_2)$ is surjective with kernel $\Hom(V_1,V_2)$.
For $V$ defined over a finite field the order of $O(V_1)$ is well known, see for instance \cite[(13.3)]{kneser}.
\end{remark}

Recall that for $R \leq E_{10}$ a submodule we write $\overline{R}:=(R+2E_{10})/2E_{10}\subseteq E_{10}\otimes \FF_2$.

We shall use the following facts from \cite{nikulin}.
Let $L$ be a non-degenerate lattice and $A$ a primitive and non-degenerate sublattice and $B:=A^\perp$.
Then \[L/(A\oplus B)=:\Gamma \leq D_A \oplus D_B\] 
is the graph $\Gamma=\Gamma_\alpha$ of an anti-isometry $\alpha \colon H_A \to H_B$ for certain subgroups $H_A \leq D_A$ and $H_B \leq D_B$. If $L$ is unimodular, then $H_A=D_A$ and $H_B=D_B$. In general we have $D_L = \Gamma^\perp / \Gamma$, and $f\oplus g \in O(D_A)\times O(D_B)$ preserves $\Gamma$ if and only if $f$ preserves $H_A$, $g$ preserves $H_B$ and $g(\alpha(x)) = \alpha(f(x))$ for all $x \in H_A$. In this case $f \oplus g$ induces an isometry on $\Gamma^\perp/\Gamma = D_L$.

\begin{lemma}\label{lem:GY-transitive_on_ADE}
Let $R \subseteq E_{10}$ be a negative definite, irreducible root sublattice and $R'$ its primitive closure.
Then the natural map 
\[\xi\colon O(E_{10})_R \to O(E_{10}\otimes \FF_2)_{\overline{R},\{\overline{R'}\}}\] 
is surjective
unless $R$ is of type $(A_9,A_9)$. In this exceptional case the image has index 2 in the codomain.
\end{lemma}
\begin{proof}
Let $K$ be the orthogonal complement of $R$ in $E_{10}$.
Since $E_{10}$ is unimodular, 
\[O(E_{10})_R = \{\id_{R}\} \times O^\sharp(K).\]

We identify $E_{10}\otimes \FF_2$ with the discriminant group of $E_{10}(2)$.
Choose and fix a primitive embedding of $E_{10}(2)$ into some even unimodular lattice $L$. Let $S$ be its orthogonal complement. Since $D_{K}=2D_{K(2)}\subset D_{K(2)}$ as groups, there is a natural map $O(D_{K(2)}) \to O(D_K)$.
Let $O^\sharp(D_{K(2)})$ denote its kernel.

There is a natural homomorphism 
\[\eta \colon O^\sharp(D_{K(2)}) \to O(E_{10}\otimes \FF_2)_{\overline{R}} \hookrightarrow O(D_{E_{10}(2)})\] 
defined as follows.
With $\Gamma = E_{10}/(R\oplus K)$ and $\Gamma^\perp = \tfrac{1}{2}E_{10}/(R\oplus K)$ its orthogonal complement in $D_{R(2)}\oplus D_{K(2)}$, we know by the discussion above this lemma that
$D_{E_{10}(2)}=\Gamma^\perp/ \Gamma$, which is a subquotient of $D_{R(2)}\oplus D_{K(2)}$. 
The image of $f \in O^\sharp(D_{K(2)})$ is given by first extending $f$ to $\id_{D_{R(2)}}\oplus f$, then restricting to $\Gamma^\perp$ and subsequently considering the induced map on $\Gamma^\perp/\Gamma =D_{E_{10}(2)}$.

We claim that $\eta$ is injective. Let $f \in \ker \eta$. Then $\id_{D_{R(2)}}\oplus f$ acts trivially on $D_{E_{10}(2)}$ and therefore
$f':=\id_{D_{R(2)}}\oplus f \oplus \id_{D_S}$ preserves $\Gamma_2 = L/(E_{10}(2)\oplus S)$. This implies that $f'$ also preserves $L/(R(2) \oplus K(2) \oplus S)$. Let $T$ be the orthogonal complement of $K(2)$ in $L$. Then $\Gamma_3 = L / (K(2) \oplus T)$ is also preserved by $f'$ and is the graph of an anti-isometry $\alpha \colon D_{K(2)}\to D_T$ with 
$\alpha \circ f = \id_{D_{T}} \circ \alpha$. Thus $f = 1$.

The homomorphism 
\[\eta_2\colon O^\sharp(K) \to O(E_{10}\otimes \FF_2)_{\overline{R}}\] 
factors as  
\[O^\sharp(K) \xrightarrow{\eta_1} O^\sharp(D_{K(2)})\xrightarrow{\eta} O(E_{10}\otimes \FF_2)_{\overline{R}}.\] 
We will compute the order of the image of $\eta_2$. Since $\eta$ is injective, it agrees with the order of $G_1:=\im \eta_1$.

A standard computation based on the theory of Nikulin \cite{nikulin} and Miranda-Morrison \cite[VIII Theorem 7.5]{miranda_morrison:embeddings} shows that $K$ is unique in its genus and that $O(K) \to O(D_K)$ is surjective. 
Let $G_2$ be the image of $O(K) \to O(D_{K(2)})$.

Since $G_1$ is the kernel of the surjective map $G_2 \to O(D_{K})$, we have 
\[ \# G_1 =\frac{\# G_2}{\#O(D_K)} =\frac{\# O(D_{K(2)})}{  \# O(D_K)  [O(D_{K(2)}): G_2 ]}.\]
By Miranda-Morrison theory the index $[O(D_{K(2)}):G_2]=1$, except for $R\cong A_9$ in which case $O(K) = \{\pm \id_{K}\}$ and $O^\sharp(K)=1$ is trivial. We record 
the order of $O(D_{K(2)})$ and $O(D_K)$ in a table. They are computed based on the formulas  found in \cite{brandhorst-veniani}.\\
\noindent
\begin{center}
\begin{tabular}{crc|crc}
$R$ & $\sharp O(D_{K(2)})$ & $\sharp O(D_K)$ & $R$ & $\sharp O(D_{K(2)})$ & $\sharp O(D_K)$  \\
\hline 
$A_1$ & $94755225600$ & 1  & $D_4$& $368640 $& $6 $\\
$A_2$ & $ 789626880 $&$ 2$ & $D_5$& $3840 $&$2 $\\
$A_3$ & $ 6635520 $&$ 2$ & $D_6$&$ 192$&$2 $\\
$A_4$ & $ 103680 $&$ 2$ & $D_7$&$ 16$&$2 $\\
$A_5$ & $ 2880$&$ 2$ & $D_8$&$4 $&$ 2$\\
$A_6$ & $ 144 $&$ 2$ & $D_9$&$2 $&$ 2$\\
$A_7$ & $ 16$&$ 2$& $E_6$ & $240$ & $2$\\
$A_8$ & $ 4$&$ 2$ & $E_7$ & $12$ & $1$\\
$A_9$ & $ 4$& $2$& $E_8$ & $2$ & $1$
\end{tabular}
\end{center}

Now that we have computed the order of $G_1$, we compare it with the order of 
$O(E_{10}\otimes \FF_2)_{\overline{R}\cup \{\overline{R'}\}}$.
Suppose that the index $[R':R]$ is odd. Then we have 
$\overline{R}=\overline{R'}$, hence 
\[O(E_{10}\otimes \FF_2)_{\overline{R},\{\overline{R'}\}}= O(E_{10}\otimes \FF_2)_{\overline{R}}.\]
Its order is computed in \Cref{lemma:order_stabilizer_W_in_V}.
The orders of $G_1$ and $O(E_{10}\otimes \FF_2)_{\overline{R}}$ agree unless $R$ is of type $(A_9,A_9)$ and then the index we seek is $2$.

It remains to treat the cases $(D_8,E_8)$ and $(A_7,E_7)$ where the index is $[R':R]=2$.
Suppose that $R$ is of type $(A_7,E_7)$. Recall from \Cref{table:root_mod2} that $\overline{R}$ is a non-degenerate quadratic space of rank $6$ contained in the quadratic space $R'\otimes \FF_2 \leq E_{10}\otimes\FF_2$, whose bilinear form has a $1$-dimensional radical. Its generator $v$ is invariant under $O(R'\otimes \FF_2)$ and hence under $O(E_{10}\otimes \FF_2)_{\overline{R},\{\overline{R'}\}}$. 
Since $\overline{R} \cup \{v\}$ generates $\overline{R'}$, $O(E_{10}\otimes \FF_2)_{\overline{R},\{\overline{R'}\}}$ is the pointwise stabilizer of $\overline{R'}$, whose order we already know from the case $(E_7,E_7)$ because $R' \cong E_7$.

For $R$ of type $(D_8,E_8)$ recall from \Cref{table:root_mod2} that $\overline{R}$ is a quadratic space of dimension $7$ with a one dimensional radical. It is a subspace of $R' \otimes \FF_2 \cong E_8 \otimes \FF_2$. 
Once again $O(E_{10}\otimes \FF_2)_{\overline{R}, \{\overline{R'}\}}$ is the pointwise stabilizer of $R'\otimes \FF_2$: to see this let $v$ be the generator of the radical of $\overline{R}$. Since $R'\otimes \FF_2$ is regular, we find $w \in R'\otimes \FF_2$ with $b(v,w)=1$. Any $c \in O(E_{10}\otimes \FF_2)_{\overline{R},\{\overline{R'}\}}$ must map $w$ to an element of $\overline{R}^\perp = \langle v,w \rangle$, i.e.
either to $w$ or to $v+w$. But $q(v+w)=b(v,w)+q(w) = 1+q(w) \neq q(w)$. Thus $w^c=w$, and therefore $c|_{R'\otimes \FF_2}=\id$.
\end{proof}

Recall from \Cref{shimada:irreducible} that the $O(E_{10})$-orbit of a root sublattice $R \leq O(E_{10})$ is determined by its type $(\tau(R),\tau(R'))$. 
We shall need a similar result for $O(E_{10})(2)$ orbits.

Before we can prove it, we need to prepare a lemma.
\begin{lemma}\label{lem:D8E8_stabilizer}
Let $B \subseteq \Delta(E_{10})$ be a configuration of roots of type $(D_8,E_8)$. Set $R=\langle B \rangle$.
Then 
\[O(E_{10}\otimes \FF_2)_{\{\overline{B}\},\{\overline{R'}\}}=O(E_{10}\otimes \FF_2)_{\overline{B}, \{\overline{R'}\}}.\]
\end{lemma}
\begin{proof}
Since $2 \nmid \det E_8$, $\overline{R'}$ is a regular subspace of $E_{10}\otimes \FF_2$. Hence, we have 
\[O(E_{10}\otimes \FF_2)_{\left\{\overline{R'}\right\}}=O(\overline{R'})\times O(V)\] 
with $V=\overline{R'}^\perp$. 
Thus, it suffices to prove that $O(\overline{R'})_{\overline{B}}=O(\overline{R'})_{\{\overline{B}\}}$. 
The only nontrivial graph automorphism of $\overline{B}$ is the one that swaps the two short arms $\overline{d_1},\overline{d_2}$ of the $D_8$ diagram. 
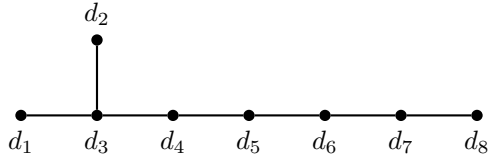
\begin{center}
 \begin{tikzpicture}[
vertex/.style = {circle, fill, inner sep=1.5pt, outer sep=0pt},
every edge quotes/.style = {auto=left, sloped, font=\scriptsize, inner sep=1pt}]
\node[vertex] (1) at (0,0)     [label=below: $d_1$] {};
\node[vertex] (2) at (1,1)     [label=above: $d_2$] {};
\node[vertex] (3) at (1,0)     [label=below: $d_3$] {};
\node[vertex] (4) at (2,0)     [label=below: $d_4$] {};
\node[vertex] (5) at (3,0)     [label=below: $d_5$] {};
\node[vertex] (6) at (4,0)     [label=below: $d_6$] {};
\node[vertex] (7) at (5,0)     [label=below: $d_7$] {};
\node[vertex] (8) at (6,0)     [label=below: $d_8$] {};
\Edges(2,3,4,5,6,7,8)
\Edges(1,3)
\end{tikzpicture}
\end{center}
Let $d_1^\vee, \dots, d_8^\vee \in R'\otimes \QQ$ of $B=\{d_1,\ldots,d_8\}$.
Then $R'$ is generated by $R$ and either $d_1^\vee$ or $d_2^\vee$ (but not both). In particular, only one of $\overline{d_1}^\vee$ or $\overline{d_2}^\vee$ lies in $\overline{R'}$. Therefore, the non-trivial graph isomorphism of $\overline{B}$ does not extend to an isometry of $\overline{R'}$.
\end{proof}
\begin{proposition}\label{prop:B-mod2-orbit}
    Let $B_1,B_2 \in \Delta^n(E_{10})$ for some $n$. 
    Let $R_i'$ be the primitive closure of $R_i:= \langle B_i \rangle$ in $E_{10}$ and suppose that $R_1 \not \cong A_9$.
    
    Then there exists $g \in O(E_{10})(2)$ with $B_1^g = B_2$ if and only if $R_1$ and $R_2$ have the same type, $\overline{B}_1=\overline{B}_2$ and $\overline{R_1'}=\overline{R_2'}$.
\end{proposition}
\begin{proof}
The necessity of the $3$ conditions is clear.

We turn to the sufficiency. By assumption $R_1$ and $R_2$ have the same type. Hence there is $g \in O(E_{10})$ with $B_1^g = B_2$.
Since $\overline{B}_1=\overline{B}_2$, and $\overline{R_1'}=\overline{R_2'}$ we have $\overline{g} \in O(E_{10}\otimes \FF_2)_{\{\overline{B_1}\},\{\overline{R_1'}\}}$.
If $R_1$ is primitive or of type $(A_7,E_7),(A_8,E_8)$, its diagram symmetry extends to an isometry of $E_{10}$, see \Cref{shimada:irreducible}. Surely it preserves $\overline{R_1'}$. Hence, after composing $g$ with this isometry, we may assume that $g$ restricts to the identity on $\overline{B}_1$ and preserves $\overline{R_1'}$.
Otherwise, $R_1\cong D_8$ is imprimitive of type $(D_8,E_8)$.
By \Cref{lem:D8E8_stabilizer}
$O(E_{10}\otimes \FF_2)_{\{\overline{B_1},\overline{R_1'}\}}=O(E_{10}\otimes \FF_2)_{\overline{B_1}, \{\overline{R_1'}\}}$.

In any case we may assume that 
\[g\in O(E_{10})_{\overline{R_1},\{\overline{R_1'}\}}.\]
By \Cref{lem:GY-transitive_on_ADE} (with the exception $A_9$) we find $h \in O(E_{10})$ acting trivially on $B_1$ and preserving $R'_1 \otimes \FF_2$ with $\overline{h}=\overline{g}^{-1}$. Then $hg \in O(E_{10})(2)$ and $B_1^{hg}=B_2$.
\end{proof}

\section{Maximal $ADE$ configurations.}\label{sec:maximal-ADE}
Now, we turn to classifying maximal connected configurations in $\R(Y,\sigma)$.

\begin{lemma}\label{lem:excludeA7E7}
Let $B \in \Delta^7(Y,\sigma)$ be of type $(A_7,E_7)$, then $\sigma$ 
is of type $(A_7, \ZZ/2\ZZ)$, $(D_8, \ZZ/2\ZZ)$, $(E_8, 0)$ or $(D_9, 0)$.
\end{lemma}
\begin{proof}
By \Cref{lem:existsRwithDeltaRequalSigma}, we find $R \subseteq S_Y$ with $\overline{\Delta}(R)=\sigma$ and set 
$S:=\langle r \in \Delta(R) \mid \overline{r} \in B\rangle$. Then $S$ is a root lattice of type $(A_7,E_7)$ contained in $R$. In particular, $E_7$ embeds into $R'$, which implies that $R'$ is one of $E_7,E_8$ or $D_9$.  
\end{proof}

\begin{lemma}\label{lem:maximal-characterization}
Let $\sigma$ be a connected component of $\DeltabarY$ and $B \in \R^i(Y,\sigma)$ connected. Let $R$ be the span of $B$. 
Then $B$ is maximal if and only if one of the following holds:
\begin{enumerate}
    \item $i = \rk \sigma$ ;
    \item $B$ is of type $(A_8,E_8)$ or $(D_8,E_8)$;
    \item $B$ is of type $(A_7,E_7)$ and $\overline{R'}\cap \overline{R'}^\perp\subseteq \sigma^\perp$. 
\end{enumerate}
\end{lemma}
\begin{proof}
(1) If $i = \rk \sigma$, then $B$ is clearly maximal.
Suppose that $i < \rk \sigma$ and $B$ is not of type $(A_7,E_7)$ $(A_8,E_8)$ or $(D_8,E_8)$. We show that $B$ is not maximal in $\Delta(Y)$.
To this end we apply \Cref{lem:existsRwithDeltaRequalSigma} to obtain a negative definite root sublattice $R_1 \subseteq S_Y$ with $\overline{\Delta}(R_1)=\sigma$. 
Let $\Delta^+(R_1) \subseteq \Delta^+(Y)$ be the positive root system of $R_1$ and recall from \Cref{roots_mod2_injective} that $\Delta^+(R_1) \to \overline{\Delta}(R_1)=\sigma$ is bijective.
Set $B_2 = \{x\in \Delta^+(R_1) \mid \overline{x} \in \overline{B} \} \subseteq R_1$. Then $B \to \overline{B} \to B_2$ is an isomorphism of graphs, hence its linear extension $R=\langle B \rangle \to R_2:
= \langle x \in \Delta(R_1) \mid \overline{x} \in \overline{B}\rangle$ is an isomorphism of lattices. Since $\overline{\Delta}(R_2)=\overline{\Delta}(R)$, $R$ and $R_2$ have the same type,
unless possibly if $R$ is of type $(A_8,A_8)$ and $R_2$ of type $(A_8,E_8)$. However, in this case $\sigma$ is of rank $9$ and therefore of type $(A_9, 0)$ or $(D_9, 0)$, and so $R_1$ is primitive in $S_Y$. Now, $E_8\cong R_2' \subseteq R_1$ is absurd because $R_1$ is irreducible.

By assumption $R=R'$ (\Cref{shimada:irreducible}) and so $\overline{R'}=\overline{R}=\overline{R_2}=\overline{R_2'}$.
By \Cref{prop:B-mod2-orbit} there is a $g \in O(S_Y)(2)$ with
\[\langle B \rangle  = R_2^g\subseteq R_1^g.\] 
Now, the elements of $B$ are indecomposable with respect to $\Delta^+(Y)\supseteq \Delta^+(R_1^g)$ and therefore part of the fundamental root system $F$ of $\Delta^+(R_1^g)$.
Since $B \subset F \in \Delta^{\rk \sigma}(Y)$ (and $i< \rk \sigma)$, $B$ is not maximal.

(2) Let $i< \rk \sigma$ and $B$ be of type $(A_8,E_8)$ or $(D_8,E_8)$, then $\sigma$ is of rank $9$ and therefore $\sigma$ of type $(A_9, 0)$ or $(D_9, 0)$. 
Suppose that $B_1$ extends $B$. Then $B_1$ is of rank $9$ and spans an $A_9$ or $D_9$ lattice, which must be primitive in $S_Y$ and therefore contain $R'$. However, $E_8$ is not a sublattice of $A_9$ and $D_9$, a contradiction. Thus, $B$ is maximal. 

(3) Let $i < \rk \sigma$ and $B$ be of type $(A_7,E_7)$. 
Suppose that $B_1$ extends $B$ and denote  $R_1=\langle B_1\rangle$. 
If $R \otimes \FF_2 \to R_1 \otimes \FF_2$ has a kernel, then $R' \subseteq R_1$ and therefore $\overline{\Delta}(R') \subseteq \overline{\Delta}(R_1)\subseteq \sigma$, which is absurd:
indeed, in this case $\Delta^+(R')$ would consist entirely of splitting roots. Let $F$ be a fundamental root system. Then some element of $B$ would be a sum of at least two elements of $F$, i.e. reducible.

Thus the map is injective. But
\[R \otimes \FF_2 \to R_1 \otimes \FF_2 \to E_{10}\otimes \FF_2\]
has a one dimensional kernel. This shows that the index of $R_1$ in $R_1'$ is even. The only possibility is that $R_1$ is of type $(D_8,E_8)$.
Set 
\[K=\overline{R'}\cap \overline{R'}^\perp\]
and note that 
$\overline{R'}=\overline{R}\oplus K$.
We have
\begin{equation}\label{eqn:max-char}
\overline{R_1}+K \subseteq \overline{R_1'}
\end{equation}
and the sum is direct (but possibly not orthogonal): otherwise, the line $K$ is contained in $\overline{R_1}$. Then $\overline{R'}=\overline{R}\oplus K \subseteq \overline{R_1}$ and both sides are of dimension $7$, i.e equal, but they are not isometric, because the left hand side is half-regular and the right hand side not. 
Since both sides of \cref{eqn:max-char} are of dimension $7+1=8$, they are equal.

Note that if $0 \neq \overline{R'}\cap \overline{R'}^\perp$ is contained in $\sigma^\perp$, then the sum in \cref{eqn:max-char} is orthogonal because $\overline{R}\subseteq \langle \sigma \rangle$. Hence
$\overline{R_1'}\cong \overline{E_8}$ is degenerate, which is a contradiction. Therefore, $\overline{R'}\cap\overline{R'}^\perp \not\subseteq \sigma^\perp$.

Let $B$ be of type $(A_7,E_7)$ and $K \not\subseteq \sigma^\perp$. We show that $B$ is not maximal.

We claim that there is an $x \in \langle \sigma \rangle\cap \overline{R}^\perp$
which is not orthogonal to $K$. 
Indeed, because $\overline{R}$ is non-degenerate, we can write
\[\langle \sigma \rangle=\overline{R} \oplus (\langle \sigma \rangle\cap \overline{R}^\perp).\]
We can assume moreover that $x$ is isotropic:
By \Cref{lem:excludeA7E7}, $\sigma$ is of type $(D_8, \ZZ/2\ZZ),(E_8, 0)$ or $(D_9, 0)$, hence 
$\langle \sigma \rangle$ is isometric to $U_2^3\oplus [0], U_2^4, U_2^4\oplus [0]$, see \Cref{table:root_mod2}. Since $\overline{R}\cong U_2^3$, $\langle \sigma \rangle\cap \overline{R}^\perp\cong [0],U_2,U_2\oplus [0]$.
Hence $\langle \sigma \rangle\cap \overline{R}^\perp$ has a basis consisting of isotropic vectors. Then $x$ can be chosen among this basis. 

Since $\overline{R'}\not\subseteq \langle \sigma \rangle$, $x\not\in \overline{R'}$. 
Let $k$ be the generator of $K$. The space 
\[V:=\langle x\rangle + \overline{R'}=\overline{R}\oplus \langle x, k\rangle\] 
is of dimension $8$ and non-degenerate
because $b(x,k)=1$.
$V$ is non-degenerate and in fact isomorphic to $E_8\otimes \FF_2$, i.e. has Witt index $4$ because $\overline{R}\oplus \langle x\rangle\cong 3U\oplus z$.

Hence there is $E_8 \cong R_1\leq S_Y$ with 
\[\overline{R_1} = V.\]
Note that $\overline{\Delta}(R_1)$ is the set of vectors $v \in V$ with $q(v)=1$. 
In particular, $\overline{\Delta}(R)\subseteq \overline{\Delta}(R_1)$.
Then we can define the root lattice
\[R_2=\{ r \in R_1 | \overline{r} \in \overline{R}\}\cong R,\]
which is of type $(A_7,E_7)$ because $\overline{R_2}=\overline{R}$ is of dimension $6$. 
We have
\[R_2' = \langle r \in \Delta(R_1) | \overline{r} \in \overline{R'} \rangle,\]
so that $\overline{R'}=\overline{R'_2}$ holds. As before we find $g \in O(S_Y)(2)$ with $R_2^g=R$. The rest of the argument is as in (1).
\end{proof}

\begin{proposition}\label{lem:Bmaximal_classification}
Let $\sigma$ be a connected component of $\overline{\Delta(Y)}$ and $B \in \R^i(Y,\sigma)$ be maximal.
\begin{enumerate}
\item If $\sigma$ is of type $(A_n, 0)$ $(n\leq 7)$, then $B$ is of type $(A_n,A_n)$. 
\item If $\sigma$ is of type $(D_n, 0)$ $(n\leq 8)$, then $B$ is of type $(D_n,D_n)$.
\item If $\sigma$ is of type $(E_6, 0)$, then $B$ is of type $(E_6,E_6)$. 
\item If $\sigma$ is of type $(A_7, \ZZ/2\ZZ)$, then $B$ is of type $(A_7,E_7)$.
\item If $\sigma$ is of type $(D_8, \ZZ/2\ZZ)$, then $B$ is of type $(A_7,E_7)$ or $(D_8,E_8)$.
\item If $\sigma$ is of type $(A_8, 0)$, then $B$ is of type $(A_8,A_8)$ or $(A_8,E_8)$.
\item If $\sigma$ is of type $(E_7, 0)$, then $B$ is of type $(A_7,E_7)$ or $(E_7,E_7)$.
\item If $\sigma$ is of type $(E_8, 0)$, then $B$ is of type $(A_8,E_8)$, $(D_8,E_8)$ or $(E_8,E_8)$ or $(A_8,A_8)$ or $(A_7,E_7)$.
\item If $\sigma$ is of type $(A_9, 0)$, then $B$ is of type $(A_9,A_9)$ or $(A_8,E_8)$.
\item If $\sigma$ is of type $(D_9, 0)$, then $B$ is of type $(D_9,D_9)$ or $(A_8,E_8)$.
\end{enumerate}
\end{proposition}
\begin{proof}
(1-4) Let $\rk \sigma = n\leq 8$. By \Cref{lem:maximal-characterization,lem:excludeA7E7}, $B$ is of rank $n$ as well. Moreover, $\Delta^+(B) \hookrightarrow \sigma$.
The only root subsystem of $A_n$ of rank $n$ is $A_n$ itself.
The same statement holds for $D_n$ and $E_6$. Thus $\Delta^+(B)$ maps bijectively to $\sigma$. This identifies the type $B$, unless $B=(A_8,E_8)$ but that is excluded by the assumption $l \leq 7$.
(5 - 8) follow directly from \Cref{lem:maximal-characterization,lem:excludeA7E7}.
(9) We can exclude $(D_8,E_8)$ because a $D_8$ root system is larger than an $A_9$ root system.
(10) Suppose that $B$ had type $(D_8,E_8)$ and let $R \leq S_Y$ with $\overline{\Delta}(R)=\sigma$. Then the preimage of $\langle \overline{B} \rangle $ in $R$ would be a sublattice of $R\cong D_9$ of type $(D_8,E_8)$, which is absurd. Suppose that $B$ had type $(A_7,E_7)$. By \Cref{lem:Bmaximal_classification}, $K:=\overline{R'}\cap \overline{R'}^\perp \subseteq \sigma^\perp$. But $\sigma^\perp$ is of dimension $1$ and totally isotropic with respect to the quadratic form, whereas $K$ is of dimension $1$ and not totally isotropic. A contradiction.  
\end{proof}

\begin{lemma}\label{lem:root_sublattices_orbits}
Let $Q$ be an irreducible root lattice of rank at most $9$. Let $C_1,C_2 \subseteq \Delta(Q)$ be irreducible ADE configurations. Let $Q_i$ be the span of $C_i$.
Assume that $Q_1 \cong Q_2$ and $Q_1' \cong Q_2'$
and that one of the following holds
\begin{enumerate}
\item $Q=Q_1$;
\item $Q_1 \neq Q_1'=Q$;
\item $A_7 \cong Q_1$, $Q_1'\cong E_7$, $Q\cong E_8$;
\item $A_8 \cong Q_1$, $Q\cong A_9$;
\item $A_8 \cong Q_1$, $Q\cong D_9$.
\end{enumerate}
Then there is an element of the Weyl group $W(Q)$ mapping $C_1$ to $C_2$.
\end{lemma}
\begin{proof}
Since the Weyl group of a root lattice acts simply transitively on its fundamental root systems, it suffices to find an element mapping $Q_1$ to $Q_2$. This solves (1) since $Q_1=Q=Q_2$.
For (2) we have $\rk Q_i=\rk Q=l$. Then the root system of $Q_i$ is $l$-maximal in the sense of \cite[Definition 2.1]{wallach:subsystem} and of characteristic $p=[Q_i':Q_i]$. 
By \cite[Theorem 3.1]{wallach:subsystem} we find an element of the Weyl group mapping $C_1$ to $C_2$. 

(3) The lattice $Q_i'$ is $l-1$ maximal in the sense of \cite[Definition 2.1]{wallach:subsystem}. By \cite[Theorem 3.1]{wallach:subsystem} we find once again an element of the Weyl group mapping $Q_1'$ to $Q_2'$. Now we can apply part (1) with $Q_1'$ in place of $Q$ to obtain an element of the Weyl group mapping $C_1$ to $C_2$.

(4-5) One checks that $A_8$ has a unique embedding into $Q$ up to $O(Q)$
Since $O(Q)=W(Q)\rtimes \langle -1 \rangle$, there is an element in $W(Q)$ mapping $Q_1$ to $Q_2$.
\end{proof}

\begin{lemma}\label{lem:max_helper}
    Let $\sigma \subseteq \DeltabarY$ be a connected component and $B_1,B_2 \in \Delta^i(Y,\sigma)$ be maximal of the same type.
 
    Then there is an element $\overline{w} \in W(\sigma)$ with $\overline{B}_1^{\overline{w}}=\overline{B}_2$. 
\end{lemma}
\begin{proof}
By \Cref{lem:existsRwithDeltaRequalSigma} we can find a root lattice $R \in \Delta^n(Y)$ with $\overline{\Delta}(R)=\sigma$.
Then 
\[W(R) \to W(\sigma)\]
is surjective. The natural map 
\[\phi:=\Delta(R)\to \overline{\Delta}(R)=\sigma\] 
is compatible with the action of $W(R)$ and $W(\sigma)$.
By \Cref{roots_mod2_injective} it restricts to a bijection on $\Delta^+(R)$ and is therefore a $2:1$ surjection.

Set 
\[C_i = \phi^{-1}(\overline{B}_i)\cap \Delta^+(R).\] 
It is an ADE configuration of the same type as to $B_i$.
By going through all cases of \Cref{lem:Bmaximal_classification} one may check that we are in one of the $5$ cases of \Cref{lem:root_sublattices_orbits}. For example if $B$ is of type $(D_8,E_8)$, then $\sigma$ is of type $(D_8, \ZZ/2\ZZ)$ or $(E_8, 0)$, $R\cong D_8,E_8$ and we are in case (1) and (2) of \Cref{lem:root_sublattices_orbits}.
Hence we obtain $w \in W(R)$ with $C_1^w=C_2$. Then $\overline{w}\in W(\sigma)$ maps $\overline{B}_1$ to $\overline{B}_2$ as desired. 
\end{proof}

\begin{proposition}\label{prop:same_AutY_orbit}
 Let $\sigma$ be a connected component of $\DeltabarY$ which is not of type $(A_9, 0)$.
 For $i=1,2$ let $B_i \in \R^n(Y,\sigma)$ be maximal of the same type. Denote by $R_i=\langle B_i \rangle$.

 Then there is an automorphism $h \in \aut(Y)$ with
 $B_1^h = B_2$ if and only if there is $\bar g \in \Gbar_Y$ with $\overline{R_1}=\overline{R_2}^g$ such that one of the following holds:
 \begin{enumerate}
  \item $[R_1':R_1]$ is odd;
  \item $B_1$ is of type $(A_7,E_7)$ and $\overline{R_1'}\cap \overline{R_1'}^\perp = \left(\overline{R_2'}\cap \overline{R_2'}^\perp\right)^{\bar g}$;
  \item $B_1$ is of type $(D_8,E_8)$ and $\overline{R_1'}+\langle \sigma \rangle = \left(\overline{R_2'}+\langle \sigma \rangle\right)^{\bar g}$.
 \end{enumerate}
\end{proposition}
\begin{proof}
The necessity follows from $\aut(Y)\subseteq G_Y$. We turn to the sufficiency of the conditions. 
Since $B_1$ and $B_2$ are maximal, \Cref{lem:transitive_on_induced_chambers} (3) applies and yields that it suffices to show that $B_1$ and $B_2$ lie in the same $G_Y$ orbit. We may therefore assume that $\overline g =1$.

By \Cref{lem:max_helper} there is a $w \in  W(Y) \subseteq G_Y$ such that $\overline{w} \in W(\sigma)$ and $\overline{B_1}^{\overline{w}}=\overline{B_2}$. 
This does not alter the conditions on $\overline{R_i'}$: since $W(\sigma)$ acts trivially on $S_Y\otimes \FF_2/\langle \sigma \rangle$, which covers case (3); 
and trivially on $\sigma^\perp$, which contains $\overline{R_i'}\cap \overline{R_i'}^\perp$ by \Cref{lem:maximal-characterization}. Thus, case (2) is covered.
We therefore assume that $\overline{B_1}=\overline{B_2}$.

We claim that in each case $\overline{R_1'}=\overline{R_2'}$ holds: In case (1) this is clear. 
In case (2), it follows from
\[U_2^3\oplus [0] \cong \overline{R_i'}=\overline{R_i} + \left(\overline{R_i'} \cap \overline{R_i'}^\perp\right).\]
In case (3) write $\overline{R_i}=S\oplus \langle e \rangle \subseteq \langle \sigma \rangle$ with $q(e)=0$, $S\cong U_2^{\oplus 3}$ and
$\overline{R_i'}=S +\langle e, f_i \rangle$ where $b(f_i,e)=1$ and $q(f_i)=0$. 
Then $f_1-f_2 \in \overline{R_i}^\perp \cap \langle \sigma \rangle$. By \Cref{lem:Bmaximal_classification} $\sigma$ is of type $(D_8, \ZZ/2\ZZ)$ or $(E_8, 0)$. Therefore  
$\overline{R_i}^\perp \cap \langle \sigma \rangle$ of dimension $1$. Since it contains the one dimensional subspace $\overline{R_i}\cap \overline{R_i}^\perp$, they are equal and so the claim $\overline{R_1'}=\overline{R_2'}$ follows.

By \Cref{prop:B-mod2-orbit} and the assumption that $\sigma$ is not of type $(A_9, 0)$, there is $h \in O(S_Y)(2)$ mapping $B_1^{w}$ to $B_2$.
If $h \in O(S_Y,\P_Y)(2) \subseteq G_Y$, then $wh \in G_Y$ maps $B_1$ to $B_2$.
If $h$ does not preserve $\P_Y$, then $-h$ preserves it and $w(-h) \in G_Y$ maps $B_1$ to $-B_2$. Since $-B_2$ is a fundamental root system of $\langle B_2 \rangle$, we find an element of the Weyl group $w_2 \in W(B_2)$ such that $w_2$ maps $-B_2$ to $B_2$.
Then $w(-h)w_2$ does the job.
\end{proof}

In the sequel, we will need to find all orbits of maximal configurations. Let us count them:
\begin{lemma}\label{lem:GYorbitA7E7}
 Let $\sigma$ be a connected component of $\DeltabarY$ of type $(A_7, \ZZ/2\ZZ)$ (respectively $(E_7, 0)$).
 Then there are at most $5$ (resp. $6$) $\Aut(Y)$-orbits of $(A_7,E_7)$ configurations 
 in $\R(Y,\sigma)$. 
\end{lemma}
\begin{proof}
Let $B \in \R^7(Y,\sigma)$ be an $(A_7,E_7)$ configuration and $R$ its span.
Since $W(\sigma)$ acts transitively on the $(A_7,E_7)$ subdiagrams of $\sigma$ (i.e. and $A_7$ subgraphs whose $\FF_2$-span is of dimension $6$), we may assume that $\overline{R}$ is fixed. 
Following in \Cref{prop:same_AutY_orbit} we count the number of possibilities for the generator $v$ of $\overline{R'}^\perp \cap \overline{R'}$. It satisfies $v \in \overline{R}^\perp$, $q(v)=1$ and must not be in $\langle \sigma \rangle$. There are $6$ (resp. $5$) such vectors.
\end{proof}

\begin{lemma}\label{lem:D8E8_max_configurations}
Let $\sigma$ be a connected component of $\DeltabarY$ of type $(D_8, \ZZ/2\ZZ)$.
There are at most $4$ $\Aut(Y)$-orbits of $(D_8,E_8)$ configurations $B\in \R^8(Y,\sigma)$
and at most $2$ orbits of maximal $(A_7,E_7)$ configurations. 
\end{lemma}
\begin{proof}
Let $V:=\langle \sigma \rangle\cong 3U_2\oplus [0]$. Then $V^\perp \cong U_2 \oplus [0]$,$V\cap V^\perp\cong [0]$ is totally isotropic of dimension $1$, hence 
\[V^\perp/\left(V\cap V^\perp\right)\cong U_2.\]

We start with the type $(D_8,E_8)$. There are $4$ regular subspaces $V \leq \overline{R} \leq E_{10} \otimes \FF_2$ of dimension $8$. Explicitly, they are of the form 
$V+\FF_2(v_0+v)$ with $v \in V^\perp$ where $v_0$ is such that $V+\FF_2 v_0$ is regular. \Cref{prop:same_AutY_orbit} yields the claim. 

Now we turn to $(A_7,E_7)$. According to \Cref{prop:same_AutY_orbit} the $\Aut(Y)$-orbit of $B$ is determined by 
$\overline{R'}\cap \overline{R'}^\perp = \FF_2 v$ where $v \in V^\perp$, by \Cref{lem:maximal-characterization}, and satisfies $q(v)=1$. 
The space $V^\perp/(V\cap V^\perp)\cong U_2$ contains two such $v$.
\end{proof}

\begin{lemma}\label{lem:E8E8_max_configurations}
Let $\sigma$ be a connected component of $\DeltabarY$ of type $(E_8, 0)$.
There are at most $3$ $\Aut(Y)$-orbits of $(D_8,E_8)$ configurations $B\in \R^8(Y,\sigma)$
and a unique orbit of maximal $(A_7,E_7)$ configurations. 
\end{lemma}
\begin{proof}
In this situation $\langle \sigma \rangle \cong U_2^4$, $S_Y \otimes \FF_2 = \sigma^\perp \oplus \langle \sigma \rangle$ and $\sigma^\perp$ is hyperbolic of dimension $2$. 
The count of maximal $(A_7,E_7)$ configurations follows with \Cref{lem:maximal-characterization} and \Cref{prop:same_AutY_orbit} and the fact that $\sigma^\perp$ contains a unique vector $v$ with $q(v)=1$.
For $(D_8,E_8)$ we may fix $\overline{R}$ and count regular subspaces $\overline{R} \subseteq \overline{R'} \subseteq S_Y\otimes \FF_2$ with $\overline{R'}\neq \langle \sigma \rangle$.
There are $3$, corresponding to the $3$ non-zero vectors in $\sigma^\perp$.
\end{proof}

\section{Orbits of rational curves}\label{sec:orbits-of-rational-curves}
In this section we complete the proof of \Cref{thm:main}.

\subsection{The case $A_n$}
\begin{lemma}\label{lem:odd_circle}
 Let $\sim$ be an equivalence relation on $\ZZ/n \ZZ$, such that for all $i$, $j \in \ZZ/n\ZZ$
 \[(i \sim j \mbox{ and } i+1\sim j+1) \mbox{ or }(i \sim j+1\mbox{ and } i+1 \sim j)\] 
 holds.
 If $n$ is odd, then $i\sim 0$ for all $i$.
\end{lemma}
\begin{proof}
Let $n$ be odd. It suffices to show $0 \sim 1$.
If $0 \not \sim 1$, then $0 \sim 2$. If $1 \sim 2$, we are done. Otherwise, $1 \sim 3$.
Suppose by induction that $0\sim 2i$ and $1 \sim 2i+1$. If $0\sim 2i+1$, we are done.
Otherwise, $0 \sim 2i+2$. Similarly, if $1 \sim 2i+2$, then $1 \sim 2i+2 \sim 0$. Else, $1\sim 2i+3$.
By induction $0 \sim 2i$ for all $i \in \ZZ/n\ZZ$. Since $n$ is odd, $1\equiv 2i \mod n$ for some $i$ and we are done.
\end{proof}

\begin{proposition}\label{lemma:An-transitive}
Let $\sigma \subseteq \DeltabarY$ be a connected component of type $(A_n,A_n)$ with $n \leq 7$.
Then any rational curve $\delta \in \R(Y,\sigma)$ is in the same $\aut(Y)$-orbit.
\end{proposition}
\begin{proof}
Let $\delta \in \R(Y)$ with $\overline{\delta} \in \sigma$. Then $\delta$ must be contained in some maximal connected $B \subseteq \R(Y)$. By \Cref{lem:Bmaximal_classification}, $B$ is of type $(A_n,A_n)$.

Since $n\leq 7$, one can check that $B^\perp$ is isotropic.
Let $f \in B^\perp$ be primitive and isotropic. Since $B$ is maximal, $W(\Delta(Y)\cap B^\perp)$ acts transitively on the $\Delta(Y)|_{B^\perp}$ chambers (\Cref{lem:transitive_on_induced_chambers}). Consequently, we may assume that $f$ is nef. Thus, $f$ induces an elliptic fibration and the elements $\delta_1,\dots,\delta_n$ of $B$ are part of a singular fiber.
This singular fiber consists of $n+1$ components: it has at least $\rk B+1 = n+1$ components; suppose that it has at least $n+2$ components, then their image in $\DeltabarY$ is a connected graph and hence contained in $\sigma$ which is of rank $n$. This is impossible.
Since the fiber contains an $A_n$ configuration, it is either of type $\widetilde{A}_n$, $\widetilde{E}_7$ or $\widetilde{E}_8$. The last two types are excluded by our assumption on $\sigma$.

Thus there is another component $\delta_{0}$ of the same fiber completing the diagram to an $\widetilde{A_n}$ configuration, i.e. a circle with $n+1$ nodes. This circle contains $n+1$ subgraphs of type $A_n$ obtained by deleting a single node.
These have the same rank $n$. By \Cref{lem:maximal-characterization}, they are therefore maximal. 

By \Cref{prop:same_AutY_orbit}, these $n+1$ subgraphs are in a single $\aut(Y)$-orbit.
This implies that their elements form a single $\aut(Y)$-orbit.
Indeed, if $n$ is odd, the central node of an $A_n$ configuration is distinguished and every vertex of the $\widetilde{A}_n$ circle is the central node of a unique $A_n$ configuration contained in it.
If $n$ is even, the two central nodes are distinguished. Again, any pair of adjacent nodes are the center of an $A_n$ diagram. Thus, every unordered adjacent pair $\{\delta_i,\delta_{i+1}\}$ can be mapped to any other unordered adjacent pair $\{\delta_j,\delta_{j+1}\}$ by an automorphism, where we see the indices as elements of $\ZZ/(n+1)\ZZ$. By \Cref{lem:odd_circle}, they form a single $\aut(Y)$-orbit.
\end{proof}

\begin{lemma}\label{lem:quadratic_eqn}
The set of integer solutions of $2x^2+6xy=-4$ is $\{\pm (-2,1), \pm (-1,1)\}$.
The set of integer solutions of $2x^2+6xy=-10$ is $\{\pm(-5,2),\pm(-1,2)\}$.
\end{lemma}
The proof is left to the reader.

\begin{lemma}\label{lem:A8_nef_isotropic}
Let $B \in \R^8(Y)$ be an $A_8$ configuration. Then there is a nef isotropic vector in $B^\perp$.
\end{lemma}
\begin{proof}
Note that $B=\{b_1,\dots,b_8\}$ (labeled in the usual way) is either of type $(A_8,E_8)$ or of type $(A_8,A_8)$. In the first case $B^\perp$ is a hyperbolic plane and in the second case $B^\perp$ is spanned by two vectors $a,b$ with $a^2=0$, $b^2=-2$ and $a.b=3$.
In either case $B^\perp$ contains some isotropic vector and $\Delta(B^\perp)$ consists of a single element up to sign.

Since $\langle B \rangle$ is negative definite of rank $8$ and is generated by classes of $(-2)$-curves, $\Pbar_{B^\perp}$ contains a $2$-dimensional face of the nef cone, it is given by 
\[\{x \in \Pbar_{B^\perp} \mid \forall r \in \Delta(Y)|_{B^\perp}\colon x.r \geq 0\}.\] 
Let $f_1$ $f_2$ be generators of the two isotropic rays of $\Pbar_{B^\perp}$. Assume that $\Delta(Y)|_{B^\perp}$ consists of at most a single element $r$. Then $r.f_1>0$ or $r.f_2>0$, hence at least one of them is the sought  isotropic nef class. 
It remains to prove the assumption that $\sharp \Delta(Y)|_{B^\perp} \leq 1$.
First suppose that $B$ is maximal. By \Cref{lem:transitive_on_induced_chambers},
\[\Delta(Y)|_{B^\perp}=\Delta(Y) \cap B^\perp \subseteq \Delta(B^\perp).\] 
The latter consists of a single element up to sign, and the claim follows in this case. 

As for the second case, suppose that $B$ is not maximal. Then there is $r \in \Delta(Y)$ such that $B_1 =\{r\} \cup B\in \Delta^9(Y)$. Thus it is of type $A_9$ or $D_9$ and the corresponding connected component $\sigma \subseteq \DeltabarY$ must be of the same type by \Cref{lem:Bmaximal_classification}.

Suppose that $\sigma$ is of type $(A_9, 0)$.
A direct computation with the $A_9$ root system shows that there is a unique $c \in \sigma$ such that $\{c,\bar b_1,\dots,\bar b_8\}$ is an $A_9$ configuration: Indeed, $\{\bar b_1,\dots,\bar b_8\}^\perp\subseteq \langle \sigma \rangle \cong R_\sigma \otimes \FF_2$ is of dimension one and spanned by $t = \bar b_1+\bar b_3+\bar b_5+\bar b_7+\bar b_9$. Thus, $c$ is either $\bar b_9$ or $\bar t+ \bar b_9$. But $q(t+\bar b_9)=0$. Thus $c=\bar b_9$.
This shows that $\overline{r}=c$ is unique. 

Let $b_1^\vee, \dots,b_8^\vee$ be the basis of $\langle B \rangle_{\QQ}$ dual to $b_1,\dots,b_8$.
Without loss of generality $r.b_8=1$ and $r.b_i=0$ for $i\neq 8$. Hence $r = b_8^\vee+y$ with $y \in (B^\perp)^\vee$. Since $(b_8^\vee)^2=-8/9$, we have $y^2=-10/9$. The lattice $(B^\perp)^\vee$ is spanned by two classes $a^\vee,b^\vee$ with $(a^\vee)^2=2/9, (b^\vee)^2=0$ and $a^\vee.b^\vee=1/3$.
By \Cref{lem:quadratic_eqn} there are (up to sign) $2$ elements $\widetilde{y} \in \left(B^\perp\right)^\vee$ with $\widetilde{y}^2=-10/9$. But they are distinct mod $2$. Since $\det A_8=9$ is not divisible by $2$,
\[S_Y \otimes \FF_2 = \left(\langle B \rangle \otimes \FF_2\right) \oplus \left(B^\perp \otimes \FF_2\right).\]
This means that exactly one $\widetilde{r}=b_8^\vee + \widetilde{y}$ is congruent to $c$ mod $2$, namely $r$. Moreover $\bar b\notin \sigma=\DeltabarY$.
Thus $\Delta(Y)|_{B^\perp}=\{\pm y\}$ as claimed.

For sigma of type $D_9$ we argue analogously and get $r = b_7^\vee+ y$ with $(b_7^\vee)^2=-14/9$ and $y^2=-4/9$.
Again, by \Cref{lem:quadratic_eqn} there are up to sign two solutions for $y \in B^\perp$ and they differ mod $2$.
\end{proof}

\begin{lemma}\label{lem:A8E8invol}
Let $B_0 \in \R^i(Y)$ be an $(A_8,E_8)$ configuration. 
Let $B_0= \{b_1,\dots, b_8\}$ be labeled such that $b_i.b_{i+1}=1$ for $i \in \{1,\dots,7\}$. Then
\begin{enumerate}
 \item there is an involution $g \in \aut(Y)$ with $g(b_j)=b_{9-j}$ for all $j \in \{1,\dots,8\}$ and $g|B_0^\perp = 1$;
  \item there exists an $\widetilde{A}_8$ configuration $B=\{b_0,\dots,b_8\} \subseteq \R(Y,\sigma)$
 \item $b_0,b_3,b_6$ are contained in the same $\Aut(Y)$-orbit;
 \item $b_1,b_2,b_4,b_5,b_7,b_8$ are contained the same $\Aut(Y)$-orbit.
 \item Any $\widetilde{A}_8$ fiber of an elliptic fibration on $Y$ contains an $(A_8,E_8)$ and an $(A_8,A_8)$ configuration. 
\end{enumerate}
\end{lemma}
\begin{proof}
(1) Set $R= \langle B \rangle \cong A_8$. Then the primitive closure $R' \cong E_{8}$ is in fact the unique even unimodular overlattice of $R$. The conditions in (1) define an isometry $\alpha$ of $R$ which extends to $R'$ by uniqueness. Then $g=\alpha \oplus \id_{B^\perp}$ preserves $S_Y=R'\oplus B^\perp$ and lies in $O^+(S_Y)$ because it fixes a vector of positive square in $B^\perp$.

We know that $O(A_8)=W(A_8)\ltimes \langle \alpha \rangle$ with $\alpha$ the diagram automorphism. 
We see that $W(A_8)$ is the kernel of $O(A_8)\to O(A_8^\vee/A_8)$ and that $\alpha$ acts as $-1$ on the discriminant group.  
Thus, $-\alpha \in W(A_8)$, i.e. $-\alpha \oplus \id_{B^\perp} \in W(R)\subseteq G_Y$. We have $g \in O^+(S_Y)$ 
and by the above $\overline{g}=\overline{-\alpha\oplus \id_{B^\perp}}$ lies in $\Gbar_Y$. Hence, by \Cref{remark:GbarY}, $g \in G_Y$. 
In particular, $g$ acts on the set of $\Delta(Y)$-chambers. 
Since $B_0\in \R^{8}(Y)$, $\P_{B^\perp}$ contains a $2$-dimensional face of the nef cone. Let $x$ be a general element of that face. 
Then $W(x)=W(B_0)$ acts transitively on the $\Delta(Y)$-chambers containing $x$. The nef cone is the one on which the elements of $B_0$ evaluate non-negatively. The same is true for $\Nef_Y^g$, because $g$ preserves $B_0$.
Hence $g$ preserves the nef cone. 

(2) By \Cref{lem:A8_nef_isotropic}, there is a primitive isotropic nef class $f \in B_0^\perp$. It induces an elliptic fibration and $B_0$ is contained in a reducible singular fiber. This fiber must be of type $\widetilde{A}_8$, $\widetilde{D}_8$ or $\widetilde{E}_8$. We can exclude $\widetilde{D}_8$ because it does not contain an $A_8$-configuration. An $\widetilde{E}_8$ fiber contains an $A_8$ configuration. But that configuration spans a primitive sublattice of $S_Y$. Hence, this is excluded as well. 

(3-4) Denote $B_i = \{b_j \mid j \neq i\}$. We claim that $B_i$ is of type $(A_8,E_8)$ for $i=0,3,6$ and of type $(A_8,A_8)$ else. Note that $b_0,\dots, b_8$ are linearly independent.
The vector 
\[v :=\tfrac{1}{3}(2 b_1+b_2+2b_4+b_5+2b_7+b_8)\]
lies in $S_Y$ because it lies in the unique even overlattice of $\langle B_0 \rangle  \cong A_8$.
We see that $v \in \langle B_i \rangle \otimes \QQ$ precisely for $i=0,3,6$. Since moreover $v \in S_Y$, $\langle B_i \rangle$ is not primitive in $S_Y$ for $i=0,3,6$ as claimed.

Using the involutions constructed in (1) with $B_0$ replaced by $B_0$, $B_3$ and $B_6$, we obtain that 
$\{b_i \mid i \in \{1,2,4,5,8\}\}$ is contained in a single orbit and that 
$\{b_i \mid i \in \{0,3,6\}\}$ is contained in a single orbit.

(5) Follows from the proof of (4) because $A_8$ has a unique primitive embedding into $E_{10}$.
\end{proof}

\begin{remark}
A more geometric approach to \Cref{lem:A8E8invol} is to start from a genus one fibration $f$ with a $I_9$ fiber, which must be simple. Its Jacobian $J(f)$ is an extremal rational elliptic surface with Mordell-Weil group $\ZZ/3\ZZ$. By \cite[Corollary 2.3]{martin2024enriquessurfaceszeroentropy} it acts as a rotation of order $3$ on the $I_9$ fiber and the inversion acts as a reflection by \cite[Lemma 2.4]{martin2024enriquessurfaceszeroentropy}. Together, they generate the same group as the involutions in \Cref{lem:A8E8invol} (1).
\end{remark}

A pair of isotropic vectors $f_1,f_2$ with $f_1.f_2=1$ is called a $U$-pair if $f_1$ is nef. A $U$-pair induces a generically finite $2:1$ map $Y \to D$ onto an anti-canonical del Pezzo surface $D$ of degree $4$ in $\PP^4$. The corresponding deck transformation is called a bielliptic involution. It preserves $f_1$ and $f_2$, and acts in a controlled way on the $(-2)$-curves in $U^\perp$. For example, any maximal $A_n$ configuration in $U^\perp$ is flipped by \cite{shimada:salem} (see also \cite[Lemma 8.7.20]{enriquesII}).  
\begin{lemma}\label{lem:nef_Upair}
Let $\sigma \subseteq \DeltabarY$ be a connected component and $f_1,f_2$ be a $U$-pair such that $\sigma \subseteq \{\overline{f}_1,\overline{f}_2\}^\perp$. Suppose that $f_1$ is nef.
Then $f_2$ is not nef if and only if $\overline{f_2-f_1} \in \DeltabarY$ if and only if $f_2-f_1 \in \R(Y)$. 
\end{lemma}
\begin{proof}
Since $f_1.f_2>0$ and $f_1$ is in the positive cone, so is $f_2$.
The set of vectors of $\Delta(Y)$ defining hyperplanes not containing $f_1$ and separating $f_1$ and $f_2$ projects onto $\Delta(Y)|_{\langle f_1,f_2\rangle}$. 
The vectors $f_1$, $f_2$ span a hyperbolic plane $U$ and $S_Y = U \oplus U^\perp$.
This implies that $\Delta(Y)|_U = \Delta(Y) \cap U$. Up to sign $U$ contains a single $-2$ vector. It is given by $r:=f_2-f_1$ and defines a face of the nef cone. Therefore, it is a $(-2)$-curve. 
If $\overline{r} \in \DeltabarY$, then $r$ is an effective splitting root and $f_2$ not nef. Otherwise, $\Delta(Y)|_U = \emptyset$ and $f_2$ is nef.
\end{proof}

\begin{proposition}\label{lemma:A8-transitive}
Let $\sigma \subseteq \DeltabarY$ be a connected component of type $(A_8, 0)$.
Then $\R(Y,\sigma)$ is contained in a single $\aut(Y)$-orbit.
\end{proposition}
\begin{proof}
Let $\delta \in \R(Y,\sigma)$, then we know that $\delta$ is contained in a maximal connected configuration $B_0$, which must be of type $(A_8,A_8)$ or $(A_8,E_8)$ by \Cref{lem:Bmaximal_classification} (2).

There exists a nef primitive isotropic vector $f_1$ in $B_0^\perp$ (\Cref{lem:A8_nef_isotropic}) and therefore an elliptic fibration with an $\widetilde{A}_8$ fiber with irreducible components $B:=\{b_0,\dots, b_8\}$.
By \Cref{lem:A8E8invol} (5), we can label them in such a way that $B_0:=\{b_1,\dots, b_8\}$ is of type $(A_8,E_8)$. By \Cref{lem:A8E8invol} (3 - 4), $B$ splits into at most two $\Aut(Y)$-orbits. One containing $b_2$ and the other one $b_6$.
Since $B_0$ is of type $(A_8,E_8)$, its orthogonal complement is a hyperbolic plane. Let $f_2$ be the unique isotropic vector orthogonal to $b_1,\dots, b_8$ and with $f_1.f_2=1$.\\

First suppose that $f_2$ is nef, i.e. the class of a half-fiber. By the same argument as in the proof of \Cref{lem:A8E8invol} (2), $B_0$ is contained in a fiber of type $\widetilde{A}_8$ of the elliptic fibration induced by $f_2$. 
This fiber is simple by \cite[Proposition 5.10]{brandhorst-gonzalez:527} because $f_2 \not \in \langle \sigma \rangle$
Let $b_{10} = 2f_2- b_1-\dots - b_8$ be the remaining irreducible component of the fiber.  
Since $f_1.f_2=1$, $b_{10}.b_0=2$. The situation is depicted in \Cref{figure:A8A8_1}.
\begin{figure}[H]
\begin{tikzpicture}
\tikzset{VertexStyle/.style= {fill=black, inner sep=1.5pt, shape=circle}}
\Vertex[NoLabel,x=1.532088886237956,y=1.2855752193730785]{3a}
\Vertex[NoLabel,x=0.34729635533386083,y=1.969615506024416]{4a}
\Vertex[NoLabel,x=-0.9999999999999996,y=1.7320508075688774]{5a}
\Vertex[NoLabel,x=-1.8793852415718166,y=0.6840402866513378]{6a}
\Vertex[NoLabel,x=-1.8793852415718169,y=-0.6840402866513373]{7a}
\Vertex[NoLabel,x=-1.0000000000000009,y=-1.732050807568877]{8a}
\Vertex[NoLabel,x=0.34729635533385994,y=-1.9696155060244163]{0a}
\Vertex[NoLabel,x=1.5320888862379556,y=-1.2855752193730792]{1a}
\Vertex[NoLabel,x=2.0,y=-4.898587196589413e-16]{2a}
\Vertex[NoLabel,x=0.2,y=-1.3]{10a}
\Edges(1a,2a,3a,4a,5a,6a,7a,8a,0a,1a)
\Edges(8a,10a,1a)
\path [-,bend left=15] (10a) edge node {} (0a);
\path [-,bend right=15] (10a) edge node {} (0a);
\tikzset{VertexStyle/.style= {inner sep=1.5pt, shape=circle}}
\Vertex[x=1.7619022191736493,y=1.47841150227904]{3}
\Vertex[x=0.39939080863393994,y=2.265057831928078]{4}
\Vertex[x=-1.1499999999999995,y=1.991858428704209]{5}
\Vertex[x=-2.161293027807589,y=0.7866463296490384]{6}
\Vertex[x=-2.1612930278075893,y=-0.7866463296490378]{7}
\Vertex[x=-1.150000000000001,y=-1.9918584287042083]{8}
\Vertex[x=0.3993908086339389,y=-2.2650578319280785]{0}
\Vertex[x=1.7619022191736489,y=-1.478411502279041]{1}
\Vertex[x=2.3,y=-5.633375276077824e-16]{2}
\Vertex[x=0.1,y=-1.0]{10}
\end{tikzpicture}
\caption{$(-2)$-curves visible so far.}\label{figure:A8A8_1}
\end{figure}
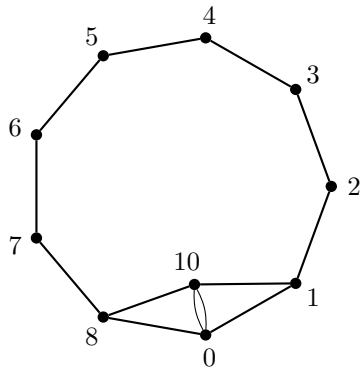
We see that $b_0,b_{10}$ is an $\widetilde{A}_1$ configuration of $(-2)$-curves. Therefore, $b_0+b_{10}$ defines another elliptic fibration with half-fiber 
\[f_3=\frac{1}{2}\left(b_0+b_{10}\right)= f_1+f_2 - b_1 - \dots - b_8.\]
Since $f_3.b_i=0$ for $i=2,\dots ,7$, these are components of a singular fiber containing an $A_6$ diagram. We calculate that $\sigma \cap \bar f_3^\perp$ maps to an $A_1\oplus A_6$ configuration in $\bar f_3^\perp/\FF_2 \bar f_3$. Then \cite{brandhorst-gonzalez:527}, implies that $f_3$ has (at least) a simple $\widetilde{A}_1$ fiber given by $b_0,b_{10}$ and a simple $\widetilde{A}_6$ fiber containing $b_2,\dots,b_7$. 
Let $b_{11} = 2f_3 - b_2 - \dots -b_7$ be the remaining component, see \Cref{figure:A8A8_2}.

\begin{figure}[H]
\begin{tikzpicture}
\tikzset{VertexStyle/.style= {fill=black, inner sep=1.5pt, shape=circle}}
\Vertex[NoLabel,x=1.532088886237956,y=1.2855752193730785]{3a}
\Vertex[NoLabel,x=0.34729635533386083,y=1.969615506024416]{4a}
\Vertex[NoLabel,x=-0.9999999999999996,y=1.7320508075688774]{5a}
\Vertex[NoLabel,x=-1.8793852415718166,y=0.6840402866513378]{6a}
\Vertex[NoLabel,x=-1.8793852415718169,y=-0.6840402866513373]{7a}
\Vertex[NoLabel,x=-1.0000000000000009,y=-1.732050807568877]{8a}
\Vertex[NoLabel,x=0.34729635533385994,y=-1.9696155060244163]{0a}
\Vertex[NoLabel,x=1.5320888862379556,y=-1.2855752193730792]{1a}
\Vertex[NoLabel,x=2.0,y=-0]{2a}
\Vertex[NoLabel,x=0.2,y=-1.3]{10a}
\Vertex[NoLabel,x=0.0554999,y=-0.342]{11a}
\Edges(1a,2a,3a,4a,5a,6a,7a,8a,0a,1a)
\Edges(8a,10a,1a)
\Edges(7a,11a,2a)
\path [-,bend left=15] (10a) edge node {} (0a);
\path [-,bend right=15] (10a) edge node {} (0a);
\path [dashed] (11a) edge node {?} (1a);
\path [dashed] (11a) edge node {?} (8a);
\tikzset{VertexStyle/.style= {inner sep=1.5pt, shape=circle}}
\Vertex[x=1.7619022191736493,y=1.47841150227904]{3}
\Vertex[x=0.39939080863393994,y=2.265057831928078]{4}
\Vertex[x=-1.1499999999999995,y=1.991858428704209]{5}
\Vertex[x=-2.161293027807589,y=0.7866463296490384]{6}
\Vertex[x=-2.1612930278075893,y=-0.7866463296490378]{7}
\Vertex[x=-1.150000000000001,y=-1.9918584287042083]{8}
\Vertex[x=0.3993908086339389,y=-2.2650578319280785]{0}
\Vertex[x=1.7619022191736489,y=-1.478411502279041]{1}
\Vertex[x=2.3,y=-5.633375276077824e-16]{2}
\Vertex[x=0.1,y=-1.0]{10}
\Vertex[x=-0.02,y=-0.0]{11}
\end{tikzpicture}
\caption{$(-2)$-curves visible so far.}\label{figure:A8A8_2}
\end{figure}
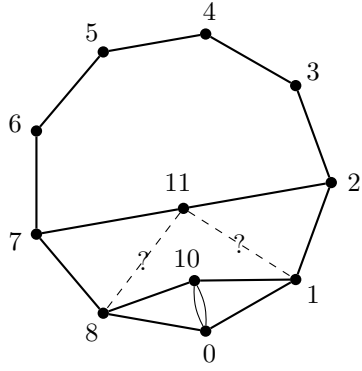
Since 
\begin{eqnarray}
    b_8.2f_3 &=& b_8.(b_0+b_{10}) = 2\\
    &=& b_8.(b_7+b_{11}) = 1 + b_8.b_{11},
\end{eqnarray}
we obtain $b_8.b_{11}=1$ and analogously $b_1.b_{11}=1$, see \Cref{figure:A8A8_3}.
\begin{figure}[H]
\begin{tikzpicture}
\tikzset{VertexStyle/.style= {fill=black, inner sep=1.5pt, shape=circle}}
\Vertex[NoLabel,x=1.532088886237956,y=1.2855752193730785]{3a}
\Vertex[NoLabel,x=0.34729635533386083,y=1.969615506024416]{4a}
\Vertex[NoLabel,x=-0.9999999999999996,y=1.7320508075688774]{5a}
\Vertex[NoLabel,x=-1.8793852415718166,y=0.6840402866513378]{6a}
\Vertex[NoLabel,x=-1.8793852415718169,y=-0.6840402866513373]{7a}
\Vertex[NoLabel,x=-1.0000000000000009,y=-1.732050807568877]{8a}
\Vertex[NoLabel,x=0.34729635533385994,y=-1.9696155060244163]{0a}
\Vertex[NoLabel,x=1.5320888862379556,y=-1.2855752193730792]{1a}
\Vertex[NoLabel,x=2.0,y=-0]{2a}
\Vertex[NoLabel,x=0.2,y=-1.3]{10a}
\Vertex[NoLabel,x=0.0554999,y=-0.342]{11a}
\Edges(1a,2a,3a,4a,5a,6a,7a,8a,0a,1a)
\Edges(8a,10a,1a)
\Edges(7a,11a,2a)
\path [-,bend left=15] (10a) edge node {} (0a);
\path [-,bend right=15] (10a) edge node {} (0a);
\path [-] (11a) edge node {} (1a);
\path [-] (11a) edge node {} (8a);
\tikzset{VertexStyle/.style= {inner sep=1.5pt, shape=circle}}
\Vertex[x=1.7619022191736493,y=1.47841150227904]{3}
\Vertex[x=0.39939080863393994,y=2.265057831928078]{4}
\Vertex[x=-1.1499999999999995,y=1.991858428704209]{5}
\Vertex[x=-2.161293027807589,y=0.7866463296490384]{6}
\Vertex[x=-2.1612930278075893,y=-0.7866463296490378]{7}
\Vertex[x=-1.150000000000001,y=-1.9918584287042083]{8}
\Vertex[x=0.3993908086339389,y=-2.2650578319280785]{0}
\Vertex[x=1.7619022191736489,y=-1.478411502279041]{1}
\Vertex[x=2.3,y=-5.633375276077824e-16]{2}
\Vertex[x=0.1,y=-1.0]{10}
\Vertex[x=-0.02,y=-0.0]{11}
\end{tikzpicture}
\caption{Final configuration.}\label{figure:A8A8_3}
\end{figure}
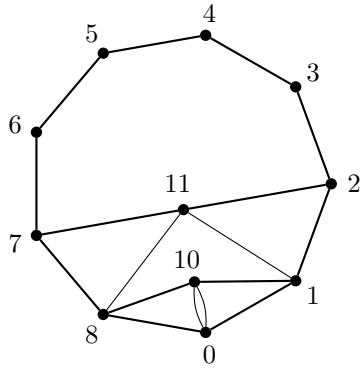
The divisor $f_4 = b_7+b_8+b_{11}$ is an $\widetilde{A}_2$ configuration of $(-2)$-curves. Moreover, it is primitive, since $f_4.b_6=1$. Hence, $f_4$ is a half fiber.
We calculate 
\[f_3.f_4 =(b_7+b_8+b_{11}).(b_0+b_{10})/2= 1.\] 
Then, following the proof of \cite[Lemma 2.6]{martin2024enriquessurfaceszeroentropy}, the linear system $|2 f_3 + 2 f_4|$ induces a separable, generically finite morphism $X \to Q\cong \mathbb{P}^1 \times \mathbb{P}^1$ of degree $2$. Let $g\in \Aut(Y)$ be the covering involution. It preserves each member of the two pencils $|2f_3|$ and $|2f_4|$. In particular, it preserves their common components, that is, $g(\{b_7,b_{11}\})=\{b_7,b_{11}\}$ and then $g(b_8)=b_8$. 
Since, by \cite[Lemma 2.4]{martin2024enriquessurfaceszeroentropy}, $g$ acts as a reflection on the half-fiber $b_7+b_8+b_{11}$, it exchanges $b_7$ and $b_{11}$. Since it also preserves the simple $I_7$ fiber of $f_3$, it exchanges $b_6$ and $b_2$. Thus $b_2$ and $b_6$ lie in the same $\Aut(Y)$-orbit, and the proof is finished in the case that $f_2$ is nef.\\

Now assume that $f_2$ is not nef. 
By \Cref{lem:nef_Upair} $b_{10}:=r=f_2-f_1$ is a $(-2)$ curve.
Since 
\[1=b_{10}.f_1= r.(b_0+\dots+b_8)=b_{10}.b_0/2,\]
$b_0+b_{10}$ is an isotropic primitive nef class, i.e. a half fiber. 
A calculation as before shows that this fibration has a double $\widetilde{A}_1$ fiber and a (simple) $\widetilde{A}_6$ fiber. The visible components of the $\widetilde{A}_6$ fiber are $b_2,\dots, b_7$. Therefore, there is another (until now invisible) component, $b_{11}$. As before $b_{11}.b_8 = b_{11}.b_1=1$ and we arrive at the configuration of smooth rational curves given in \Cref{figure:A8A8_3}.
\begin{figure}[H]
\begin{tikzpicture}
\tikzset{VertexStyle/.style= {fill=black, inner sep=1.5pt, shape=circle}}
\Vertex[NoLabel,x=1.532088886237956,y=1.2855752193730785]{3a}
\Vertex[NoLabel,x=0.34729635533386083,y=1.969615506024416]{4a}
\Vertex[NoLabel,x=-0.9999999999999996,y=1.7320508075688774]{5a}
\Vertex[NoLabel,x=-1.8793852415718166,y=0.6840402866513378]{6a}
\Vertex[NoLabel,x=-1.8793852415718169,y=-0.6840402866513373]{7a}
\Vertex[NoLabel,x=-1.0000000000000009,y=-1.732050807568877]{8a}
\Vertex[NoLabel,x=0.34729635533385994,y=-1.9696155060244163]{0a}
\Vertex[NoLabel,x=1.5320888862379556,y=-1.2855752193730792]{1a}
\Vertex[NoLabel,x=2.0,y=-0]{2a}
\Vertex[NoLabel,x=0.2,y=-1.3]{10a}
\Vertex[NoLabel,x=0.0554999,y=-0.342]{11a}
\Edges(1a,2a,3a,4a,5a,6a,7a,8a,0a,1a)
\Edges(7a,11a,2a)
\path [-,bend left=15] (10a) edge node {} (0a);
\path [-,bend right=15] (10a) edge node {} (0a);
\path [-] (11a) edge node {} (1a);
\path [-] (11a) edge node {} (8a);
\tikzset{VertexStyle/.style= {inner sep=1.5pt, shape=circle}}
\Vertex[x=1.7619022191736493,y=1.47841150227904]{3}
\Vertex[x=0.39939080863393994,y=2.265057831928078]{4}
\Vertex[x=-1.1499999999999995,y=1.991858428704209]{5}
\Vertex[x=-2.161293027807589,y=0.7866463296490384]{6}
\Vertex[x=-2.1612930278075893,y=-0.7866463296490378]{7}
\Vertex[x=-1.150000000000001,y=-1.9918584287042083]{8}
\Vertex[x=0.3993908086339389,y=-2.2650578319280785]{0}
\Vertex[x=1.7619022191736489,y=-1.478411502279041]{1}
\Vertex[x=2.3,y=-5.633375276077824e-16]{2}
\Vertex[x=0.1,y=-1.0]{10}
\Vertex[x=-0.02,y=-0.0]{11}
\end{tikzpicture}
\caption{Case $f_2$ not nef.}\label{figure:A8A8_4}
\end{figure}
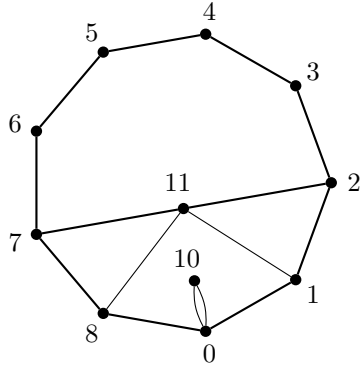
The same reasoning as before provides an involution exchanging $b_2$ and $b_6$.
\end{proof}

\begin{lemma}\label{A9:helper}
Let $\sigma$ be a connected component of $\DeltabarY$ and let $E \in \R^9(Y,\sigma)$ be an $A_9$ configuration. Then there exists an $E_{10}$ configuration $e_1,\dots, e_{10}$
with $E=\{e_1,\dots, e_9\}$ such that the following two vectors
\begin{eqnarray*}
f_1 &=& 4e_1   +8e_2   +12e_3   +16e_4   +20e_5   +24e_6   +28e_7   +17e_8   +6e_9   +15e_{10}  \\
f_2 &=& e_2 +2 e_3 + 3 e_4 +4 e_5+ 5e_6+ 6e_7+ 4 e_8 + 2e_9 +3 e_{10}
\end{eqnarray*}
are nef and isotropic.
In addition, $e_1,\dots, e_8 \in f_1^\perp$ and $e_2,\dots, e_9 \in f_2^\perp$.
\end{lemma}
\begin{proof}
Let $E=\{e_1,\dots, e_9\} \in \R^9(Y,\sigma)$ be an $A_9$ configuration. Since the $A_9$ lattice has a unique embedding into $E_{10}\cong S_Y$ up to isomorphism, we find $e_{10} \in S_Y$ such that $e_1,\dots, e_{10}$ form an $E_{10}$ configuration and $e_{10}.e_7=1$. 
By direct computation, one checks the following facts:
$f_1^2=f_2^2=0$, $f_1.f_2=4$, $f_1.e_i=0$ for $i=1,\dots,8$ and that $f_2.e_i=0$ for $i=2,\dots,9$, $f_i.e_j\geq 0$ for all $i \in \{1,2\}$ and $j \in \{1,\dots, 9\}$.

We first show that $f_1$ is nef.
From \Cref{lem:A8_nef_isotropic} we know that there is a primitive nef isotropic vector in $\langle e_1,\dots,e_8\rangle^\perp=:V$. Since $V$ is of rank $2$ and isotropic it contains exactly $4$ primitive isotropic vectors. Two of these are $\pm f_1$. Let us call $\pm f_1'$ the other two, where we may assume that $f_1.f_1'>0$. 
Recall that any two vectors in the closure of the positive cone have non-negative intersection. Either $f_1$ and $f_1'$  or $-f_1$ and $-f_1'$ are in the closure of the positive cone. The hyperplane $e_9^\perp$ cuts the positive cone in two components. We have $f_1.e_9>0$ and therefore $f_1'.e_9<0$. Thus of the $4$ isotropic vectors, either $f_1$ or $-f_1'$ is nef, which one it is depends on whether $f_1$ or $-f_1'$ is in the closure of the positive cone. If $f_1$ is in the positive cone, it is nef and we are done. 
So assume that $-f_1$ is in the positive cone. We let $\alpha \in O(S_Y)_{E}$ be the unique non-trivial element. Let $h$ be the generator of $E^\perp\cap S_Y$ in the positive cone. Since $\alpha$ preserves $E$, it satisfies $\alpha(e_i)=e_{10-i}$ for $i=1,\dots,9$ and therefore acts as $-1$ on the discriminant group of $\langle E \rangle \cong A_9$. Hence $\alpha$ maps $h$ to $-h$. Thus $\alpha$ exchanges $\P_Y$ and $-\P_Y$. Then $\alpha(f_1)$ is in the closure of the positive cone and by the same argument as above it is nef. Thus after relabeling $e_i$ to $e_{10-i}$ for $i=1,\dots,9$ and replacing $e_{10}$ by its image under $\alpha$, we may assume that $f_1$ is nef. 

Since $f_1.f_2>0$, $f_2$ is in the closure of the positive cone. As before $V_2:=\langle e_2,\dots,e_9\rangle^\perp$ contains $4$ isotropic vectors and exactly one of them is nef: the one in the closure of the positive cone with positive intersection with $e_1$ (taking the place of $e_9$), i.e. $f_2$.

\end{proof}

\begin{proposition}
    Let $\sigma \subseteq \DeltabarY$ be a connected component of type $(A_9, 0)$.
Then $\R(Y,\sigma)$ is contained in a single $\aut(Y)$-orbit.
\end{proposition}
\begin{proof}
By \Cref{lem:existsRwithDeltaRequalSigma} there is an $A_9$ configuration $B=\{b_1,\dots,b_9\} \in \R^9(Y,\sigma)$ labeled in the usual way.
It contains two $(A_8,A_8)$ configurations $b_1,\dots,b_8$ and $b_2,\dots,b_9$.
by \Cref{lem:A8_nef_isotropic} they give rise to two elliptic fibrations with $\widetilde{A}_8$ singular fibers $C =\{b_0,b_1,\dots,b_8\}$ respectively $D=\{b_2,\dots, b_9,b_{10}\}$ and half fibers $f_1$ respectively $f_2$.

Since $A_9$ has a unique embedding into $E_{10}$, we may complete $B$ to a basis $e_1,\dots, e_{10}$ forming an $E_{10}$-diagram where $e_i=b_i$ for $i=1,\dots, 9$ and $e_{10}.e_7=1$. Then $f_1$ and $f_2$ 
have coordinates as in \Cref{A9:helper}.

As before, set $C_i = C \setminus \{b_i\}$ and $D_j = D\setminus \{b_j\}$.
By \Cref{lem:A8E8invol} (5) each of $C$ and $D$ contain a subconfiguration of type $(A_8,E_8)$. Therefore, by \Cref{lem:A8E8invol} (3-4) each of $C$ and $D$ splits into (at most) two $\Aut(Y)$-equivalence classes containing $3$ and $6$ elements. 
With $2f_1 = \sum_{i=0}^8 b_i$
we calculate that 
\[b_0 = 2f_1-b_1\dots -b_8\equiv b_1 -b_3 + b_4 - b_6 + b_7 \mod 3S_Y\]
and
\[b_{10} = 2f_2-b_2\dots -b_{9}\equiv b_2 -b_4 + b_5 - b_7 + b_8 \mod 3S_Y.\]
Hence
\[-b_0  + b_1 -b_3 + b_4 - b_6 + b_7 + b_8 \in 3S_Y\]
is a $3$-divisible class in $S_Y$, which is contained in the span of $C_i$ for $i=2,5,8$. Therefore, $C_2$, $C_5$ and $C_8$ are of type $(A_8,E_8)$ and the other $C_i$ of type $(A_8,A_8)$. Thus, $b_2\sim b_5 \sim b_8$ and $b_0\sim b_1 \sim b_3 \sim b_4 \sim b_6\sim b_7$.
Likewise $D_i$ is of type $(A_8,E_8)$ for $i=3,6,9$, $b_3\sim b_6 \sim b_9$ and $b_{10} \sim b_2 \sim b_4 \sim b_5 \sim b_7 \sim b_8 $.
Thus $C \cup D$ is contained in a single $\Aut(Y)$-orbit. 

Let $b \in \R(Y,\sigma)$ be the class of a smooth rational curve. It is contained in a maximal connected configuration $B'$, which is of type $(A_9,A_9)$ or $(A_8,E_8)$ by \Cref{lem:Bmaximal_classification}.
By \Cref{prop:same_AutY_orbit} $B'$ can be mapped to $B$ or into $C$.
\end{proof}

\begin{proposition}
    Let $\sigma$ be a connected component of $\DeltabarY$ of type $(A_7, \ZZ/2\ZZ)$. 
    Then $\Aut(Y)$ acts transitively on $\R(Y,\sigma)$.
\end{proposition}
\begin{proof}
We choose generators $e_1,\dots, e_{10} \in S_Y$ forming an $E_{10}$ configuration as in \Cref{figure:E10}.
The vector
\[e_{11} = -\begin{pmatrix}
    0 & 0 & 0 & 1 & 2 & 3 & 4 & 3 & 2 & 2
\end{pmatrix},\]
is minus the highest root of the $E_7$ in the $E_{10}$ diagram. 
Then $a_1=e_4,\dots, a_6=e_9, a_7=e_{11}$ form an $(A_7,E_7)$ diagram because their $\QQ$-span contains $e_{10}$.

By \Cref{lem:BofTypeSigmaexists} and \Cref{shimada:irreducible}, we may assume (up to a change of basis) that the $a_i$ are in $\R(Y,\sigma)$ and that the usual isotropic vector of the $\widetilde{E}_8$ configuration
\[f_1= \begin{pmatrix}
0&1&2&3&4&5&6&4&2&3
\end{pmatrix}
\]
is nef (\Cref{lem:isotropic_in_Rperp}). Let 
\[a_8^1:= 2f_1 - a_1-\dots - a_7\] 
be the remaining component of the corresponding $I_8$ fiber.
Consider the primitive isotropic vector
\[f_2 = f_1+e_1+e_2\]
with $f_1.f_2=1$ and $f_2.a_i=0$ for $i=1,\dots 7$. 
By \Cref{lem:nef_Upair} either $f_2$ is nef or $f_2-f_1$ is a $(-2)$-curve.
   
Suppose that $f_2$ is nef. Set 
\[a_8^2 = 2f_2-a_1 - \dots - a_7 \in \R(Y,\sigma).\]
Since 
\[4 = 2f_1.2f_2 = a_8^1.2f_2 = a_8^1.(a_1+a_7+a_8^2)=2+a_8^1.a^8_2,\]
we have $a_8^1.a_8^2=2$.
Therefore, $2d_1:=a_8^1+a_8^2$ is an $\widetilde{A}_1$-fiber. The classes $a_2,\dots ,a_6$ are orthogonal to $d_1$, hence contained in a fiber. By assumption, $\dim \langle \sigma \rangle = 6$, hence $\dim \langle \sigma \rangle \cap d_1^\perp= 5$. We claim that the fiber is of type $I_6$. Otherwise, its span in $S_Y$ would contain a $A_6$, $D_6$ or $E_6$ root lattice $S$, which is primitive by \Cref{shimada:irreducible}. This is impossible because then $\dim S\otimes \FF_2 =6 > 5$. Let $b:=2d_1 - a_2-\dots -a_6 \in \R(Y,\sigma)$ be the remaining component.
Then $2 d_2 := a_1+a_2+b$ defines yet another elliptic fibration with an $I_3$ fiber and $f_1$ and $d_2$ form a $U$-pair. The corresponding bielliptic involution exchanges $a_1$ and $a_2$.

\begin{figure}[ht!]
\begin{tikzpicture}[
vertex/.style = {circle, fill, inner sep=1.5pt, outer sep=0pt},
every edge quotes/.style = {auto=left, sloped, font=\scriptsize, inner sep=1pt}]
 \node[vertex] (1) at (1.407,1.407)  [label=above: $a_7$] {};
 \node[vertex] (2) at ( 0, 2) [label=above: $a_8^1$] {};
 \node[vertex] (2a) at ( 0, 1.4) [label=below: $a_8^2$] {};
 \node[vertex] (3) at ( -1.407,  1.407) [label=above: $a_1$] {};
 \node[vertex] (4) at ( -2,  0) [label=left: $a_2$] {};
 \node[vertex] (c) at (0,  0) [label=below: $b$] {};
 \node[vertex] (5) at ( -1.407,  -1.407) [label=below: $a_3$] {};
 \node[vertex] (6) at ( 0,  -2) [label=below: $a_4$] {};
 \node[vertex] (7) at ( 1.407, -1.407) [label=below: $a_5$] {};
 \node[vertex] (8) at ( 2,  0) [label=right: $a_6$] {};
 \path[every node/.style={}]
  (2) edge[bend left=30] node [left] {} (2a);
  \path[every node/.style={}]
  (2) edge[bend right=30] node [left] {} (2a);
\Edges(1,2,3,4,5,6,7,8,1)
\Edges(1, 2a, 3)
\Edges(4,c,8)
\Edges(1,c,3)
\end{tikzpicture}
\begin{tikzpicture}[
vertex/.style = {circle, fill, inner sep=1.5pt, outer sep=0pt},
every edge quotes/.style = {auto=left, sloped, font=\scriptsize, inner sep=1pt}]
 \node[vertex] (1) at (1.407,1.407)  [label=above: $a_7$] {};
 \node[vertex] (2) at ( 0, 2) [label=above: $a_8^1$] {};
 \node[vertex] (2a) at ( 0, 1.4) [label=below: $\tilde{a}_8^2$] {};
 \node[vertex] (3) at ( -1.407,  1.407) [label=above: $a_1$] {};
 \node[vertex] (4) at ( -2,  0) [label=left: $a_2$] {};
 \node[vertex] (c) at (0,  0) [label=below: $\tilde b$] {};
 \node[vertex] (5) at ( -1.407,  -1.407) [label=below: $a_3$] {};
 \node[vertex] (6) at ( 0,  -2) [label=below: $a_4$] {};
 \node[vertex] (7) at ( 1.407, -1.407) [label=below: $a_5$] {};
 \node[vertex] (8) at ( 2,  0) [label=right: $a_6$] {};
 \path[every node/.style={}]
  (2) edge[bend left=30] node [left] {} (2a);
  \path[every node/.style={}]
  (2) edge[bend right=30] node [left] {} (2a);
\Edges(1,2,3,4,5,6,7,8,1)
\Edges(4,c,8)
\Edges(1,c,3)
\end{tikzpicture}
\caption{(-2)-curves depending on $f_2$ nef or not }
\end{figure}
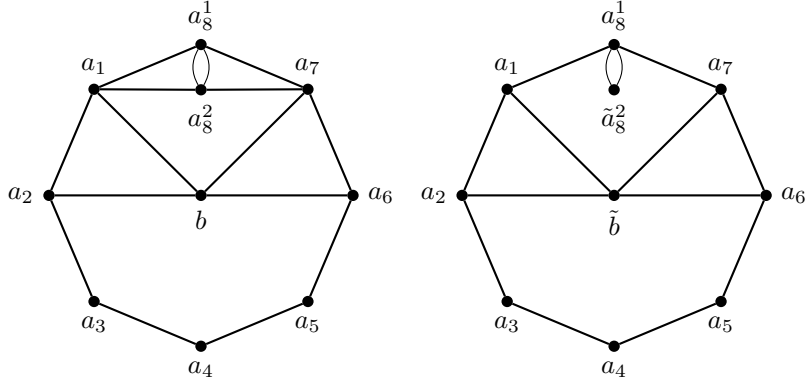

Suppose that $f_2$ is not nef. Set $\tilde{a}_8^2=f_2-f_1 = e_1+e_2$. Then $2\tilde{e}_2=a_8^1+\tilde{a}_8^2$ is still an $\widetilde{A}_1$-fiber and by the same argument one obtains a bielliptic involution exchanging $a_1$ and $a_2$. 

The $\widetilde{A_7}$ configuration \[B_1:= \{a_1,\dots, a_7,a_8^1\}\] contains $8$ configurations $B_{1,j}:=\{B_1\}\setminus \{a_i^1\}$ (with $a_j^1:=a_j$ for $j=1,\dots 7$) of type $(A_7,E_7)$. 
Set $R_{1,j}=\langle B_{1,j} \rangle$.
and let $v_{i,j}$ be the generator of $\overline{(R_{i,j})'}\cap \overline{R_{i,j}}^\perp$ for $i=1,2$.
We calculate that
\[v_{1,2j} = \bar e_{10}+ \bar e_4+\bar e_6\]
\[v_{1,2j+1} = v_{1,2}+\bar f_1=\bar e_1.\]
Because of the relation $2f_i = \sum_{k=1}^8 a_{i,k}$,  $\overline{R_{i,j}}$ is independent of $j$.
Therefore, by \Cref{prop:same_AutY_orbit} with $\overline{g}=\id$ the $B_{1,2j}$ lie in a single $\Aut(Y)$-orbit 
and so do the $B_{1,2j+1}$.
This implies that $a_1,a_3,a_5,a_7$ and $a_2,a_4,a_6,a_8^1$ each lie in a single $\Aut(Y)$-orbit. Since $a_1$ can be mapped to $a_2$, $B_1$ is contained in a single orbit. 
By the same reasoning the elements of \emph{any} $(A_7,E_7)$-configuration $B_i \in \R(Y,\sigma)$ are contained in a single $\Aut(Y)$-orbit. 

By \Cref{lem:GYorbitA7E7} there are at most $6$ orbits of maximal $(A_7,E_7)$-configurations. Let us find them. First suppose that $\DeltabarY=\sigma$. Then there are exactly $6$ orbits.

Consider the vectors $f_1,f_2,$
\begin{eqnarray}
f_3 &=& \begin{pmatrix}
1&  1&  2&  3&  4&  5&  6&  4&  2&  3
\end{pmatrix}\\
f_4 &=& \begin{pmatrix}
1&  3&  5&  6&  9&  12&  15&  10&  5&  8
\end{pmatrix}\\
f_5 &=& \begin{pmatrix}
1&  3&  7&  9&  13&  17&  21&  14&  7&  11
\end{pmatrix}.
\end{eqnarray}
We have $f_1, f_j \in (B_{1,1})^\perp$ and $f_1.f_j=1$ for $j=2,\dots, 5$. Since $\overline{f_j - f_1} \in \sigma^\perp$, it does not lie in $\sigma=\DeltabarY$, and cannot be a $(-2)$-curve. Therefore, \Cref{lem:nef_Upair} applies and yields that all the $f_j$ are nef.  
Each is the half fiber of an $I_8$ fiber and contains $8$ $(A_7,E_7)$ diagrams. Depending on the parity of $k$ (and the labeling) we have $v_{i,k} =\bar f_i+v_{1,2}$ or $v_{i,k}= v_{1,2}$ for $i=2,3$, and for $i=4,5$
$v_{i,k}= v_{1,1}+\bar f_i$ or $v_{i,k}=v_{1,1}$. This gives a total of $6$ different $v_{i,k}$ and we reach all orbits. 

If $\sigma \neq \DeltabarY$, then some $f_i$ may not be nef. By \Cref{lem:nef_Upair}, $r= f_i - f_1 \in \Delta(Y)$ and the corresponding reflection maps $f_1$ to $f_i$. Hence $f_1$ and $f_i$ are in the same $W(Y)$-orbit, and therefore the number of $(A_7,E_7)$-orbits drops by one as well. 

Let $b \in \R(Y,\sigma)$, then $b$ is contained in a maximal configuration $B_i$, which must be of type $(A_7,E_7)$ by \Cref{lem:Bmaximal_classification} (4) and thus in one of the orbits we have already found. 
Since $\Aut(Y)$ acts transitively on the nodes of each of the $6$ $A_7$ configurations and they all intersect, we can map $b$ to $a_1$.
\end{proof}

\subsection{The case $D_n$}
\begin{proposition}\label{lemma:Dn-transitive}
Let $\sigma \subseteq \DeltabarY$ be a connected component of type $(D_n, 0)$ with $n \leq 8$.
Then any rational curve $\delta \in \R(Y,\sigma)$ is in the same $\aut(Y)$-orbit.
\end{proposition}
\begin{proof}
Let $\delta \in \R(Y)$ with $\overline{\delta} \in \sigma$. Then $\delta$ must be contained in some maximal connected $B \subseteq \R(Y)$. By \Cref{lem:Bmaximal_classification}, $B$ is of type $(D_n,D_n)$.
Label $B = \{\delta_1,\dots, \delta_n\}$ such that $\delta_i.\delta_{i+1}=1$ for $2\leq i \leq n$ and $\delta_1.\delta_3=1$.

Since $n\leq 8$, we may (as in the $A_n$ case) assume that $B$ consists of components of a fiber of some elliptic fibration on $Y$. By maximality, the fiber has $n+1$ components and contains a $D_n$ configuration. This is only possible if the fiber is of type
$\widetilde{D}_n$ or of type $\widetilde{E_8}$. The second one is impossible because $\sigma$ is too small.
Thus, the fiber is of type $\widetilde{D}_n$. It contains four $D_n$ subdiagrams. Since they are of rank $n$, they are maximal and by \Cref{lem:Bmaximal_classification} of type $(D_n,D_n)$.
By \Cref{prop:same_AutY_orbit} they are in the same $\aut(Y)$-orbit in $\R^n(Y)$.

We claim that $\aut(Y)$ also acts transitively on the four $A_{n-1}$-subdiagrams of $\widetilde{D}_n$ as follows: Let $1,2,3,4$ label the end points of the $\widetilde{D}_n$ diagram with $1,2$ on the same side of it. Label the four $A_{n-1}$ diagrams by their endpoints. 
Thus we have $(1,3),(1,4),(2,3),(2,4)$. Let $\sim$ be the equivalence relation by being in the same orbit. We use transitivity on the four $D_n$ subdiagrams, also labeled by their endpoints.
Mapping $(1,3,4)$ to $(2,3,4)$, we see that (up to relabeling $3$ and $4$) we may assume that $(1,3) \sim (2,3)$ and $(1,4)\sim (2,4)$. 
Then mapping $(1,3,4)$ to $(1,2,4)$ we see that either $(1,3) \sim (1,4)$ or $(1,3)\sim (2,4) \sim (1,4)$. Thus $(1,3)\sim (1,4)$ in any case. Now mapping $(1,3,4)$ to $(2,3,4)$ yields the remaining equivalences: $(1,3) \sim (2,3)\sim (2,4)$ and proves the claim.

Any $A_{n-1}$ configuration in $\R^{n-1}(Y,\sigma)$ extends to a maximal configuration which must be of type $(D_n,D_n)$ and is mapped under $\aut(Y)$ to the original $D_{n}$. This shows that $\aut(Y)$ acts transitively on the set of all $A_{n-1}$ configurations in $\R^{n-1}(Y,\sigma)$.
Since $\sigma$ is of type $(D_n,D_n)$, $\overline{\delta}_1,\dots, \overline{\delta}_n \in S_Y\otimes \FF_2$ are linearly independent. Hence there exists an $\overline{f} \in S_Y\otimes \FF_2$ with $f.\delta_1=1$ and $f.\delta_i=0$ for $i \in 2,\dots, n$. By modifying $\overline{f}$ with an element in $\overline{B}^\perp$ we may assume $\overline{f}^2=0$.
By \cite{brandhorst-gonzalez:527} there exists an $\widetilde{A}_{n-1}$ configuration on $Y$. Now we can argue as in \Cref{lemma:An-transitive} to see that any two elements of an $A_{n-1}$ configuration lie in the same $\aut(Y)$-orbit.
Since any element of $D_n$ is part of an $A_{n-1}$ configuration we are done:
They all lie in the same $\aut(Y)$-orbit.
\end{proof}
The argument in the proof of \Cref{lemma:Dn-transitive} breaks for $(D_8,E_8)$ because 
$D_8$ has two even overlattices and thus \Cref{prop:same_AutY_orbit} applies in a different way.

\begin{lemma}\label{lem:AutOrbitD8E8}
Let $\sigma \subseteq \DeltabarY$ be a connected component of type $(D_8, \ZZ/2\ZZ)$.
Then there exists $B \in \R^8(Y,\sigma)$ of type $(D_8,E_8)$. Any rational curve $\delta \in \R(Y,\sigma)$ is in the same $\aut(Y)$-orbit as some $b_i \in B$. Further, $b_i, b_j \in B$ are in the same $\aut(Y)$ orbit if their distance (in the graph $B$) is even.
\end{lemma}
\begin{proof}
As in \Cref{lem:AutOrbitD8E8}, we work in the coordinates of $S_Y \cong \ZZ^{10}$ given by the $E_{10}$ diagram in \Cref{figure:E10}.
Set
\[d_1 = \begin{pmatrix}
0& 0& -2& -3& -4& -5& -6& -4& -2& -3
\end{pmatrix},\]
$d_2 = e_3,\dots, d_7=e_8$ and $d_8=e_{10}$.
Then $B_1 = \{d_1,\dots, d_8\}$ is a $D_8$ configuration and of type $(D_8,E_8)$ since $d_1+d_4+d_6+d_8 $ is divisible by $2$. 
Let $f_1 \in B_1^\perp$ be the isotropic vector
\[f_1=\begin{pmatrix}
   0&  1&  2&  3&  4&  5&  6&  4&  2&  3
  \end{pmatrix}.
\]
By \Cref{lem:isotropic_in_Rperp,shimada:irreducible,lem:existsRwithDeltaRequalSigma}, up to a change of coordinates, one can assume that $B_1 \subseteq \R(Y)$, that $f_1$ is nef and that $\overline{\Delta}(B_1)=\sigma$.

Let $h_{D_8}$ be the highest root of the positive root system $\Delta^+(B_1)$. Then
$d_0 := 2f_1 - h_{D_8} \in \R(Y)$ is a fiber component of the elliptic fibration of $Y$ with half fiber $f_1$. Set $\widetilde{B_1} = \{d_0,\dots, d_8\}$. It is a $\widetilde{D}_8$ configuration.

To get another $(-2)$ curve, we need another isotropic class
\[
f_2 :=
\begin{pmatrix}
1&  2&  4&  6&  8&  10&  12&  8&  5&  6
\end{pmatrix}.
\]
It satisfies $f_1.f_2=1$. By \Cref{lem:nef_Upair} either $f_2$ is nef or $r:=f_2-f_1$ is a $(-2)$ curve. However, $\overline{r} \not \in \sigma$ 
and $\overline{r} \not \in \sigma^\perp$. Since $\DeltabarY \subseteq \sigma \cup \sigma^\perp$, $\overline{r} \not\in\DeltabarY$. Hence $f_2$ is nef. 

The $A_7$ configuration $C_1=\{d_1,\dots, d_6,d_8\}$ is contained in $ f_2^\perp$. Define $a_0 := 2f_2 - h_2 \in \R(Y)$, where $h_2$ is the highest root of $C_1$. The set $\widetilde{C}_1=C_1 \cup \{a_0\}$ consists of the components of an $\widetilde{A}_7$-fiber of the elliptic fibration on $Y$ with half fiber $f_2$.

The following picture shows the dual graph of the curves exhibited so far.
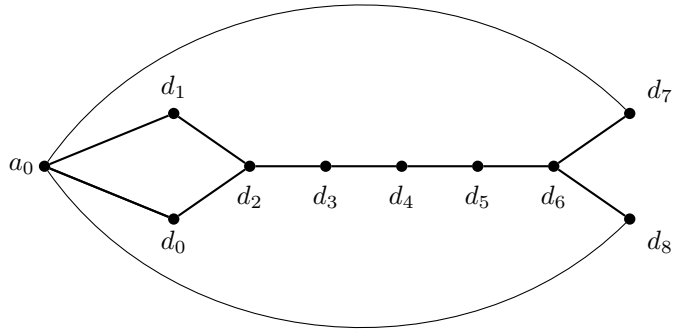
\begin{center}
 \begin{tikzpicture}
\tikzset{VertexStyle/.style= {fill=black, inner sep=1.5pt, shape=circle}}
\Vertex[NoLabel,x=-1.7,y=0]{e}
\Vertex[NoLabel,x=0,y=0.7]{-1}
\Vertex[NoLabel,x=0,y=-0.7]{0}
\Vertex[NoLabel,x=1,y=0]{1}
\Vertex[NoLabel,x=2,y=0]{2}
\Vertex[NoLabel,x=3,y=0]{3}
\Vertex[NoLabel,x=4,y=0]{4}
\Vertex[NoLabel,x=5,y=0]{5}
\Vertex[NoLabel,x=6,y=-0.7]{6}
\Vertex[NoLabel,x=6,y=0.7]{7}
\Edges(0,1,2,3,4,5,6)
\Edges(5,7)
\Edges(-1,1)
\Edges(e,0)
\Edges(e,-1)
\Edges(e,0)
  \path[every node/.style={}]
  (e) edge[bend left=50] node [left] {} (7);
    \path[every node/.style={}]
  (e) edge[bend right=50] node [left] {} (6);
\tikzset{VertexStyle/.style= {inner sep=1.5pt, shape=circle}}
\Vertex[x=-0,y=-1]{$d_0$}
\Vertex[x=-0,y=1.05]{$d_1$}
\Vertex[x=1,y=-0.4]{$d_2$}
\Vertex[x=2,y=-0.4]{$d_3$}
\Vertex[x=3,y=-0.4]{$d_4$}
\Vertex[x=4,y=-0.4]{$d_5$}
\Vertex[x=5,y=-0.4]{$d_6$}
\Vertex[x=6.4,y=1]{$d_7$}
\Vertex[x=6.4,y=-1]{$d_8$}
\Vertex[x=-2,y=0]{$a_0$}
\end{tikzpicture}
\end{center}
Let $L=\langle d_0,\dots,d_8,a_0\rangle$ be the sublattice of the curves visible so far.
It is of rank $10$ and determinant $4^3$. Thus $[S_Y:L]=2^3$ and $L\otimes \FF_2 \to S_Y\otimes \FF_2$ has a kernel of dimension $3$.
This gives the following $3$ independent relations
\begin{align}\label{eqn:D8_relation1}
\overline{d}_1+\overline{d}_3+\overline{d}_5+\overline{d}_8 &= 0 \\
\overline{d}_0+\overline{d}_1+\overline{d}_7+\overline{d}_8 &= 0 \label{eqn:D8_relation2}\\
\overline{a}_0+\overline{d}_2+\dots+\overline{d}_6+\overline{d}_8 &= 0 . \label{eqn:D8_relation3}
\end{align}
We see that $a_0,d_0,d_1,d_7,d_8$ is a $\widetilde{D_4}$ configuration of rational curves. Let 
\[f_3= (2a_0+d_0+d_1+d_7+d_8)/2\] be the corresponding half fiber. Then $d_3,d_4,d_5 \in f_3^\perp$ are components of another fiber, which, by inspection of $\overline{f_3}^\perp \cap \sigma$, must be of type $\widetilde{D_4}$. The missing two components are 
\[
r_1 := \begin{pmatrix}
2&  4&  8&  13&  18&  24&  30&  21&  12&  15
\end{pmatrix}
\]
and
\[
r_2 :=\begin{pmatrix}
2&  6&  12&  18&  24&  31&  38&  27&  16&  19\end{pmatrix}.
\]
The intersection numbers are as in \Cref{figure:D8E8}:
\begin{figure}[H]
 \begin{tikzpicture}[
vertex/.style = {circle, fill, inner sep=1.5pt, outer sep=0pt},
every edge quotes/.style = {auto=left, sloped, font=\scriptsize, inner sep=1pt}]
\node[vertex] (1) at (-2,2)     [label=above: $d_1$] {};
\node[vertex] (2) at (0,2)     [label=above: $d_2$] {};
\node[vertex] (3) at (2,2)     [label=above: $d_3$] {};
\node[vertex] (4) at (2,0)     [label=right: $d_4$] {};
\node[vertex] (5) at (2,-2)     [label=right: $d_5$] {};
\node[vertex] (6) at (0,-2)     [label=below: $d_6$] {};
\node[vertex] (7) at (-2,-2)     [label=left: $d_7$] {};
\node[vertex] (a0) at (-2,0)     [label=left: $a_0$] {};
\node[vertex] (0) at (-1,1)     [label=above: $d_0$] {};
\node[vertex] (r1) at (1,1)     [label=10: $r_1$] {};
\node[vertex] (r2) at (1,-1)     [label=above:$r_2$] {};
\node[vertex] (8) at (-1,-1)     [label=above: $d_8$] {};

\Edges(1,2,3,4,5,6,7,a0,1)
\Edges(a0,0,2)
\Edges(a0,8,6)
\Edges(2,r1,4)
\Edges(6,r2,4)
\end{tikzpicture}
\caption{}\label{figure:D8E8}
\end{figure}

They coincide with the dual graph of an Enriques surface with a numerically trivial automorphism as in \cite[8.2.23 (B)]{enriquesII}.

Define 
$B_1=\{d_1,\dots, d_8\}$,
$B_2 = (B_1\cup\{r_1\})\setminus \{d_3\}$, $B_3 = (B_1\cup\{r_2\})\setminus \{d_5\}$
and $B_4 = \{d_0,d_1,d_2,\dots,d_6,d_7\}$.
Set $R_i=\langle B_i \rangle$ for $i=1,\dots 4$. It is of type $(D_8,E_8)$. 
By a direct computation, one checks that the $\overline{R_i'}$ are pairwise distinct. 
By \Cref{lem:D8E8_max_configurations} every $(D_8,E_8)$ configuration of rational curves is in the same $\Aut(Y)$-orbit as one of $B_1,\dots, B_4$.

Define $\widetilde{C}_2=\{a_0,d_0,d_2,\dots,d_6, d_8\}$.
By \cref{eqn:D8_relation1,eqn:D8_relation2,eqn:D8_relation3}, second of its $8$ $A_7$ subdiagrams is of type $(A_7,E_7)$.
A direct computation reveals that the radical of their primitive closures mod $2$ fall into two classes. By \Cref{lem:D8E8_max_configurations} every $(A_7,E_7)$ configuration is in the same $\Aut(Y)$-orbit as these. 

Now, if $r \in \R(Y,\sigma)$, then $r$ is part of a maximal configuration, which is either of type $(A_7,E_7)$ or $(D_8,E_8)$
by \Cref{lem:Bmaximal_classification}. Therefore $r$ is in the same orbit as one of the nodes of the diagram.  

By \Cref{eqn:D8_relation1,eqn:D8_relation2,eqn:D8_relation3} every $A_7$ configuration in $\widetilde{C}_1$ is of type $(A_7,E_7)$. By a direct calculation and \Cref{lem:maximal-characterization} (3) every second one is maximal. Thus every other extends to a maximal configuration in $\R^8(Y,\sigma)$, which must be of type $(D_8,E_8)$ by \Cref{lem:Bmaximal_classification}. By \Cref{prop:same_AutY_orbit}, it 
maps to one of $B_1,B_2,B_3,B_4$. In particular, its center vertex is mapped to $d_4$. This shows that $d_2,d_4,d_6,a_0$ lie in a single $\aut(Y)$-orbit.
Similarly the maximal $(A_7,E_7)$ subdiagrams all lie in a single orbit and so do their central vertices $d_1,d_3, d_5$ and $d_8$. For each of $d_0,r_1,r_2,d_7$ one can write down a maximal $(A_7,E_7)$ diagram containing it as a central node and which is in the same orbit as the one we have seen before.
This leaves us with a total of at most two orbits as claimed.\\

The bielliptic involution associated to the $U$-pair $f_1,f_2$ is numerically trivial, hence in the center of $\Aut(Y)$. It fixes point-wise $4$ smooth rational curves \cite[Table 8.8,Theorem 8.8.21, Remark 8.8.22]{enriquesII}. (Looking at \Cref{figure:D8E8} these must be the vertices $a_0,d_2,d_4,d_6$ of valency greater than two.).
This shows that there are at least two orbits.
\end{proof}

\begin{proposition}\label{lemma:D9}
Let $\sigma \subseteq \DeltabarY$ be a connected component of type $(D_9, 0)$. 
Then $\DeltabarY=\sigma$, $Y$ has finite automorphism group, and $12$ rational curves their dual graph is as depicted below. They fall into two $\Aut(Y)$ orbits, distinguished by their valency $2$ or $3$. 
\end{proposition}
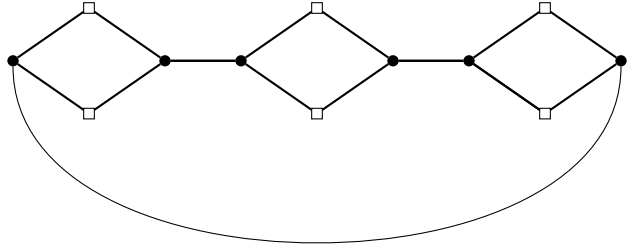
\begin{center}
 \begin{tikzpicture}
\tikzset{VertexStyle/.style={fill=black, inner sep=1.5pt, shape=circle}}
\Vertex[NoLabel,x=1,y=0]{b1}
\Vertex[NoLabel,x=3,y=0]{b3}
\Vertex[NoLabel,x=4,y=0]{b4}
\Vertex[NoLabel,x=6,y=0]{b6}
\Vertex[NoLabel,x=7,y=0]{b7}
\Vertex[NoLabel,x=9,y=0]{b10}
\tikzset{VertexStyle/.style= {style=draw, inner sep=2pt, shape=rectangle}}
\Vertex[NoLabel,x=2,y=0.7]{b2}
\Vertex[NoLabel,x=2,y=-0.7]{b2m}
\Vertex[NoLabel,x=5,y=0.7]{b5}
\Vertex[NoLabel,x=5,y=-0.7]{b5m}
\Vertex[NoLabel,x=8,y=0.7]{b8}
\Vertex[NoLabel,x=8,y=-0.7]{b9}
\Edges(b1,b2,b3,b4,b5,b6,b7,b8,b10)
\Edges(b1,b2m,b3,b4,b5m,b6,b7,b9,b10)
\Edges(b7,b9)
  \path[every node/.style={}]
  (b1) edge[bend right=90] node [left] {} (b10);
\end{tikzpicture}
\end{center}
\begin{proof}
Since $\langle \sigma \rangle$ is a quadratic space of dimension $9$ and $1$-dimensional radical, we have  $\sigma^\perp \subseteq \langle \sigma \rangle$. Therefore $\DeltabarY=\sigma$ and this is an Enriques surface of finite automorphism group of type \(\mathrm{II}\) \cite[Main Theorem]{kondo:enriques_finite_automorphism}. If it is generic, then it is a $(D_9,D_9)$-generic Enriques surface.
Following \cite[(7.4)]{brandhorst_shimada:tautaubar}, we compute that $\Nef_Y$ is an $L_{1,25}/S_Y(2)$-chamber 
of type $\mathrm{12B}$, it has $12$ walls, giving $12$ rational curves and they decompose under $\Aut(Y)$ into two orbits of size $6$. Indeed, it decomposes into at least $2$ orbits and by looking at the various subgraphs of type $\widetilde{A}_8$ and using \Cref{lem:A8E8invol}(3-4) we find enough automorphisms to get at most $2$ orbits. 
\end{proof}

\subsection{The case $E_n$}
\begin{proposition}\label{lemma:E6-transitive}
Let $\sigma \subseteq \DeltabarY$ be a connected component of type $(E_6,E_6)$.
Then any rational curve $\delta \in \R(Y,\sigma)$ is in the same $\aut(Y)$-orbit.
\end{proposition}
\begin{proof}
The argument parallels that of \Cref{lemma:An-transitive,lemma:Dn-transitive}. As before we find a maximal $B \in \R^6(Y,\sigma)$ and thus of type $(E_6,E_6)$ by \Cref{lem:Bmaximal_classification}.
Further there is an elliptic fibration contracting $B$. Therefore the corresponding fiber $\widetilde{B} \supset B$ is of type $\widetilde{E_6}$.
It contains $3$ subdiagrams of type $E_6$ which must be maximal by \Cref{lem:maximal-characterization} and of type $(E_6,E_6)$ as well \Cref{lem:Bmaximal_classification}.  By \Cref{prop:same_AutY_orbit} they are in the same $\aut(Y)$-orbit in $\R^6(Y)$.

Any $A_5$ configuration in $\R^5(Y,\sigma)$ extends to a maximal one, i.e. an $(E_6,E_6)$.
Hence by \Cref{prop:same_AutY_orbit}
they are all in the same $\aut(Y)$-orbit as the unique $A_5$-subdiagram of the $E_6$ configuration $B$.

Now, any node of $\widetilde{B}$ is contained in an $A_5$-configuration and these configurations overlap in the central node.
Thus, it remains to show that the nodes in any $A_5$-configuration form a single $\Aut(Y)$-orbit.
For this one uses another elliptic fibration with an $\widetilde{A}_5$ fiber and argues as in the ($A_n$, $n\leq 7$)-case in \Cref{lemma:An-transitive} using that all $A_5$-configurations are in the same orbit. 

For its existence note that $E_6$ contains $A_5 \oplus A_1$ as an index $2$ sublattice. Let $\bar e \in E_6\otimes \FF_2$ such that $\bar e^\perp$ is the corresponding hyperplane cutting out $A_5 \oplus A_1$. 
Define $\bar f:=\bar e+\bar g$ with $g \in \sigma^\perp$ and $q(e)=q(g)$. 
Then $\bar f$ is isotropic and defines the required fibration with an $\widetilde{A_5}$ fiber (cf. \cite{brandhorst_shimada:tautaubar}).
\end{proof}

\begin{proposition}\label{lem:AutOrbitE7E7}
Let $\sigma \subseteq \DeltabarY$ be a connected component of type $(E_7, 0)$.
Then there exists $B \in \R^7(Y,\sigma)$ of type $(E_7,E_7)$. Further, any rational curve $\delta \in \R(Y,\sigma)$ is in the same $\aut(Y)$-orbit as some $b_i \in B$ and $b_i, b_j \in B$ are in the same $\aut(Y)$ orbit if and only if their distance (in the graph $B$) is even.
\end{proposition}
\begin{proof}
Since there is an isotropic vector $\overline{f} \in \sigma^\perp$, $Y$ has an elliptic fibration with an $\widetilde{E}_7$ fiber. Let $\widetilde{B}_1$ denote the set of its irreducible components.
It contains an $E_7$ subdiagram, denoted $B_1$ and an $A_7$-subdiagram whose set of vertices is denoted $B_2$. Note that $\overline{\Delta}(B_1)=\sigma$.
By \Cref{lem:Bmaximal_classification} the types of $B_1$ and $B_2$ are $(E_7,E_7)$ and $(A_7,E_7)$.
Let $\{e_1,\dots, e_{10}\}$ be a fundamental root system of $S_Y \cong E_{10}$ labeled as in \Cref{figure:E10}.
By \Cref{lem:isotropic_in_Rperp,shimada:irreducible,lem:existsRwithDeltaRequalSigma}, up to a change of coordinates, we may assume that $B_1=\{e_4,\dots, e_{10}\}$
and that the $\widetilde{E}_7$ half fiber $f_1$ has coordinates
\[f_1 = \begin{pmatrix}
0&1&2&3&4&5&6&4&2&3
\end{pmatrix}.\]
Let $h_{E_7}$ be the highest root of the positive root system $\Delta^+(B_1)$. Define $a_7:= 2 f_1 + h_{E_7}$,
which is given in coordinates by
\[a_7 =\begin{pmatrix}
0& 2& 4& 5& 6& 7& 8& 5& 2& 4
\end{pmatrix}.\]
Then $\widetilde{B}_1 = B_1 \cup \{a_7\}$.
Let $f_2 \in S_Y$ be the isotropic vector given by
\[f_2 = \begin{pmatrix}
    1&3&5&6&7&8&9&6&3&4
\end{pmatrix}.\]
The vectors $(f_1,f_2)$ form a $U$-pair, and by \Cref{lem:nef_Upair} $f_2$ is nef or $r=f_2-f_1$ is a $(-2)$-curve. The latter is impossible because $\overline{r}\notin \sigma \cup \sigma^\perp \supseteq \DeltabarY$.

Since $\overline{e}_{10} \in \sigma$ and $f_2.e_{10}\equiv 1 \mod 2$, we have $\sigma \not\subseteq \overline{f}_2^\perp$ and therefore $f_2$ induces an elliptic fibration with an $\widetilde{A}_7$ reducible fiber given by $a_1:=e_4,\dots, a_6:=e_9, a_7$,   and
\begin{eqnarray}
a_0 =
\begin{pmatrix}
2& 4& 6& 6& 7& 8& 9& 6& 3& 4
\end{pmatrix}.\end{eqnarray}
It contains the $A_7$ subdiagram of $\widetilde{E}_7$ as $a_1,\dots, a_7$.
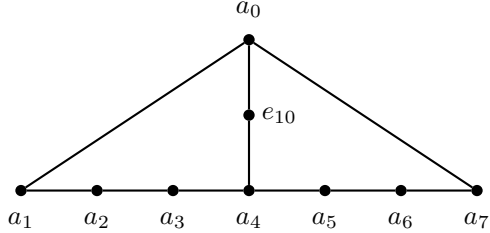
\begin{center}
 \begin{tikzpicture}
\tikzset{VertexStyle/.style= {fill=black, inner sep=1.5pt, shape=circle}}
\Vertex[NoLabel,x=3,y=0]{4a}
\Vertex[NoLabel,x=4,y=0]{5a}
\Vertex[NoLabel,x=5,y=0]{6a}
\Vertex[NoLabel,x=6,y=0]{7a}
\Vertex[NoLabel,x=7,y=0]{8a}
\Vertex[NoLabel,x=8,y=0]{9a}
\Vertex[NoLabel,x=9,y=0]{9b}
\Vertex[NoLabel,x=6,y=1]{10a}
\Vertex[NoLabel,x=6,y=2]{10b}
\Edges(4a,5a,6a,7a,8a,9a,9b)
\Edges(7a,10a,10b)
\Edges(4a,10b,9b)
\tikzset{VertexStyle/.style= {inner sep=1.5pt, shape=circle}}
\Vertex[x=3,y=-0.4]{$a_1$}
\Vertex[x=4,y=-0.4]{$a_2$}
\Vertex[x=5,y=-0.4]{$a_3$}
\Vertex[x=6,y=-0.4]{$a_4$}
\Vertex[x=7,y=-0.4]{$a_5$}
\Vertex[x=8,y=-0.4]{$a_6$}
\Vertex[x=9,y=-0.4]{$a_7$}
\Vertex[x=6.4,y=1]{$e_{10}$}
\Vertex[x=6,y=2.4]{$a_0$}
\end{tikzpicture}
\end{center}

Now, we obtain $8$ $A_7$ configurations in the $\widetilde{A}_7$ diagram. They are given by $C_2^{i}:=\{a_j \mid j \neq i\}$ for $i=0,\dots,7$.
They are maximal, hence of type $(A_7,E_7)$.
Set 
\[d_0 = \tfrac{1}{2}(a_0+a_2+a_4+a_6)\]
and 
\[d_1 = \tfrac{1}{2}(a_1+a_3+a_5+a_7).\]
Note that $2f_2 = a_0 + \dots+ a_7$ and hence $d_0 +d_1= f_2$.

The primitive closure of $\langle C_2^i\rangle $ is generated by $C_2^i$ and $d_{i+1 \mod 2}$.
By \Cref{prop:same_AutY_orbit}, the $C_2^{2i}$ with $i=0, 1,2,3$ form a single $\aut(Y)$-orbit, and so do the $C_2^{2i+1}$.
This shows that $e_4,e_6,e_8$ lie in a single orbit, and so do $e_5$,$e_7$, and $e_9$.
To deal with $e_{10}$, we use the fact that $B_3 = \{a_2,a_3,a_4,e_{10},a_0,a_7,a_5\}$ is another $E_7$-configuration. By \Cref{prop:same_AutY_orbit} there is an automorphism mapping $B_1$ to $B_3$ and therefore $e_{10}$ to $a_5$. 

Set $R=\langle B_2 \rangle = \langle a_1,\dots a_7 \rangle$ and note that 
\[\overline{R} = \langle \bar a_1,\dots, \bar a_6\rangle=\langle \bar e_4,\dots, \bar e_9\rangle. \]
If $r \in \R(Y,\sigma)$ is a rational curve, it extends to a maximal configuration of type $(A_7,E_7)$ or $(E_7,E_7)$. In the second case it can be mapped to $B_1$. In the first case, by 
\Cref{lem:GYorbitA7E7} (and a calculation in coordinates), to one of $5$ configurations, whose span $R_i$ is such that
$\overline{R_i'}=\overline{R}+\langle v \rangle $ with $v$ one of
\[\bar e_1, \bar e_2,\bar e_1+\bar e_2,\bar e_1+\bar e_3,\bar e_1+\bar e_3+\bar e_{10}.\]

Set $R_1:=\langle C_2^1\rangle $ and $R_2:=\langle C_2^2\rangle$ giving $\bar e_2$ and $\bar e_1+\bar e_3$. 
The isotropic vector 
\[f_3 = \begin{pmatrix} 1&  3&  7&  9&  11&  13&  15&  10&  5&  7\end{pmatrix} \in (C_2^0)^\perp\]
satisfies $f_1.f_3=1$ and is nef by the same reasoning as for $f_2$.
Let $a_0^3$ be the missing component of the $\widetilde{A}_7$ fiber of $f_3$. 
By the same reasoning as before $a_0^3$ can be mapped to $C^0_2$.
With $R_3=\langle a_2,\dots a_7, a_0^3\rangle$ we obtain $\overline{R_3'}=\overline{R}+\langle \bar e_1+\bar e_3+ \bar e_{10}\rangle$.

Consider the isotropic vectors and $f_4=f_2+e_1$ and $f_5=f_2+e_1+e_2$.
They satisfy $f_2.f_4=1=f_2.f_5$ and $f_4,f_5 \in (C_2^7)^\perp$. 
If $e_1\notin \R(Y)$, then $f_4$ is nef
and we obtain $R_4$ with $\overline{R_4'}=\overline{R}+\langle \bar e_1 +\bar e_2\rangle$.
If $e_1\in \R(Y)$, then $\bar e_2$ and $\bar e_1+\bar e_2$ are in the same $W(\sigma)$-orbit.
Similarly, from $f_5$ we obtain $\bar e_1$ (whether or not it is nef).
This shows that there are at most two $\Aut(Y)$-orbits of rational curves in $\R(Y,\sigma)$.\\

It remains to show that $\aut(Y)$ decomposes $\R(Y,\sigma)$ into at least two orbits.

We consider an $(E_7\oplus A_1,E_8)$-generic Enriques surface $Y'$. It has a numerically trivial involution by \cite[Theorem 1.18]{brandhorst_shimada:tautaubar} (the source assumes $\mathbb{C}$, but it works for $p\neq 2$ just the same. Alternatively, take the bielliptic involution of the $U$-pair in $(E_7\oplus A_1)^\perp$). This involution fixes $4$ disjoint smooth rational curves $C_1,\dots, C_4$ (and some isolated points) \cite[Table 8.8,Theorem 8.2.21, Remark 8.2.22]{enriquesII}. 
Since it is numerically trivial, it lies in the center of $\Aut(Y')$ and $\Aut(Y')$ permutes its fixed locus $C_1,\dots, C_4$. Since there is an $E_7$ configuration of $(-2)$ curves in $\R(Y',\sigma)$, there is a rational curve different from $C_1,\dots, C_4$. This shows that there are at least two distinct $\aut(Y')$-orbits on $\R(Y',\sigma_1')$ where $\sigma'_1$ is the $E_7$ component of $\DeltabarY$ and $\sigma'_2$ the $A_1$ component. 
Hence, there are exactly $3$ $\aut(Y')$-orbits on $\R(Y')$, one on $\R(Y',\sigma'_2)$ and two on $\R(Y',\sigma_1)$, which are distinguished by their images in $S_Y\otimes \FF_2$ (by whether $C.C_i \equiv 1 \mod 2$ for some $i$ or not). Let $c$ be the class of $C_1+C_2+C_3+C_4$. \\

Note that since $\sigma$ is of dimension $7$, it is the unique component of $\DeltabarY$ of this rank and therefore preserved by $G_Y$. The same is true for the non-trivial vector in $\sigma^\perp\cap\sigma$, which we call $\bar v$. Set $\sigma_2=\{v\}$. Besides $\sigma$ and possibly $\sigma_2$, $\Delta(Y)$ can have at most another connected component, which is of type $A_1$ or $A_2$. Therefore $G_Y$ acts trivially on the connected components and $\Gbar_Y=W(\DeltabarY)$.

We identify $S_Y$ and $S_{Y'}$ in such a way that $\sigma'_1=\sigma$. Then $\Nef_{Y'}$ is contained in the $\Delta(Y,\sigma)$-chamber $D_\sigma$ containing $\Nef_{Y}$. This chamber is tesselated by $\Delta(Y')$ chambers in a natural way since $\Delta(Y,\sigma)=\Delta(Y',\sigma_1') \subseteq \Delta(Y')$. 

Let $f \in \aut(Y)$ and let $B\subseteq \R(Y,\sigma^\perp)$ with $R:=\langle B \rangle$ negative definite and such that $\overline{\Delta}(R)=\sigma^\perp \cap \DeltabarY$. It exists because the components of $\DeltabarY$ besides $\sigma$ are rank at most $2$. Then there is $w \in W(B)$ with $\bar w \in W(\overline{\Delta}(Y)\setminus \sigma)$ such that $\bar f \in \Gbar_{Y'}=\overline{W(Y')}$.  Hence $f\circ w=f'\circ w'$ with $f' \in \aut(Y')$ and $w' \in W(Y')$.
Since $f\circ w$ preserves $B^\perp\cap \Nef_Y$, it preserves an interior point of $D_\sigma$ and therefore $D_\sigma$.  

In fact, we know that $w' \in W(\Delta(Y',\sigma_2'))$ because any reflection in an element of $\Delta(Y,\sigma)$ would leave $D_\sigma$. Since $w$ and $w'$ act trivially on $\sigma$, the actions of $f$ and $f'$ on $\langle \sigma \rangle \subseteq \DeltabarY$ agree. Hence $f$ preserves $\bar c$. This separates the two orbits.
\end{proof}

\begin{proposition}
    For $\sigma$ of type $(E_8,0)$ $\Aut(Y)$ acts with $4$ orbits on $\R(Y,\sigma)$.
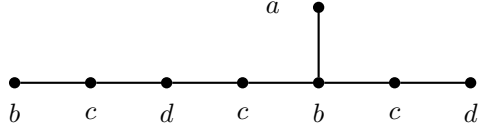
    \begin{center}
 \begin{tikzpicture}
\tikzset{VertexStyle/.style= {fill=black, inner sep=1.5pt, shape=circle}}
\Vertex[NoLabel,x=3,y=0]{4a}
\Vertex[NoLabel,x=4,y=0]{5a}
\Vertex[NoLabel,x=5,y=0]{6a}
\Vertex[NoLabel,x=6,y=0]{7a}
\Vertex[NoLabel,x=7,y=0]{8a}
\Vertex[NoLabel,x=8,y=0]{9a}
\Vertex[NoLabel,x=9,y=0]{9b}
\Vertex[NoLabel,x=7,y=1]{10a}
\Edges(4a,5a,6a,7a,8a,9a,9b)
\Edges(8a,10a)
\tikzset{VertexStyle/.style= {inner sep=1.5pt, shape=circle}}
\Vertex[x=3,y=-0.4]{$b$}
\Vertex[x=4,y=-0.4]{$c$}
\Vertex[x=5,y=-0.4]{$d$}
\Vertex[x=6,y=-0.4]{$c$}
\Vertex[x=7,y=-0.4]{$b$}
\Vertex[x=8,y=-0.4]{$c$}
\Vertex[x=9,y=-0.4]{$d$}
\Vertex[x=6.4,y=1]{$a$}
\end{tikzpicture}
\end{center}
\end{proposition}
\begin{proof}
Such an Enriques surface has a numerically trivial automorphism of type (a)
\cite[Theorem 8.2.21, Remark 8.2.22]{enriquesII}. Therefore it contains $10$ $(-2)$ curves whose dual graph is as below:

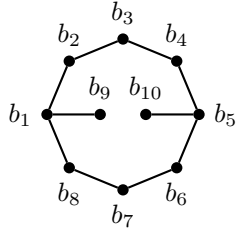
\begin{center}

\begin{tikzpicture}[
vertex/.style = {circle, fill, inner sep=1.5pt, outer sep=0pt},
every edge quotes/.style = {auto=left, sloped, font=\scriptsize, inner sep=1pt}]
 \node[vertex] (1) at (0.707,0.707)  [label=above: $b_4$] {};
 \node[vertex] (2) at ( 0, 1) [label=above: $b_3$] {};
 \node[vertex] (3) at ( -0.707,  0.707) [label=above: $b_2$] {};
 \node[vertex] (4) at ( -1,  0) [label=left: $b_1$] {};
 \node[vertex] (5) at ( -0.707,  -0.707) [label=below: $b_8$] {};
 \node[vertex] (6) at ( 0,  -1) [label=below: $b_7$] {};
 \node[vertex] (7) at ( 0.707, -0.707) [label=below: $b_6$] {};
 \node[vertex] (8) at ( 1,  0) [label=right: $b_5$] {};
 \node[vertex] (9) at ( -0.3,  0) [label=above: $b_9$] {};
 \node[vertex] (10) at ( 0.3,  0) [label=above: $b_{10}$] {};
\Edges(1,2,3,4,5,6,7,8,1)
\Edges(4,9)
\Edges(8,10)
\end{tikzpicture}
\end{center}
Its symmetry is $C_2 \times C_2$ and acts with $4$ orbits. In the diagram we see $4$ ways to form an $E_8$ configuration. This gives four automorphisms and hence the $C_2 \times C_2$ symmetry of the diagram is realized by the automorphism group action of $Y$ on it. 

Now, we have to ensure that every $(-2)$-curve $r \in \R(Y,\sigma)$ is equivalent to one in the diagram.
The curve $r$ is contained in a maximal configuration $B$, which is of type $(E_8,E_8),(A_8,E_8)$, $(D_8,E_8)$, $(A_8,A_8)$ or $(A_7,E_7)$ by \Cref{lem:Bmaximal_classification}.
By \Cref{lem:E8E8_max_configurations} and \Cref{prop:same_AutY_orbit}, these come in 
a single $\Aut(Y)$-orbit each, except for $(D_8,E_8)$ which comes in $3$ orbits. 
We check that each orbit meets a configuration $B_0$ in the diagram.

Let $b_1,\dots, b_{10}$ be the curves in the diagram and $L$ the sublattice they span. Its determinant is $16$ and therefore $[S_Y:L]=4$.
The diagram contains two $\widetilde{D}_8$ configurations
\[2f_1 = b_8+b_9+2(b_1+b_2+b_3+b_4+b_5)+b_6+b_{10},\]
\[2f_2 = b_2+b_9+2(b_1+b_8+b_7+b_6+b_5)+b_4+b_{10}\]
with half fibers $f_1,f_2 \in S_Y$, which together with $L$ span $S_Y$.

For $B$ of type $(A_7,E_7),(E_8,E_8)$ and $(A_8,A_8)$ take as $B_0$
one of 
\[\{b_2,\dots ,b_8\},\;\{b_1,\dots,b_7,b_{10}\},\;\{b_9,b_1,\dots,b_7\}.\]
For $B$ of type $(A_8,E_8)$ we can use \Cref{lem:A8E8invol} to obtain an $\widetilde{A}_8$ configuration, which contains an $(A_8,A_8)$ configuration such that each element of the $\widetilde{A}_8$ configuration is equivalent to an element of its $(A_8,A_8)$-subconfiguration.

For $B$ of type $(D_8,E_8)$ we have a total of $3$ possibilities for $\overline{R'}+\langle \sigma \rangle$ and on may check that each of them occurs in the diagram.

Finally, we need to show that there are at least $4$-orbits of rational curves. If $\Delta(Y)=\sigma$, then $Y$ is of zero entropy \cite{martin2024enriquessurfaceszeroentropy}. In particular, it admits a \emph{unique} elliptic fibration of positive rank. It is therefore preserved by $\Aut(Y)$. The class $b_1+\dots b_8$ is a fiber of this fibration and therefore preserved. Further, we see that $b_1$ and $b_5$ are fixed by the numerically trivial automorphism. Therefore, $\Aut(Y)$ acts with at least $3$ orbits on $b_1,\dots,b_8$, and $b_9$ must belong to a different orbit. 
If $\DeltabarY$ is larger than $\sigma$, then it is $\sigma \cup \tau$ with $\tau$ of type $(A_1,0)$. Such an Enriques surface has a finite automorphism group \cite[Theorem 8.9.3]{enriquesII} of type $I$ and the following dual graph of $(-2)$-curves \cite[Table 8.11]{enriquesII}
\begin{center}
\begin{tikzpicture}[
vertex/.style = {circle, fill, inner sep=1.5pt, outer sep=0pt},
every edge quotes/.style = {auto=left, sloped, font=\scriptsize, inner sep=1pt}]
 \node[vertex] (1) at (0.707,0.707)  [label=above: $b_4$] {};
 \node[vertex] (2) at ( 0, 1) [label=above: $b_3$] {};
 \node[vertex] (3) at ( -0.707,  0.707) [label=above: $b_2$] {};
 \node[vertex] (4) at ( -1,  0) [label=left: $b_1$] {};
 \node[vertex] (5) at ( -0.707,  -0.707) [label=below: $b_8$] {};
 \node[vertex] (6) at ( 0,  -1) [label=below: $b_7$] {};
 \node[vertex] (7) at ( 0.707, -0.707) [label=below: $b_6$] {};
 \node[vertex] (8) at ( 1,  0) [label=right: $b_5$] {};
 \node[vertex] (9) at ( -0.5,  0)  {};
 \node[vertex] (10) at ( 0.5,  0) [label=above: ] {};
 \node[vertex] (11) at ( -0.15,  0) [label=above: ] {};
 \node[vertex] (12) at ( 0.15,  0) [label=above: ] {};
\Edges(1,2,3,4,5,6,7,8,1)
\Edges(4,9)
\Edges(8,10)
\path [-,bend left=15] (10) edge node {} (12);
\path [-,bend right=15] (10) edge node {} (12);
\path [-,bend left=15] (11) edge node {} (12);
\path [-,bend right=15] (11) edge node {} (12);
\path [-,bend left=15] (9) edge node {} (11);
\path [-,bend right=15] (9) edge node {} (11);
\end{tikzpicture}
\end{center}
Clearly, there are at most $4$ orbits.
\end{proof}

\bibliographystyle{alpha}
\bibliography{literature}

\begin{thebibliography}{MMV24}

\bibitem[BGA24]{brandhorst-gonzalez:527}
Simon Brandhorst and Víctor González-Alonso.
\newblock 527 elliptic fibrations on {E}nriques surfaces.
\newblock \href{https://arxiv.org/abs/2408.00306}{arXiv:2408.00306}, 2024.

\bibitem[BS22]{brandhorst_shimada:tautaubar}
Simon Brandhorst and Ichiro Shimada.
\newblock Automorphism groups of certain {E}nriques surfaces.
\newblock {\em Found. Comput. Math.}, 22(5):1463--1512, 2022.

\bibitem[BV24]{brandhorst-veniani}
Simon Brandhorst and Davide~Cesare Veniani.
\newblock Hensel lifting algorithms for quadratic forms.
\newblock {\em Math. Comp.}, 93(348):1963--1991, 2024.

\bibitem[CD85]{cossec-dolgachev:automorphisms_nodal}
F.~Cossec and I.~Dolgachev.
\newblock {On automorphisms of nodal Enriques surfaces}.
\newblock {\em Bulletin (New Series) of the American Mathematical Society},
  12(2):247 -- 249, 1985.

\bibitem[CDL25]{cdl:enriquesI}
Fran{\c{c}}ois Cossec, Igor Dolgachev, and Christian Liedtke.
\newblock {\em Enriques surfaces {I}}.
\newblock Singapore: Springer, 2nd edition edition, 2025.

\bibitem[DK25]{enriquesII}
Igor Dolgachev and Shigeyuki Kond{\=o}.
\newblock {\em Enriques surfaces {II}}.
\newblock Singapore: Springer, 2025.

\bibitem[Dol84]{dolgachev:OnAutomorphismsOfEnriquesSurfaces}
I.~Dolgachev.
\newblock On automorphisms of {E}nriques surfaces.
\newblock {\em Invent. Math.}, 76(1):163--177, 1984.

\bibitem[Ebe13]{ebeling:lattices-and-codes}
Wolfgang Ebeling.
\newblock {\em Lattices and codes}.
\newblock Advanced Lectures in Mathematics. Springer Spektrum, Wiesbaden, third
  edition, 2013.
\newblock A course partially based on lectures by Friedrich Hirzebruch.

\bibitem[Kne02]{kneser}
Martin Kneser.
\newblock {\em Quadratische {F}ormen}.
\newblock Springer-Verlag, Berlin, 2002.
\newblock Revised and edited in collaboration with Rudolf Scharlau.

\bibitem[Kon86]{kondo:enriques_finite_automorphism}
Shigeyuki Kond{\=o}.
\newblock Enriques surfaces with finite automorphism groups.
\newblock {\em Jpn. J. Math., New Ser.}, 12:191--282, 1986.

\bibitem[MM09]{miranda_morrison:embeddings}
R.~Miranda and D.~R. Morrison.
\newblock Embeddings of integral quadratic forms.
\newblock \url{http://www.math.ucsb.edu/~drm/manuscripts/eiqf.pdf}, 2009.

\bibitem[MMV24]{martin2024enriquessurfaceszeroentropy}
Gebhard Martin, Giacomo Mezzedimi, and Davide~Cesare Veniani.
\newblock Enriques surfaces of zero entropy.
\newblock \href{https://arxiv.org/abs/2406.18407}{arXiv:2406.18407}, 2024.

\bibitem[Nam85]{namikawa:periods_of_enriques}
Yukihiko Namikawa.
\newblock Periods of {Enriques} surfaces.
\newblock {\em Math. Ann.}, 270:201--222, 1985.

\bibitem[Nik79]{nikulin}
V.~V. Nikulin.
\newblock Integer symmetric bilinear forms and some of their geometric
  applications.
\newblock {\em Izv. Akad. Nauk SSSR Ser. Mat.}, 43(1):111--177, 238, 1979.

\bibitem[Shi16]{shimada:salem}
Ichiro Shimada.
\newblock Automorphisms of supersingular {{\(K3\)}} surfaces and {Salem}
  polynomials.
\newblock {\em Exp. Math.}, 25(4):389--398, 2016.

\bibitem[Shi21]{shimada:ADE}
Ichiro Shimada.
\newblock Rational double points on {E}nriques surfaces.
\newblock {\em Sci. China Math.}, 64(4):665--690, 2021.

\bibitem[Wal68]{wallach:subsystem}
Nolan~R. Wallach.
\newblock On maximal subsystems of root systems.
\newblock {\em Canadian J. Math.}, 20:555--574, 1968.

\bibitem[Wan20]{wang}
Long Wang.
\newblock On automorphisms and the cone conjecture for {E}nriques surfaces in
  odd characteristic.
\newblock \href{https://arxiv.org/abs/1908.07928}{arXiv:1908.07928}, 2020.

\end{thebibliography}

\end{document}